\documentclass[final,1p,times]{elsarticle}
\usepackage{amscd} 
\usepackage{mathrsfs}
\usepackage[english]{babel}
\usepackage{amsfonts,amsmath,amsxtra,amsthm,amssymb,latexsym}
\usepackage{verbatim}
\biboptions{square,comma,numbers}

\newtheorem{theorem}{Theorem}[section]
\newtheorem{lemma}[theorem]{Lemma}
\newtheorem{corollary}[theorem]{Corollary} 
\theoremstyle{definition}

\newtheorem{remark}[theorem]{Remark} 
\numberwithin{equation}{section}

\def\z*{\bar z}

\def\B{\mathsf B}

\def\uno{\mathsf 1}
\def\H{\mathscr H}

\def\fh{\mathfrak h}
\def\fk{\mathfrak k}

\def\C{\mathcal C}

\def\Di{\mathscr D}
\def\D{\text{\rm dom}}
\def\dom{\text{\rm dom}}
\def\ran{\text{\rm ran}}
\def\supp{\text{\rm supp}}

\def\RE{\mathbb R}
\def\CO{{\mathbb C}}
\def\o{$\bar{\text{\rm o}}$}

\def\SL{{\rm SL}}
\def\DL{{\rm DL}}

\def\U{U}

\def\NA{\mathbb N}

\def\ph*{\phi_\star}

\def\t {\tilde}
\def\be{\begin{equation}}
\def\ee{\end{equation}}
\def\min{{\rm min}}
\def\max{{\rm max}}
\def\comp{{\rm comp}}

\journal{Journal of Differential Equations}
\begin{document}
\begin{frontmatter}

\author[man]{Andrea Mantile}
\ead{andrea.mantile@univ-reims.fr}
\address[man]{Laboratoire de Math\'{e}matiques, Universit\'{e} de Reims -
FR3399 CNRS, Moulin de la Housse BP 1039, 51687 Reims, France}

\author[pos]{Andrea Posilicano\corref{cor}}
\ead{andrea.posilicano@uninsubria.it}
\cortext[cor]{Corresponding author}
\address[pos]{DiSAT - Sezione di Matematica, Universit\`a dell'Insubria, Via Valleggio 11, I-22100 Como, Italy}

\author[sin]{Mourad Sini}
\ead{mourad.sini@oeaw.ac.at}
\address[sin]{RICAM, Austrian Academy of
Sciences, Altenbergerstr. 69, A-4040 Linz, Austria}

\title{Self-adjoint elliptic operators with boundary conditions on not closed hypersurfaces}

\begin{abstract}
\noindent
The theory of self-adjoint extensions of symmetric operators is used
to construct self-adjoint realizations of a second-order elliptic differential operator on $\mathbb{R}^{n}$ with
linear boundary conditions on (a relatively open part of) a compact hypersurface. Our approach
allows to obtain Kre\u\i n-like resolvent formulae where the reference operator
coincides with the ''free'' operator with domain $H^{2}(\mathbb{R}^{n})$; this provides  an useful tool for the scattering problem
from a hypersurface. Concrete examples of this construction are developed in
connection with the standard boundary conditions, Dirichlet, Neumann, Robin, $\delta$ and  $\delta^{\prime}$-type, assigned either on a $(n-1)$ dimensional compact boundary $\Gamma=\partial\Omega$ or on a relatively open part $\Sigma\subset\Gamma$.
Schatten-von Neumann estimates for the difference of the powers of resolvents of the free and the perturbed operators are also proven; these give existence and completeness of the wave operators of the 
associated scattering systems.
\end{abstract}


\end{frontmatter}


\begin{section}{Introduction.}

This work is concerned with the self-adjoint realizations of symmetric, second-order
elliptic operators 
\begin{equation}
Au (x) =\sum_{1\le i,j\le n}\partial_{x_{i}}(
a_{ij}(  x)  \partial_{x_{j}}u(  x))  -V(
x)  u(x)\,,\quad x\equiv(x_{1},\dots,x_{n})\in\mathbb{R}^{n} \,,\label{P_diff}%
\end{equation}
with boundary conditions on (relatively open parts of) hypersurfaces which are boundaries $\Gamma$ of bounded open sets $\Omega\subset\mathbb{R}^{n}$. We assume $\Omega$ to be of class $\C^{1,1}$; that suffices in case the boundary conditions are globally imposed on $\Gamma$ whereas,  in the case of boundary conditions on $\Sigma\subset\Gamma$, we require more regularity on $\Gamma$, even if it suffices to assume $\Sigma$ to be of class $\C^{1,0}$, i.e $\Sigma$ has a Lipschitz boundary. By \cite{GM}, \cite{BM}, we expect that our results can be extended to the case in which also $\Gamma$ is merely Lipschitz. As regards the conditions on the coefficients $a_{ij}$ and $V$, for simplicity we assume that they are both in $\C^{\infty}_{b}(\RE^{n})$, the standard regularity hypotheses allowing to use the classical results on mapping properties of surface potentials (as given, for example, in \cite[Chapter 6]{McLe}). However our regularity assumptions could possibly be relaxed. For example, in the case $a_{ij}=\delta_{ij}$, it should suffice to work with any $(-\Delta)$-bounded potential $V$; moreover,  by following the results concerning the surface potentials provided in \cite{Agra} and references therein, we expect that our analysis could also be adapted to the case where the $a_{ij}$'s are bounded and Lipschitz and $V$ is bounded. 

\par
When 
defined on the domain $\dom(A) =H^{2}(\mathbb{R}^{n}) $, the operator $A$ is self-adjoint and bounded from
above. 
We then consider the same differential operator $A$ but now acting on a domain characterized by linear boundary conditions on $\Gamma$ or on a relatively open part 
$\Sigma\subset\Gamma$. Using the abstract theory of self-adjoint extensions of
symmetric operators developed in \cite{P01}-\cite{P08}, we construct these
models as singular perturbations of the ''free operator'' with domain $H^{2}(\mathbb{R}^{n}) $. This allows
us to describe all possible linear boundary conditions within an unified
framework where the corresponding self-adjoint operators $A_{\Pi,\Theta}$ are
parametrized through couples $\left(  \Pi,\Theta\right)$, where $\Pi$ is an orthogonal projector on the Hilbert trace space $H^{\frac32}(\Gamma)\oplus H^{\frac12}(\Gamma)$ and $\Theta$ is a self-adjoint operator in the Hilbert space given by the range of $\Pi$. Our approach naturally yields to
Kre\u\i n-type formulae expressing the resolvent of the self-adjoint extension
$A_{\Pi,\Theta}$ in terms of the unperturbed resolvent $(  -A+z)
^{-1}$, plus a non-perturbative term; under suitable regularity assumptions of
the parameters $(  \Pi,\Theta)  $, the difference $(
-A_{\Pi,\Theta}+z)  ^{-k}-(  -A+z)  ^{-k}$ is of
trace class (for sufficiently large $k$) and the Birman-Kato criterion allows to consider $\{
A,A_{\Pi,\Theta}\}  $ as a scattering system provided with the corresponding wave operators.
\par
Singular perturbations supported on manifolds of lower dimension have been the
object of a large number of investigations (see for instance \cite{ABR}-\cite{AGS}, \cite{BEL}-\cite{BLL}, \cite{BSS}-\cite{BraTe}, \cite{DPP}-\cite{ExYo3}, \cite{FT}, \cite{FT1}, \cite{He}, \cite{KV}-\cite{Lot}, \cite{Pav}, \cite{P01} and references therein). These have
mainly concerned the case of $\delta$-perturbed Schr\"{o}dinger operators and
are generally motivated by the quantum dynamical modelling, as the case of
leaky quantum graphs, or the quantum interaction with charged
surfaces.
\par
Covering a wider class of models, the analysis developed in our work have been
inspired by the scattering problem from a compact hypersurface with
abstract boundary conditions. When these conditions are encoded by the
extension $A_{\Pi,\Theta}$, the scattered field $u_{\text{\rm sc}}$ corresponding to an incident
wave $u_{\text{\rm in}}$ is expected to be related to a limit absorption principle for  $\left(
-A_{\Pi,\Theta}+z\right) ^{-1}$.
In particular, the result obtained in the simpler case of point
scatterers (see \cite{HuMaSi}) suggests the relation%
\begin{equation*}
u_{\text{\rm sc}}=\lim_{\CO_{+}\ni z\rightarrow \lambda\in\RE}\left(  (  -A_{\Pi,\Theta}+z)  ^{-1}(  -A+z)  u_{\text{\rm in}}\right)  -u_{\text{\rm in}}\,,
\end{equation*}
where the limit is to be understood in an appropriate operator topology. 
This, using the Kre\u\i n resolvent identity for $(  -A_{\Pi,\Theta
}+z)  ^{-1}-(  -A+z)  ^{-1}  $, would lead to an
explicit characterization of the scattered field in terms of a factorized
formula depending on the incident wave. Different applications of this type of
formulas can be foreseen. In the most standard cases (Dirichlet, Neumann and
impedance boundary conditions on $\Gamma$ or $\Sigma\subset\Gamma$), they have
been exploited in the analysis of the corresponding inverse scattering problem
for surfaces reconstruction (see \cite{KG} for an introduction to the
factorization method). In this connection, our result could provide an
unified method to derive factorized formulas for the scattered field for a
large class of scattering problems with rather general linear boundary conditions.
\par
The first part of this work is devoted to the construction of self-adjoint
elliptic operators with abstract boundary conditions on $\Gamma$. In Section 2 we briefly recall the main results of the extension theory of
symmetric operators according to \cite{P01}-\cite{P08}, while, in Section 3 the mapping properties of the trace operators and of the single and
double layer operators, related to the surface $\Gamma$ and to the operator
$A$, are reviewed. In the Section 4, we introduce our model through the
symmetric operator $S$ given by the restriction of $A:H^{2}(\RE^{n})\subset L^{2}(\RE^{n})\to L^{2}(\RE^{n})$, to the dense linear set 
$\{u\in H^{2}(\RE^{n}): u|\Gamma=\partial_{\underline a}u|\Gamma=0\}$,
being $\partial_{\underline a}$
the co-normal derivative on $\Gamma$. The construction of the self-adjoint extensions of $S$, parametrized
through couples  $(  \Pi,\Theta)  $ on the trace space
$H^{\frac32}(  \Gamma)  \oplus H^{\frac12}(  \Gamma)  $, is then
given in Theorem \ref{T1} and Corollary \ref{C1}, where a Kre\u\i n-like resolvent formula
is also provided. The Schatten-von Neumann type estimates for the difference of the powers of resolvents, together with the spectral properties of these extensions and the
existence and completeness of the wave operators are then given in 
Theorems \ref{teo-schatten}, \ref{teo-schatten2} and Corollary \ref{spectrum}.
\par
The second part of the work is devoted to applications. In Section 5 we construct the
standard models, i.e. Dirichlet, Neumann, impedance (or Robin), $\delta$ and
$\delta^{\prime}$-type boundary conditions on $\Gamma$, in terms of extensions of $S$. The
main issue concerns the determination of the parameters $(  \Pi
,\Theta)  $ corresponding to the required constraints. In the case of
global conditions on $\Gamma$, this task is simplified by the nature of $\Pi$,
which, in the above mentioned cases, identifies with the projection onto the $H^{\frac32}(\Gamma)$ component for the Dirichlet case, with the projection onto the $H^{\frac12}(\Gamma)$ component for the Neumann case and with $\Pi=1$ in the Robin case;
then, the determination of $\Theta$ for the corresponding boundary conditions easily follows (almost) from algebraic arguments. The case of Dirichlet, Neumann, impedance, $\delta$ and
$\delta^{\prime}$-type conditions assigned only on a relatively open subset
$\Sigma\subset\Gamma$ is more complex and requires further work: in particular the analysis of
self-adjoint operators related to compressions of sesquilinear forms on the subspaces $H^{\frac32}(\Sigma)$ and $H^{\frac12}(\Sigma)$. This point is developed in Section 6 and in the Appendix. At least to our knowledge, the Kre\u\i n formulae we provide in the case of boundary conditions on not closed hypersurfaces $\Sigma\subset\Gamma$ do not appear in the past literature. 

\end{section}
\begin{section}{Preliminaries: Self-adjoint extensions of symmetric operators.}\label{Pre1}
Given the self-adjoint operator $$A:\D(A)\subseteq\H\to\H$$ in the Hilbert space $\H$ (equipped with the scalar product $\langle\cdot,\cdot\rangle_{\H}$), let $$\tau:\D(A)\to\fh$$ be continuous (w.r.t. the graph norm in $\D(A)$) and surjective onto the auxiliary Hilbert space $\fh$ (equipped with the scalar product $\langle\cdot,\cdot\rangle_{\fh}$). We further assume that  ker$(\tau)$ is dense in $\H$ and introduce the densely defined, closed, symmetric operator $$S:=A|\ker(\tau)\,.$$ Our aim is to provide all self-adjoint extensions of $S$; here we use the approach developed in \cite{P01}-\cite{P08} to which we refer for proofs and for the connections with other well known approaches to this problem (von Neumann's theory and boundary triples theory) .\par
For notational convenience we do not identify $\fh$ with its dual $\fh'$ and we denote by $J:\fh\to\fh'$ the duality mapping (a bijective isometry) given by the canonical isomorphism from $\fh$ onto $\fh'$, i.e. $J(\varphi)$ is the differential of the function $\varphi\mapsto \frac12\,\|\varphi\|^{2}_{\fh}$ (see e.g. \cite[Section 3.1]{Aub}); 
$\fh'$ inherits a Hilbert space structure by the scalar product $\langle\phi_{1},\phi_{2}\rangle_{\fh'}:=\langle J^{-1}\phi_{1},J^{-1}\phi_{2}\rangle_{\fh}$, so that $J$ becomes an unitary map, and we denote by $\langle\cdot,\cdot\rangle_{\fh'\fh}$ the $\fh'$-$\fh$ duality $$\langle\phi,\varphi\rangle_{\fh'\fh}:=\langle J^{-1}\phi,\varphi\rangle_{\fh}\,.
$$ The inverse $J^{-1}:\fh'\to\fh$ gives the duality mapping from $\fh'$ to its dual $\fh''\equiv\fh$; we denote by $\langle\cdot,\cdot\rangle_{\fh\fh'}$ the $\fh$-$\fh'$ duality defined by $\langle\varphi,\phi\rangle_{\fh\fh'}:=\langle J\varphi,\phi\rangle_{\fh'}=\langle \varphi,J^{-1}\phi\rangle_{\fh}$.
\par
Given a densely defined linear operator $$\Xi:\dom(\Xi)\subseteq\fh\to\fh\,$$ 
we denote by $\Xi'$ 
the dual operator $$\Xi':\dom(\Xi')\subseteq \fh'\to\fh'\,,\qquad \Xi\phi:=\phi'\,,$$ $$\dom(\Xi'):=\{\phi\in\fh':\exists\phi'\in\fh'\ \text{such that $\langle\phi',\varphi\rangle_{\fh'\fh}=\langle\phi,\Xi\varphi\rangle_{\fh'\fh}$ for all $\varphi\in\dom(\Xi)$}\}\,;$$
this is related to 
the (Hilbert) adjoint $\Xi^{*}:\dom(\Xi^{*})\subseteq \fh\to\fh$  by  $$\Xi^{*}=J^{-1}\Xi'J\,,\qquad 
\dom(\Xi^{*})=J^{-1}(\dom(\Xi'))\,.
$$ 
In the case $\Xi:\dom(\Xi)\subseteq\fh'\to\fh$, the dual operator  is defined in a similar way: 
$$
\Xi':\dom(\Xi')\subseteq \fh'\to\fh
\,,\qquad\Xi'\phi:=\varphi\,,
$$
$$\dom(\Xi'):=\{\phi\in\fh':\exists\varphi\in\fh\ \text{such that  $\langle\varphi,\psi\rangle_{\fh\fh'}=\langle\phi,\Xi\psi\rangle_{\fh'\fh}$ for all $\psi\in\dom(\Xi)$}\}\,,
$$
or, equivalently,  
$$
\Xi':\dom(\Xi')\subseteq \fh'\to\fh
\,,\qquad\Xi'\phi:=\varphi\,,
$$

$$\dom(\Xi'):=\{\phi\in\fh':\exists\varphi\in\fh\ \text{such that  $\langle\varphi,J^{-1}\psi\rangle_{\fh}=\langle J^{-1}\phi,\Xi\psi\rangle_{\fh}$ for all $\psi\in\dom(\Xi)$}\}\,.
$$
The latter definition shows that $\Xi=\Xi'$ if and only if $\t\Xi:=\Xi J:J^{-1}(\dom(\Xi))\subseteq \fh\to\fh$ is self-adjoint, i.e $\t\Xi^{*}=\t\Xi$; by a slight abuse of terminology we say that $\Xi$ is self-adjoint (resp. symmetric) whenever $\Xi=\Xi'$ (resp. $\Xi\subset \Xi'$). 
\par 
For any $z\in\rho(A)$ we define $R_{z}\in\B(\H,\dom(A))$ and $G_{z}\in\B(\fh',\H)$ by  
$$R_{z}:=(-A+z)^{-1}\,,\qquad G_{z}:\fh'\to\H\,, \quad G_{z}:=(\tau R_{\bar z})^{\prime}\,,
$$
i.e.
\be\label{gz}
\forall \phi\in\fh'\,,\forall u\in\H\,,\quad\langle G_{z}\phi,u\rangle_{\H}=\langle\phi,\tau(-A+\bar z)^{-1}u\rangle_{\fh'\fh}\,.
\ee
By our hypotheses on the map $\tau$ one gets (see \cite[Remark 2.9]{P01})
\be\label{zero}
\ran(G_{z})\cap \D(A)=\{0\}\,,
\ee  
and (see \cite[Lemma 2.1]{P01})
\be\label{f1}
(z-w)\,R_{w}G_{z}=G_{w}-G_{z}\,,
\ee
so that 
\be\label{ff}
\ran(G_{w}-G_{z})\subseteq \D(A)
\ee
and
\be\label{f2}
A(G_{z}-G_{w})=z\,G_{z}-w\,G_{w}\,.
\ee
Now, in order to simplify the exposition and since such an hypothesis holds true in the  
applications further considered, we suppose that $A$ has a spectral gap, i.e. 
$$
\rho(A)\cap \RE\not=\emptyset\,.
$$
Then we pose 
$$G :=G_{\lambda_{\circ}}\,,\quad \lambda_{\circ}\in \rho(A)\cap\RE$$ and define 
\be\label{weyl}
M_{z}:=\tau(G -G_{z})\equiv (z-\lambda_{\circ}) G ^{\prime}G_{z}:\fh'\to\fh\,.
\ee
Given an orthogonal projection $\Pi:\fh\to\fh$, the dual map $\Pi':\fh'\to\fh'$ is an orthogonal projection in $\fh'$ (since $\Pi'=J\Pi J^{-1}$ and $J$ is unitary) and, by \cite[Proposition 3.5.1]{Aub}, $\ran(\Pi)'=\fh'/\ran(\Pi')^{\perp}=\ran(\Pi')$ and $\ran(\Pi')'=\fh/\ran(\Pi)^{\perp}=\ran(\Pi)$. Thus, for a densely defined linear map $\Xi:\dom(\Xi)\subseteq\ran(\Pi')\to\ran(\Pi)$, one has $\Xi':\dom(\Xi')\subseteq\ran(\Pi')\to\ran(\Pi)$. As in the case $\ran (\Pi)=\fh$ we say that $\Xi$ is symmetric whenever $\Xi\subset \Xi'$ and that is self-adjoint 
whenever $\Xi=\Xi^{'}$; one has that $\Xi=\Xi'$ (resp. $\Xi\subseteq \Xi'$) if and only if $\t\Xi=\t\Xi^{*}$ (resp. $\t\Xi\subseteq\t\Xi^{*}$), where $\t\Xi:=\Xi J$, $\dom(\t\Xi):=J^{-1}(\dom(\Xi))$. Finally, given a self-adjoint $\Theta:\dom(\Theta)\subseteq\ran(\Pi')\to\ran(\Pi)$, we define 
$$
Z_{ \Pi,\Theta}:=\{z\in\rho(A): \Theta+\Pi M_{z}\Pi'\ \text{has a bounded inverse on $\ran(\Pi)$ to $\ran(\Pi')$} \}\,.
$$
\begin{theorem}\label{ext}  Any self-adjoint extension of $S=A|\ker(\tau)$ is of the kind $A_{\Pi,\Theta}$, where $\Pi:\fh\to\fh$ is an orthogonal projection, $\Theta:\dom(\Theta)\subseteq\ran(\Pi')\to\ran(\Pi)$ is a self-adjoint operator and
$$
\D(A_{\Pi,\Theta}):=\{u=u_{\circ}+G \phi\,,\ u_{\circ}\in\D(A)\,,\ \phi\in\dom(\Theta)\,,\ \Pi\tau u_{\circ}=\Theta\phi\}\,,
$$
$$  
A_{\Pi,\Theta}u:=Au_{\circ}+\lambda_{\circ}G \phi\,.
$$ 
Moreover $Z_{\Pi,\Theta}$ is not void, $\CO\backslash\RE\subseteq Z_{ \Pi,\Theta}\subseteq \rho(A_{\Pi,\Theta})$, and the resolvent of the self-adjoint extension $A_{\Pi,\Theta}$ is given by the Kre\u\i n's type formula
\be\label{res}
(-A_{\Pi,\Theta}+z)^{-1}=R_{z}+G_{z}\Pi'( \Theta+\Pi M_{z}\Pi')^{-1}\Pi \tau R_{z}\,,\quad z\in Z_{\Pi,\Theta}\,.
\ee
\end{theorem}
\begin{proof} Let us pose  $\t\Theta:=\Theta J:J^{-1}(\dom(\Theta))\subseteq \ran(\Pi)\to\ran(\Pi)$, $\t G_{z}:=G_{z}J:\fh\to\H$ and $\t M_{z}:=M_{z}J:\fh\to\fh$. Thus $\t\Theta$ is self-adjoint in $\ran(\Pi)$,
\begin{align*}
G_{z}\Pi'( \Theta+\Pi M_{z}\Pi')^{-1}\Pi \tau R_{z}
=&
G_{z}J\Pi J^{-1}((\Theta J+\Pi\t M_{z}\Pi)J^{-1})^{-1}\Pi \tau R_{z}\\
=&\t G_{z}\Pi(\t\Theta+\Pi \t M_{z}\Pi)^{-1}\Pi \tau G_{z}
\end{align*}
and $Z_{\Pi,\Theta}=\{z\in \rho(A):0\in\rho(\t\Theta+\Pi \t M_{z}\Pi)\}$. Therefore, 
by \cite[Theorem 2.1]{P08} (see Remark \ref{not} below), 
the linear operator 
$$\hat A_{\Pi,\Theta}:\dom(\t A_{\Pi,\Theta})\subseteq \H\to\H\,,\quad 
(-\hat A_{\Pi,\Theta}+z):=(-A+z)u_{z}\,,$$
$$
\dom(\hat A_{\Pi,\Theta})=\{u=u_z+G_{z}\Pi'(\Theta+\Pi M_{z}\Pi')^{-1}\Pi\tau u_{z},\ u_{z}\in\dom(A)\,,\ z\in Z_{\Pi,\Theta}\}
$$ 
is a $z$-independent self-adjoint extension of $A|\ker(\tau)$; moreover its resolvent is given by  \eqref{res}. Let us now show that $\hat A_{\Pi,\Theta}=A_{\Pi,\Theta}$. At first we pose 
$\phi_z:=(\Theta+\Pi M_{z}\Pi')^{-1}\Pi\tau u_z$, 
so that, since the definition of $\hat A_{\Pi,\Theta}$ is $z$-independent, $u\in\D(\hat A_{\Theta})$ if and only if for any $z\in Z_{\Pi,\Theta}$ there exists $u_z\in\D(A_{\circ})$,  $\Pi \tau u_z=(\Theta+\Pi M_{z}\Pi)\phi_{z}$, such that $u=u_z+G_z\phi_{z}$.
Therefore, by \eqref{f1}, 
$$
u_{z}-u_{w}=G_w\phi_{w}-G_z\phi_{z}=
G_{z}(\phi_{w}-\phi_{z})+(z-w)R_zG_w\phi_{w}\,.
$$
By \eqref{zero}, one obtains $G_{z}(\phi_{w}-\phi_{z})=0$. Since $G_{z}$ is injective (it is the adjoint of a surjective map), this gives $\phi_{z}=\phi_{w}$, i.e. the definition of $\phi_{z}$ is $z$-independent. 
Thus, posing $u_{\circ}:=u_z+(G_z-G )\phi$, one has $u=u_{\circ}+G \phi$, with $u_{\circ}\in\dom(A)$ and
$$
\Pi\tau u_{\circ}=\Pi\tau u_z+\Pi\tau(G_z-G )\phi=(\Theta+\Pi M_{z}\Pi)\phi
-\Pi M_{z}\Pi\phi=\Theta\phi\,.
$$
Then, by \eqref{f2},
\begin{align*}
A_{\Pi,\Theta}u=Au_z+zG_z\phi
=
Au_{\circ}-A(G_{z}-G )\phi+z G_z\phi
=
Au_{\circ}+\lambda_{\circ} G \phi\,.
\end{align*}
Finally, by \cite[Corollary 3.2]{P08} (also see \cite[Theorem 4.3]{P03}),  any self-adjoint extension of $A|\ker(\tau)$ is of the kind $A_{\Pi,\Theta}$ for some couple $(\Pi,\Theta)$. 
\end{proof}
\begin{remark}\label{not} 
Let us notice that the operators denoted by $G_{z}$ and $\Gamma_{z}$ in  \cite{P04} and \cite{P08} here correspond to $\t G_{z}$ and $\t M_{z}$ respectively.
\end{remark}
Let us remark that we have not used neither the adjoint $S^{*}$ nor the defect space  
$\ker(S^{*}-z)$. However these can be readily obtained:
\begin{lemma}\label{adj}
$$
\dom(S^{*})=\{u=u_{\circ}+G \phi\,,\ u_{\circ}\in\D(A)\,,\ \phi\in\fh'\}\,,
$$
$$
S^{*}u=Au_{\circ}+\lambda_{\circ}G \phi\,,
$$
and 
$$\ker(S^{*}-z)=\{G_{z}\phi\,,\ \phi\in\fh'\}\,,\quad z\in\rho(A)\,.
$$ 
\end{lemma}
\begin{proof} By \cite[Theorem 3.1]{P04} (see Remark \ref{not}) one has
$$
\dom(S^{*})=\{u=u_{*}+\frac12\,(G_{i}+G_{-i})\phi\,,\ u_{*}\in\dom(A)\,,\ \phi\in\fh'\}
$$
$$
S^{*}u=Au_{*}+\frac{i}2\,(G_{i}-G_{-i})\phi\,.
$$
Thus, posing $u_{\circ}:=u_{*}+\frac12\,(G_{i}-G )\phi +\frac12\,(G_{-i}-G )\phi$, one has 
$$\dom(S^{*})=\{u=u_{\circ}+G \phi\,,\ u_\circ\in\dom(A)\,,\ \phi\in\fh'\}$$ and
$$
S^{*}u=Au_{\circ}-\frac12\,A(G_{i}-G )\phi -\frac12\,A(G_{-i}-G )\phi
+\frac{i}2\,(G_{i}-G_{-i})\phi\,.
$$
By \eqref{f2} one then obtains $S^{*}u=Au_{\circ}+\lambda_{\circ}G \phi$. \par The vector  $u=u_{\circ}+G \phi\in\dom(S^{*})$ belong to $\ker(S^{*}-z)$ if and only if 
$(\lambda_{\circ}-z)G \phi=(-A+z)u_{\circ}$. This gives $u=(\lambda_{\circ}-z)R_{z}G \phi+G \phi$; by \eqref{f1} one gets $u=G_{z}\phi$.
\end{proof}
\begin{remark}By Lemma \ref{adj} and Theorem \ref{ext}, any self-adjoint extension of $S$ is of the kind
$$
A_{\Pi, \Theta}:=S^{*}|\D(A_{\Pi, \Theta})\,,
$$
$$
\D(A_{\Pi, \Theta})=\{u\in\D(S^{*}):\beta_{0}u\in\D( \Theta)\,,\ \Pi\beta_{1} u
= \Theta\beta_{0}u\}\,,
$$
where
$$
\beta_0:\D(S^{*})\to\fh'\,,\quad\beta_0  u:=\phi\,.
$$
$$
\beta_1 :\D(S^{*})\to\fh\,,\quad\beta_1 u:=\tau u_{\circ}\,,
$$
Moreover (using \cite[Theorem 3.1]{P04}) one has the abstract Green's identity
$$
\langle S^{*}u,v\rangle_{\H}-\langle u,S^{*}v\rangle_{\H}=
\langle\beta_1u,\beta_{0}v\rangle_{\fh\fh'}-\langle\beta_{0}u,\beta_1v\rangle_{\fh'\fh}
$$
Let us also notice that $u=G_{z}\phi$ solves the adjoint (abstract) boundary value problem
\be\label{bvp}
\begin{cases}
S^{*}u=z\,u&\\
\beta_{0}u=\phi\,.&\end{cases}
\ee
\end{remark}
\begin{remark} By \cite[Theorem 3.1]{P04}, the triple $(\fh,\Gamma_1 ,\Gamma_2)$, where $\Gamma_{1}:=-J^{-1}\beta_{0}$ and $\Gamma_{2}:=\beta_{1}$ is a boundary triplet for $S^{*}$, i.e. $\Gamma_1$ and $\Gamma_2$ are surjective and 
$$
\langle S^{*}u,v\rangle_{\H}-\langle u,S^{*}v\rangle_{\H}=
\langle\Gamma_1u,\Gamma_{2}v\rangle_{\fh}-\langle\Gamma_{2}u,\Gamma_1v\rangle_{\fh}
$$ 
holds true. The Weyl function of the boundary triple $(\fh,\Gamma_{1} ,\Gamma_1)$ is the bounded linear operator $M_{z}J:\fh\to\fh$, where $M_{z}$ is defined in \eqref{weyl} (see \cite[Theorem 3.1]{P04}). For Boundary Triple Theory we refer to \cite{DM}, \cite{Sch} and references therein.
\end{remark}
We conclude the section with the following result:
\begin{lemma}\label{bt}
Given the linear operator $\Xi:\dom(\Xi)\subseteq\ran(\Pi')\to \ran(\Pi)$, let us define the linear operator $A_{\Pi,\Xi}$ by 
$$
A_{\Pi,\Xi}:=S^*|\dom (A_{\Pi,\Xi})\,,\quad \dom(A_{\Pi,\Xi}):=\{u\in\D(S^{*}): \beta_{0}u\in\dom(\Xi)\,,\ \Pi\beta_{1}u=\Xi\beta_{0}u\}\,.
$$
Then $A_{\Pi,\Xi}$ is self-adjoint 
if and only if $\Xi$ is self-adjoint. 
\end{lemma}
\begin{proof} Since we can re-write $\dom(A_{\Pi,\Xi})$ as
$$\dom(A_{\Pi,\Xi})=\{u\in\D(S^{*}): \Gamma_{1}u\in J^{-1}(\dom(\Xi) )\,,\ \Pi\Gamma_{2}u=-\Xi J\Gamma_{1}u\}\,,$$
and since $(\fh,\Gamma_{1},\Gamma_{2})$ is a boundary triple, by e.g. \cite[Lemma 14.6 and Theorem 14.7]{Sch},  $A_{\Pi,\Xi}$ is self-adjoint 
if and only if $\Xi J:J^{-1}\dom(\Xi )\subseteq\ran(\Pi)\to \ran(\Pi)$ is self-adjoint. 
The latter is equivalent to $\Xi$ self-adjoint. 
\end{proof}

\end{section}
\begin{section}{Preliminaries: Sobolev spaces and boundary-layer operators}
\begin{subsection}{Sobolev spaces.}\label{sobolev} 
Let $\Omega$ be a non-empty open subset of $\RE^{n}$; $H^{k}(\Omega)$, $k\in\NA$, denotes the usual Sobolev-Hilbert spaces $H^{k}(\Omega):=\{u\in L^{2}(\Omega): \partial^{\alpha}u\in L^{2}(\Omega)\,,\ |\alpha|\le k\}$, where $\partial^{\alpha}:\Di'(\Omega)\to \Di'(\Omega)$ denotes the distributional partial derivatives of order $|\alpha|$. In the case $\Omega=\RE^{n}$, the scale of Sobolev-Hilbert spaces $H^{s}(\RE^{n})$, $s\in\RE$, is defined by $H^{s}(\RE^{n}):=\{u\in \mathscr{S}'(\RE^{n}):\smallint_{\RE^{n}}|\hat u(\kappa)|^{2}(|\kappa|^{2}+1)^{s}d\kappa<+\infty\}$, where $\mathscr{S}'(\RE^{n})$ is the space of tempered distributions and $\hat u$ denotes Fourier's transform. $H^{-s}(\RE^{n})$ identifies
with the dual space of $H^{s}(\RE^{n}) $; we denote by $\langle\cdot,\cdot\rangle_{-s,s}$ the $H^{-s}$-$H^{s}$ duality pairing.\par
Let us now suppose that $\Omega$ is bounded  
and let $\Gamma$ denote its boundary: $\Gamma=\partial\Omega$. 
We further suppose that  $\Omega\subset\RE^{n}$ is of class
$\mathcal{C}^{k,1}$, $k\ge 0$, i.e we suppose that 
its boundary $\Gamma$ is a manifold of dimension $n-1$ whose
local maps are Lipschitz-continuous, together with their inverses, up to the
order $k$. In the particular case $\C^{0,1}$, $\Omega$ is referred to as a Lipschitz domain. 
The scale of Sobolev-Hilbert space $H^{s}(\Omega)$, $s\in\RE$, is then defined by $H^{s}(\Omega):=\{u|\Omega: u\in H^{s}(\RE^{n})\}$, $u|\Omega$ denoting the restriction of $u$ to $\Omega$. 
In the case $s\in\NA$, this definition reproduces the previous one (see e.g. \cite[1.4.3.1]{Grisv}).\par
The Sobolev spaces of $L^{2}$-functions on $\Gamma$, next denoted with
$H^{s}\left(  \Gamma\right)  $, are defined by using an atlas of $\Gamma$ and the Sobolev space on flat, open, bounded, $(n-1)$-dimensional domains (see e.g. \cite[Section 1.3.3]{Grisv}, \cite[Chapter 3]{McLe}, \cite[Section 3.1]{Wl}); they are well defined up to the order $\left\vert
s\right\vert =k+1$, and $H^{-s}\left(  \Gamma\right)  $ identifies
with the dual space of $H^{s}\left(  \Gamma\right)  $. We denote by $\langle\cdot,\cdot\rangle_{-s,s}$ the $H^{-s}$-$H^{s}$ duality pairing.
If $\Gamma$ is a $\C^{2}$ manifold, considering the Riemannian structure inherited from 
$\RE^{n}$, we have that $H^{s}(\Gamma)$ identifies with  $\dom((-\Delta_{\Gamma})^{\frac{s}2})$ with respect of the scalar product
\begin{equation}
\langle \phi,\varphi\rangle_{H^{s}(\Gamma)}:=\langle \Lambda^{s}\phi,\Lambda^{s}\varphi\rangle_{L^{2}(\Gamma)}\,,\quad \Lambda:=(-\Delta_{\Gamma}+1)
^{\frac{1}{2}}\,, \label{Sobolev_Gamma}%
\end{equation}
being $\Delta_{\Gamma}$ the self-adjoint operator in $L^{2}(\Gamma)$ corresponding to the Laplace-Beltrami operator on the complete Riemannian manifold $\Gamma$ (see e.g. \cite[Remark 7.6, Chapter 1]{LiMa1}). 
According to this definition, $\Lambda^{r}$ is self-adjoint in $H^{s}(\Gamma)$ with domain $H^{s+r}(\Gamma)$ and acts as a unitary map $\Lambda^{r} :H^{s}\left(  \Gamma\right)\to H^{s-r}\left(  \Gamma\right)$. In particular $\Lambda^{2s}:H^{s}(\Gamma)\to H^{-s}(\Gamma)$ plays the role of the duality mapping $J$ introduced in Section \ref{Pre1}, and one has 
$\langle\phi,\psi\rangle_{-s,s}=\langle\Lambda^{-s}\phi,\Lambda^{s}\psi\rangle_{L^{2}(\Gamma)}$. 
In the case $\Gamma$ is not $\C^{2}$ one can use the definition (which works in the case $\Omega$ is of class $\C^{0,1}$) $\Delta_{\Gamma}:=\text{\rm Div}\circ\nabla_{\rm tan}$ provided in \cite[Theorem 1.2]{GMMM}. The two definitions coincide in the case $\Gamma$ is $\C^{2}$ (see \cite[Section 7]{GMMM}).
\par
Denoting by $0=\lambda_{0}<\lambda_{1}\le \dots\le \lambda_{k}\le \lambda_{k+1}\le \dots$, the increasing sequence of the eigenvalues of the self-adjoint operator $-\Delta_{\Gamma}:\dom(-\Delta_{\Gamma})\subset L^{2}(\Gamma)\to L^{2}(\Gamma)$ and by $\{\varphi_{k}\}_{k=1}^\infty$ the corresponding  normalized eigenfunctions, one has 
$$
\|\phi\|^{2}_{H^{s}(\Gamma)}=\sum_{k=0}^{\infty}(\lambda_{k}+1)^{s}|\langle\phi,\varphi_{k}\rangle_{L^{2}(\Gamma)}|^{2}\,.
$$
Thus, given $r<s<t$, for any $\epsilon>0$ there exists $c_{\epsilon}>0$, $c_{\epsilon}\uparrow\infty$ as $\epsilon\downarrow 0$, such that 
\be\label{rst}
\|\phi\|_{H^{s}(\Gamma)}\le \epsilon\,\|\phi\|_{H^{t}(\Gamma)}+c_{\epsilon}
\|\phi\|_{H^{r}(\Gamma)}\,.
\ee
If $\Sigma\subset\Gamma$ is relatively open,  
then $H^{s}\left(\Sigma\right)$,  $|s|\le k+1$, is constructed using an atlas of $\Sigma$ and the Sobolev spaces on flat  $(n-1)$-dimensional domains. If $\Sigma$ is of class $\C^{0,1}$, i.e. if its boundary is a Lipschitz manifold, then a continuation map allows the identification $H^{s}(\Sigma)  =\{  \phi|\Sigma: \phi\in H^{s}(\Gamma)\}$ and $\|\varphi\|_{H^{s}(\Sigma)}=\inf\{\,\|\phi\|_{H^{s}(\Gamma)}: \varphi=\phi|\Sigma\}$ (see e.g. \cite[Theorem 1.4.3.1]{Grisv}, \cite[Definition 3.6]{ShazPiri}, \cite[Section 4.3]{HW}).   
\par
In the sequel, we shall also use some closed subspaces of $H^{s}(\Gamma)$: let $X\subset \Gamma$ denote either $\Sigma^{c}\equiv\Gamma\backslash\Sigma$ or $\overline \Sigma$ and $s\ge 0$,  we define 
$$H^{s}_X(\Gamma):=\{\phi\in H^{s}(\Gamma):\supp(\phi)\subseteq X\}\equiv H^{s}(\Gamma)\cap L^{2}_X(\Gamma)\,,$$
where $L^{2}_X(\Gamma)$ denotes the space of square-integrable functions with essential support contained in $X$,
and 
$$
H^{-s}_X(\Gamma):=\{\phi\in H^{-s}(\Gamma):\langle\phi,\psi\rangle_{-s,s}=0,\  \text{for any $\psi\in H^{s}_{\overline{X^{c}}}(\Gamma)$}\}\,.$$ 
Therefore
\begin{align}\label{perp2}
\forall \phi\in H^{-s}_X(\Gamma),\,\forall 
\varphi\in H^{s}_{\overline {X^{c}}}(\Gamma)\,,\qquad
\langle\phi,\varphi\rangle_{-s,s}=0\,.
\end{align}
By the continuos (with dense range) embeddings $H^{s}(\Gamma)\hookrightarrow 
H^{r}(\Gamma)$, $r<s$, one gets
$$
H^{ r}_{X}(\Gamma)\cap H^{s}(\Gamma)=H^{s}_{X}(\Gamma)\,,\quad H_{X}^{s}(\Gamma)\subset 
H_{X}^{r}(\Gamma)\,\ (\text{dense inclusion}) \,.
$$ 
One has the identifications $$H^{s}(\Sigma)'\simeq H^{-s}_{\overline\Sigma}(\Gamma)\,,\qquad H^{s}_{\overline\Sigma}(\Gamma)'\simeq H^{-s}(\Sigma)$$ (see \cite[Proposition 3.5 and remarks at page 111]{ShazPiri}, \cite[Lemma 4.3.1]{HW}) and  the (strict, whenever $s\ge \frac12$) inclusions
$$
H^{s}_{\overline\Sigma}(\Gamma)\subseteq H^{s}(\Sigma)\subseteq L^{2}(\Sigma)\subseteq 
H^{-s}(\Sigma)\subseteq H^{-s}_{\overline\Sigma}(\Gamma)\,,\quad s>0\,.
$$
Denoting by $H_{0}^{s}(\Sigma)$, $s>0$, the completion of $H_{\overline\Sigma}^{s}(\Gamma)$ with respect to the norm of $H^{s}(\Sigma)$, one has (see \cite[relation (4.3.10)]{HW})
$$
H^{s}_{\overline\Sigma}(\Gamma)=\begin{cases} H^{s}_{0}(\Sigma)&s\not=[s]+\frac12\\
H^s_{00}(\Sigma)&s=[s]+\frac12\,,\end{cases}
$$
where
$$
H^s_{00}(\Sigma):=\{\phi\in H^s_{0}(\Sigma):d^{-\frac12}D^{\alpha}\phi\in L^{2}(\Sigma)\,,|\alpha|=[s]\}\,,
$$
$d$ denotes the distance to the boundary $\partial\Sigma$ and $D$ denotes the covariant derivative. 
In particular, since $H^{\frac12}_{0}(\Sigma)=H^{\frac12}(\Sigma)$,  it results
$$
 \phi\in H^{\frac12}(\Sigma)\quad\text{\rm and}\quad\int_{\Sigma}\frac{|\phi(x)|^{2}}{d(x)}\,d\sigma_{\Gamma}(x)<+\infty\quad\iff\quad\t\phi\in H^{\frac12}_{\overline\Sigma}(\Gamma)\,,
$$ 
where $\sigma_{\Gamma}$ denotes the surface measure and here $\t\phi$  denotes the extension by zero.\par
We shall also need the Hilbert orthogonal 
\begin{align}\label{perp}
H^{s}_{\Sigma^{c}}(\Gamma)^{\perp}
=\{\phi\in H^{s}(\Gamma): \langle\Lambda^{2s}\phi,\psi\rangle_{-s,s}=0,\ \text{for any $\psi\in H^{s}_{\Sigma^{c}}(\Gamma)$}\}
=\Lambda^{-2s}H^{-s}_{\overline \Sigma}(\Gamma)\,.
\end{align}
Let $\Pi_{\Sigma}:H^{s}(\Gamma)\to H^{s}(\Gamma)$ be the orthogonal projection onto $H^{s}_{\Sigma^{c}}(\Gamma)^{\perp}$, then (see e.g. \cite[page 77]{McLe}) the map 
$$\U_{\Sigma}: H^{s}_{\Sigma^{c}}(\Gamma)^{\perp}\to H^{s}(\Sigma)\,,\quad
{\rm U}_{\Sigma}(\Pi_{\Sigma}\phi):=(\Pi_{\Sigma}\phi)|\Sigma=\phi|\Sigma$$ is an unitary isomorphism. Therefore we can regard $H^{s}(\Sigma)$ as a closed subspace of $H^{s}(\Gamma)$. Using the decomposition $\phi=(\uno-\Pi_{\Sigma})\phi\oplus\U_{\Sigma}^{-1}(\phi|\Sigma)$, 
the restriction operator $R_{\Sigma}\phi:=0\oplus{\rm U}_{\Sigma}\Pi_{\Sigma}\phi=0\oplus(\phi|\Sigma)$ is the orthogonal projection from $H^{s}(\Gamma)\simeq H^{s}_{\Sigma^{c}}(\Gamma)\oplus H^{s}(\Sigma)$ onto $H^{s}(\Sigma)$.
\end{subsection}
\begin{subsection}{Sobolev Multipliers}
Let us now introduce the following notation: we write $\psi\in M^{s}(\Gamma)$, $s\ge 0$, whenever $\psi$ is a multiplier in $H^{s}(\Gamma)$, i.e. 
\begin{align}\label{mult}
&M^{s}(\Gamma):=\{\psi\in H^{s}(\Gamma): \text{$\psi\phi\in H^{s}(\Gamma)$ for any $\phi\in H^{s}(\Gamma)$}\}\\
=&\{\psi\in H^{s}(\Gamma):\text{$\exists\, m_{\psi}\ge 0$ s.t. $\|\psi\phi\|_{H^{s}(\Gamma)}\le m_{\psi}\,\|\phi\|_{H^{s}(\Gamma)}$}\}\,.\nonumber
\end{align}
The equality holds, by the closed graph theorem, since the map $\phi\mapsto\psi\phi$ is closed and everywhere defined. Notice that $M^{s}(\Gamma)\subseteq L^{\infty}(\Gamma)$:
$$
\|\psi\|_{L^{\infty}(\Gamma)}=\lim_{k\to+\infty}\|\psi\|_{L^{2k}(\Gamma)}=
\lim_{k\to+\infty}\|\psi^{k}\|^{1/k}_{L^{2}(\Gamma)}\le 
\lim_{k\to+\infty}\|\psi^{k}\|^{1/k}_{H^{s}(\Gamma)}\le m_{\psi}\lim_{k\to+\infty}\|1\|^{1/k}_{H^{s}(\Gamma)}=m_{\psi}\,.
$$
By the same kind of proofs which hold in the flat case (see \cite[Proposition 3.5.1, Corollary 3.5.7]{MS})
one has
$$
M^{s}(\Gamma)\subseteq M^{r}(\Gamma)\,, \quad r\le s\,, 
$$
and
$$
\text{$\psi\in M^{s}(\Gamma)$ and $1/\psi\in L^{\infty}(\Gamma)$}\quad\Rightarrow \quad1/\psi\in M^{s}(\Gamma)\,.
$$
We recall some relatively simple sufficient conditions in order that a given function belongs to $M^{s}(\Gamma)$. By \cite[Theorem 3.20]{McLe}, $$W^{k,\infty}(\Gamma)\subseteq M^{s}(\Gamma)\,,\qquad k\ge \max\{1,s\}\,,
$$
where $W^{k,\infty}(\Gamma)$ denotes the set of functions in $\C^{k-1}(\Gamma)$ with $k$-order distributional derivatives in $L^{\infty}(\Gamma)$. By \cite[Proposition 4.5]{Wl} this result can be improved to 
$$
A^{s}(\Gamma)\subseteq M^{s}(\Gamma)\,,
$$
where (see \cite[Section 4.1]{Wl}) $A^{k}(\Gamma)=W^{k,\infty}(\Gamma)$ whenever $s=k$ and, in case $s>0$ is not an integer, 
$$
A^{s}(\Gamma):=\left\{\phi\in W^{[s],\infty}(\Gamma): q_{i,\lambda,\phi}\in L^{\infty}(f_{i}(U_{i}))\,,\ i\in I\,,\ \lambda=s-[s]\,\right\}\,,$$
$$
q_{\phi,\lambda,i}(x):=\sum_{|\alpha|\le [s]}\int_{f_{i}(U_{i})}\frac{|\partial^{\alpha}\phi_{i}(x)-\partial^{\alpha}\phi_{i}(y)|^{2}}{\|x-y\|^{n-1+2\lambda}}\,dy\,,\quad \phi_{i}:= (\varphi_{i}\phi)\circ f_{i}^{-1}\,.
$$
Here $\{(U_{i},f_{i})\}_{i\in I}$ is an admissible atlas of $\Gamma$ and $\{\varphi_{i}\}_{i\in I}$ is a subordinate partition of unity. \par Notice that, for any $s-[s]<\kappa<1$, one has $\t\C^{[s],\kappa}(\Gamma)\subseteq A^{s}(\Gamma)$,  where $\t\C^{k,\kappa}(\Gamma)$ denotes the set of  functions in $W^{k,\infty}(\Gamma)$ having H\"older continuos (of exponent $\kappa$) $k$-order derivatives.
\par 
In the case $\Gamma$ is a smooth manifold, by \cite[Theorem 24]{CRT-N} one gets 
$$
H^{s}(\Gamma)=M^{s}(\Gamma)\,,\qquad s>\frac12\,(n-1)\,.
$$ 
By the embedding $H^{s}(\Gamma)\hookrightarrow L^{q}(\Gamma)$, $\frac1q=\frac12-\frac s{n-1}$, which holds whenever $s<\frac12\,(n-1)$, and by \cite[Theorem 27]{CRT-N}, one gets
$$
L^{(n-1)/s}_{s}(\Gamma)\cap L^{\infty}(\Gamma)\subseteq M^{s}(\Gamma)\,,\qquad s<\frac12\,(n-1)\,,
$$ 
where  $$L_{s}^{q}(\Gamma):=\{\phi\in L^{q}(\Gamma):(-\Delta_{\Gamma})^{\frac{s}2}\phi\in L^{q}(\Gamma)\}\,.
$$ 
\end{subsection}
\begin{subsection}{Trace maps.}
For a bounded open domain $\Omega$ of class $\C^{k,1}$, we pose
$$\Omega_{+}:=\mathbb{R}^{n}\backslash\overline\Omega\,,\qquad 
\Omega_{-}:=\Omega\,,$$ while $\nu$ denotes the outward
normal vector on $\Gamma$. The one-sided, zero-order, trace operators $\gamma_{0}^{\pm}$ act on a
smooth function $u\in\mathcal{C}^{\infty}\left(  \overline{\Omega}_{\pm
}\right)  $ as $\gamma_{0}^{\pm}u=u|{\Gamma
}$, where $\varphi|{\Gamma}$ is the restriction to $\Gamma
$. These maps uniquely extend to bounded linear operators (see e.g. \cite[Theorem 3.37]{McLe})
\begin{equation}
\gamma_{0}^{\pm}\in{\B}(  H^{s}\left(  \Omega_{\pm}\right)
,H^{s-\frac{1}{2}}\left(  \Gamma\right)  )  \,,\qquad\frac
{1}{2}<s\leq k+1\,. \label{Trace_Gamma_plusmin_est}%
\end{equation}
Then, given $a_{ij}\in\C^{\infty}(\RE^{n})$, $a_{ij}(x)=a_{ji}(x)$ such that  
\be\label{matrice}
\forall x,\xi\in\RE^{n}\,,\quad \sum_{1\le i,j\le n}a_{ij}(x)\xi_{i}\xi_{j}\ge c_{\circ}\,|\xi|^{2}\,,\quad c_{\circ}>0\,,
\ee
we define one-sided, first-order, trace operators 
\begin{equation}
\gamma_{1}^{\pm}\in{\B}(  H^{s}\left(  \Omega_{\pm}\right)
,H^{s-\frac{3}{2}}\left(  \Gamma\right)  )  \,,\qquad\frac
{3}{2}<s\leq k+1\,,
\end{equation}
by the zero-order trace of the co-normal derivative:
\be
\gamma_{1}^{\pm}u:=\sum_{1\le i,j\le n}\nu_{i}\gamma_{0}^{\pm}(a_{ij}\partial_{x_{j}}u)\label{Trace_op_test}%
\ee
Using the maps $\gamma_{0}^{\pm}$ and $\gamma^{\pm}_{1}$ we define the two-sided, bounded,  trace operators
$$
\gamma_{0}:H^{s}(\Omega_{-})\oplus H^{s}(\Omega_{+})\to H^{s-\frac12}(\Gamma)\,,\quad  \gamma_{0}(u_{-}\oplus u_{+}):=\frac12(\gamma_{0}^{+}u_{+}+\gamma_{0}^{-}u_{-})\,,
$$
$$
\gamma_{1}:H^{s}(\Omega_{-})\oplus H^{s}(\Omega_{+})\to H^{s-\frac32}(\Gamma)\,,\quad \gamma_{1}(u_{-}\oplus u_{+}):=\frac12(\gamma_{1}^{+}u_{+}+\gamma_{1}^{-}u_{-})
$$
and
$$
[\gamma_{0}]:H^{s}(\Omega_{-})\oplus H^{s}(\Omega_{+})\to H^{s-\frac12}(\Gamma)\,,\quad  [\gamma_{0}](u_{-}\oplus u_{+}):=\gamma_{0}^{+}u_{+}-\gamma_{0}^{-}u_{-}\,,
$$
$$
[\gamma_{1}]:H^{s}(\Omega_{-})\oplus H^{s}(\Omega_{+})\to H^{s-\frac32}(\Gamma)\,,\quad [\gamma_{1}](u_{-}\oplus u_{+}):=\gamma_{1}^{+}u_{+}-\gamma_{1}^{-}u_{-}\,.
$$
Posing
$$
H^{s}(\RE^{n}\backslash\Gamma):=H^{s}(\Omega_{-})\oplus H^{s}(\Omega_{+})\,,
$$
by \cite[Theorem 3.5.1]{Agra}, one has 
\be\label{H1}
H^{s}(\RE^{n})=H^{s}(\RE^{n}\backslash\Gamma)\,,\quad 0\le s<\frac12\,,
\ee
\be\label{H2}
H^{s}(\RE^{n})=H^{s}(\RE^{n}\backslash\Gamma)\cap\ker([\gamma_{0}])\,,\quad \frac 12<s<\frac 32\,,
\ee
\be\label{H3}
H^{s}(\RE^{n})=H^{s}(\RE^{n}\backslash\Gamma)\cap\ker([\gamma_{0}])\cap
\ker([\gamma_{1}])\quad \frac32< s<\frac 52\,.
\ee
More generally, given a relatively open subset $\Sigma\subset\Gamma$, one has
\be\label{H4}
H^{s}(\RE^{n}\backslash\overline\Sigma)=\{u\in H^{s}(\RE^{n}\backslash\Gamma):\supp([\gamma_{0}]u)\subseteq\overline \Sigma\}\,,\quad \frac12<s<\frac32\,,
\ee
and
\be\label{H5}
H^{s}(\RE^{n}\backslash\overline\Sigma)=\{u\in H^{s}(\RE^{n}\backslash\Gamma):\supp([\gamma_{0}]u)\cup \supp([\gamma_{1}]u)\subseteq\overline \Sigma \}\,,\quad \frac32< s<\frac52\,.
\ee
In the following we use the notations
$$
H^{s-}(\RE^{n}):=\bigcap_{r<s}H^{r}(\RE^{n})\,,\quad H^{s-}(\RE^{n}\backslash\Gamma):=\bigcap_{r<s}H^{r}(\RE^{n}\backslash\Gamma)\,,\quad H^{s-}(\RE^{n}\backslash\overline\Sigma):=\bigcap_{r<s}H^{r}(\RE^{n}\backslash\overline\Sigma)\,.
$$ 
\end{subsection}
\begin{subsection}{Boundary-layer operators.}\label{layers} In what follows $A$ denotes the 2nd order, symmetric, elliptic partial differential operator  
$$
A:\Di'(\RE^{n})\to \Di'(\RE^{n})\,,\quad Au:=\sum_{1\le i,j\le n}\partial_{x_i}(a_{ij}\partial_{x_j}u)-Vu\,,
$$
where we suppose $a_{ij},\, \partial_{x_{i}}a_{ij},\, V\in\C_{b}^{\infty}(\RE^{n})$, $a_{ij}(x)=a_{ji}(x)$ and that \eqref{matrice} holds true. When restricted to $H^{1}(\RE^{n})$, $A$ provides a bounded operator in $\B(H^{1}(\RE^{n}),H^{-1}(\RE^{n}))$ by the identity
$$\langle Au,v\rangle_{-1,1}=-\sum_{1\le i,j\le n}\langle a_{ij}\partial_{x_i}u,\partial_{x_j}v\rangle_{L^{2}(\RE^{n})}+
\langle Vu,v\rangle_{L^{2}(\RE^{n})}\,.$$
Moreover the sesquilinear form 
$$
F:H^{1}(\RE^{n})\times H^{1}(\RE^{n})\subset L^{2}(\RE^{n})\times L^{2}(\RE^{n})\to\RE \,,\quad 
F(u,v):=-\langle Au,v\rangle_{-1,1}$$
satisfies 
\be\label{F0}
\forall u\in H^{1}(\RE^{n})\,,\quad F(u,u)\ge c_{\circ}\|\nabla u\|^{2}_{L^{2}(\RE^{n})}-\|V_{\text{neg}}\|_{\infty} \,\|u\|^{2}_{L^{2}(\RE^{n})}\,,
\ee
where $V_{\text{neg}}$ denotes the negative part of $V$; therefore $F$ is closed and semibounded  . By \eqref{dtheta}, the corresponding self-adjoint operator is then given by the restriction of $A$ to the domain $D_{A}$,
$$
D_{A}:=\{u\in H^{1}(\RE^{n}): Au\in L^{2}(\RE^{n})
\}\equiv H^{2}(\RE^{n})\,.
$$
Thus $A: H^{2}(\RE^{n}) \subset L^{2}(\RE^{n})\to L^{2}(\RE^{n})$ is self-adjoint, $(-A+z)^{-1}\in\B(L^{2}(\RE^{n}),H^{2}(\RE^{n}))$ for any $z\in\rho(A)$ and $(\|V_{\text{neg}}\|_{\infty},+\infty)\subseteq\rho(A)$ . By \eqref{F0}, for any $\lambda>\|V_{\text{neg}}\|_{\infty}$ one obtains
\be
\forall u\in H^{1}(\RE^{n})\,,\quad \|(-A+\lambda) u\|_{H^{-1}(\RE^{n})}\ge c\, \|u\|_{H^{1}(\RE^{n})}\,,
\ee
and so 
\be\label{-1+1}
(-A+\lambda)^{-1}\in \B(H^{-1}(\RE^{n}),H^{1}(\RE^{n}))\,.
\ee
By \eqref{-1+1} and by elliptic regularity, see e.g. \cite[Theorem 6.22]{Grubb}, $(-A+\lambda)\in \B(H^{m+2}(\RE^{n}),H^{m}(\RE^{n}))$ is a bijection; thus,  by the inverse mapping theorem, one has 
\be\label{m,m+2}
(-A+\lambda)^{-1}\in \B(H^{m}(\RE^{n}),H^{m+2}(\RE^{n}))\,,\quad m\ge -1\,.
\ee
Given the bounded open set $\Omega\subset\RE^{n}$ of class $\C^{k,1}$, $k\ge 0$, the single and
double-layer operators 
$$  
\SL_{z}: H^{-\frac32}(\Gamma)\to L^{2}(\RE^{n})\,,\qquad
\DL_{z}: H^{-\frac12}(\Gamma)\to L^{2}(\RE^{n})\,,
$$
related to $A$ and $\Gamma$ are defined by 
\be\label{simple}
 \langle\SL_{z}\phi,u\rangle_{L^{2}(\RE^{n})}:=\langle\phi,\gamma_{0}(-A+\bar z)^{-1}u\rangle_{-\frac32,\frac32}\,,
\quad u\in L^{2}(\RE^{n})\,,
\ee
\be\label{double}
\langle\DL_{z}\varphi,u\rangle_{L^{2}(\RE^{n})}:=\langle\varphi,\gamma_{1}(-A+\bar z)^{-1}u\rangle_{-\frac12,\frac12}\,,
\quad u\in L^{2}(\RE^{n})
\,.
\ee
Let $g_{z}(x,y)$ be the integral kernel of the resolvent 
$(-A+z)^{-1}$; it is a smooth function for $x\not=y$ 
(see e.g. \cite[Lemma 6.3]{McLe}). Therefore \eqref{simple} and \eqref{double} give, if $x\notin\Gamma$ and $\phi,\,\varphi\in L^{2}(\Gamma)$, 
\begin{equation}
\SL_{z}\phi(x)=\int_{\Gamma}g_{z}(x,y)\,\phi(y)\,d\sigma_{\Gamma}(y)\,, \label{S_Gamma}%
\end{equation}
and%
\begin{equation}
\DL_{z} \varphi(x)=\sum_{1\le i,j\le n}\int_{\Gamma}\nu_{i}(y)a_{ij}(y)\partial_{x_{j}}g_{z}(x,y)\,\varphi(y)\,d\sigma_{\Gamma}(y)\,,\label{D_Gamma}%
\end{equation}
where $\sigma_{\Gamma}$ denotes the surface measure. We need the following mapping properties:
\begin{lemma}\label{mapprop} For any $\lambda>\|V_{\text{neg}}\|_{\infty}$ one has
$$
\SL_{\lambda}\in\B(H^{-\frac12}(\Gamma),H^{1}(\RE^{n}))\,,\qquad 
\DL_{\lambda}\in\B(H^{\frac12}(\Gamma),H^{1}(\Omega_{\pm}))\,
$$
\end{lemma}
\begin{proof} By \eqref{-1+1}, \eqref{simple} and by $\gamma_{0}\in\B(H^{1}(\RE^{n}),H^{\frac12}(\Gamma))$, if $\phi\in H^{-\frac12}(\Gamma)$ then 
$$
|\langle\SL_{\lambda}\phi,u\rangle_{L^{2}(\RE^{n})}|
\le\|\phi\|_{H^{-\frac12}(\Gamma)}\|\gamma_{0}(-A+\lambda)^{-1}u\|_{H^{\frac12}(\Gamma)}\le c\,\|\phi\|_{H^{-\frac12}(\Gamma)}\|u\|_{H^{-1}(\RE^{n})}\,.
$$ 
By \cite[Lemma 4.3]{McLe} and \cite[Theorem 3.30]{McLe} one has
$$
\|\gamma_{1}(-A+\lambda)^{-1}u\|_{H^{-\frac12}(\Gamma)}\le c\left(\|(-A+\lambda)^{-1} u\|_{H^{1}(\Omega_{\pm})}+\|u\|_{H^{1}(\Omega_{\pm})'}\right)\,.
$$
Thus, by \eqref{-1+1} and \eqref{double}, if $\varphi\in H^{\frac12}(\Gamma)$ there follows
$$
|\langle\DL_{\lambda}\varphi,u\rangle_{L^{2}(\Omega_{\pm})}|
\le\|\varphi\|_{H^{\frac12}(\Gamma)}\|\gamma_{1}(-A+\lambda)^{-1}1_{\Omega_{\pm}}u\|_{H^{-\frac12}(\Gamma)}\le c\,\|\varphi\|_{H^{\frac12}(\Gamma)}\|u\|_{H^{1}(\Omega_{\pm})'}\,.
$$ 
\end{proof}
Since $g_{z}(x,y)$ is a smooth function for $x\not=y$, one has $\SL_{z}\phi\,,\,\DL_z\varphi\in \C^{\infty}(\RE^{n}\backslash\Gamma)$ and (see \cite[eqs. (6.18) and (6.19)]{McLe}) 
\be\label{ker}
\forall x\notin\Gamma\,,\qquad A\,\SL_{z}\phi(x)=z\, \SL_{z}\phi(x)\,,\quad A\,\DL_{z}\varphi(x)=z\,\DL_{z}\varphi(x)\,.
\ee
Therefore, setting
$$\SL^{\pm}_{z}\phi:=\SL_{z}\phi|\Omega_{\pm}\,,\quad \DL^{\pm}_{z}\varphi:=\DL_{z}\varphi|\Omega_{\pm}\,,
$$
one has 
\be\label{lay-pm}
\SL_{z}^{\pm}\phi\in \ker(A^{\max}_{\pm}-z)\,,
\qquad \DL_{z}^{\pm}\varphi\in \ker(A^{\max}_{\pm}-z)\,,
\ee
where 
\be
A^{\max}_{\pm}=A|\dom(A^{\max}_{\pm})\,,\quad \dom(A^{\max}_{\pm}):=\{u_{\pm}\in L^{2}(\Omega_{\pm}): Au_{\pm}\in L^{2}(\Omega_{\pm})\}\,.
\ee
In the case $\Omega$ is of class $\C^{1,1}$, by proceeding as in the proof of theorem 6.5 in \cite[Section 6, Chapter 2]{LiMa1}  (see the comment in \cite{Grisv} before Theorem 1.5.3.4), the maps 
$\gamma_{0}^{\pm}$ and $\gamma_{1}^{\pm}$ can be extended to  
$$
\hat\gamma_{0}^{\pm}\in\B(\dom(A^{\max}_{\pm}),H^{-\frac12}(\Gamma))\,,\quad
\hat\gamma_{1}^{\pm}\in\B(\dom(A^{\max}_{\pm}),H^{-\frac32}(\Gamma))
$$
(here $\dom(A_{\max}^{\pm})$ has the graph norm), which in turn provide us with the bounded maps
$$
\hat \gamma_{0}:\dom(A^{\max}_{-})\oplus \dom(A^{\max}_{+})\to H^{-\frac12}(\Gamma)\,,\quad \hat \gamma_{0}(u_{-}\oplus u_{+}):=\frac12(\hat\gamma_{0}^{+}u_{+}+\hat\gamma_{0}^{-}u_{-})\,,
$$
$$
\hat \gamma_{1}:\dom(A^{\max}_{-})\oplus \dom(A^{\max}_{+})\to H^{-\frac32}(\Gamma)\,,\quad \hat \gamma_{1}(u_{-}\oplus u_{+}):=\frac12(\hat\gamma_{1}^{+}u_{+}+\hat\gamma_{1}^{-}u_{-})
$$
and
$$
[\hat \gamma_{0}]:\dom(A^{\max}_{-})\oplus \dom(A^{\max}_{+})\to H^{-\frac12}(\Gamma)\,,\quad [\hat \gamma_{0}](u_{-}\oplus u_{+}):=\hat\gamma_{0}^{+}u_{+}-\hat\gamma_{0}^{-}u_{-}\,,
$$
$$
[\hat \gamma_{1}]:\dom(A^{\max}_{-})\oplus \dom(A^{\max}_{+})\to H^{-\frac32}(\Gamma)\,,\quad [\hat \gamma_{1}](u_{-}\oplus u_{+}):=\hat\gamma_{1}^{+}u_{+}-\hat\gamma_{1}^{-}u_{-}\,.
$$
These maps, together with \cite[Theorem 7.2]{McLe}, give, whenever $\Omega$ is of class $\C^{k,1}$, $k\ge 1$, and for any $|s|\le k$, the bounded operators
\be\label{trace1}
\hat\gamma_{0}^{\pm}\SL_{z}\in\B(H^{s-\frac12}(\Gamma),H^{s+\frac12}(\Gamma))\,,\qquad
\hat\gamma^{\pm}_{1}\SL_{z}\in\B(H^{s-\frac12}(\Gamma),H^{s-\frac12}(\Gamma))\,,
\ee
\be\label{trace2}
\hat\gamma^{\pm}_{0}\DL_{z}\in\B(H^{s+\frac12}(\Gamma),H^{s+\frac12}(\Gamma))\,,\qquad
\hat\gamma_{1}^{\pm}\DL_{z}\in\B(H^{s+\frac12}(\Gamma),H^{s-\frac12}(\Gamma))\,.
\ee
Moreover the single and double layer operators satisfy the jump relations
\be\label{jump}
[\hat\gamma_{0}]\SL_{z}\phi=[\hat\gamma_{1}]\DL_{z}\phi=0\,,\quad 
[\hat\gamma_{1}]\SL_{z}\phi=-[\hat\gamma_{0}]\DL_{z}\phi=-\phi\,.
\ee
By the first relation in \eqref{jump} one gets
$$
\hat\gamma_{0}^{\pm}\SL_{z}=\hat\gamma_{0}\SL_{z}\,,\qquad \hat\gamma_{1}^{\pm}\DL_{z}=\hat\gamma_{1}\DL_{z}\,.
$$
Notice that in \eqref{trace1}, \eqref{trace2} and \eqref{jump} the extended trace operators coincide with the usual ones whenever the range spaces are Sobolev space on $\Gamma$ of strictly positive index.\par
For any $\lambda\in\rho(A)\cap\RE$ both the bounded operators $\gamma_{0}\SL_{\lambda}$ and $\hat\gamma_{1}\DL_{\lambda}$ are symmetric w.r.t. the $H^{-\frac12}(\Gamma)$-$H^{\frac12}(\Gamma)$ pairing (see \cite[Theorems 6.15 and 6.17]{McLe}: 
$$
\forall\phi,\varphi\in H^{-\frac12}(\Gamma)\,,\quad\langle\phi,\gamma_{0}\SL_{\lambda}\varphi\rangle_{-\frac12,\frac12}=
\langle\gamma_{0}\SL_{\lambda}\phi,\varphi\rangle_{\frac12,-\frac12}\,,
$$
$$
\forall\phi,\varphi\in H^{\frac12}(\Gamma)\,,\quad\langle\phi,\hat\gamma_{1}\DL_{\lambda}\varphi\rangle_{\frac12,-\frac12}=
\langle\hat\gamma_{1}\DL_{\lambda}\phi,\varphi\rangle_{-\frac12,\frac12}\,.
$$
Moreover these operators are coercive:
\begin{lemma} Let $\lambda>\max(V_{\text{neg}})$. Then there exist $c_{0}>0$ and $c_{1}>0$ such that 
\be\label{coercive1}
\forall\phi\in H^{-\frac12}(\Gamma)\,,\quad \langle\phi,\gamma_{0}\SL_\lambda\phi\rangle_{-\frac12,\frac12}\ge c_0\, \|\phi\|^{2}_{H^{-\frac12}(\Gamma)}
\ee
and
\be\label{coercive2}
\forall\varphi\in H^{\frac12}(\Gamma)\,,\quad -\langle\hat\gamma_{1}\DL_\lambda\varphi,\varphi\rangle_{-\frac12,\frac12}\ge c_{1}\, \|\varphi\|^{2}_{H^{\frac12}(\Gamma)}\,.
\ee
\end{lemma}
\begin{proof} In the case $A=\Delta$ and $\lambda=1$, the proof is given in 
\cite[Lemma 1.14 (c)] {KG} as regards $\hat\gamma_{0}\SL_{\lambda}$ and in \cite[Theorem 1.26 (e)]{KG} as regards 
$\hat\gamma_{1}\DL_{\lambda}$. Here we provide an alternative proof which adapts to our hypotheses. \par
By \cite[Lemma 4.3]{McLe} for any $u_{\pm}\in \D(A^{\max}_\pm)\cap H^1(\Omega_{\pm})$ one has $\hat\gamma^{\pm}_1
u_{\pm}\in H^{-\frac12}(\Gamma)$ and 
\be\label{4.3}
\|\hat\gamma^{\pm}_1
u_{\pm}\|^{2}_{H^{-\frac12}(\Gamma)}\le c\left(\|u_{\pm}\|^{2}_{H^{1}(\Omega_{\pm})}
+\|A^{\max}_{\pm}u_{\pm}\|^{2}_{H^{-1}(\Omega_{\pm})}\right)\,;
\ee
moreover for such a $u_{\pm }$ and for any $v_{\pm}\in H^1(\Omega_{\pm})$ the "half" Green's formula holds (see \cite[Theorem 4.4]{McLe}):
\begin{align*}
&\langle (-A_{\pm}^\max+\lambda) u_{\pm},v_{\pm}\rangle_{L^2(\Omega_{\pm})}\\
=& \sum_{1\le i,j\le n}\langle
a_{ij}\partial_{i}u_{\pm},\partial_{j} v_{\pm}\rangle_{L^2(\Omega_{\pm})} +\langle
(V+\lambda)u_{\pm},v_{\pm}\rangle_{L^2(\Omega_{\pm})}\pm\langle\hat\gamma^{\pm}_1
u_{\pm},\gamma_0^{\pm}v_{\pm}\rangle_{-\frac12,\frac12}\,.
\end{align*} 
Thus, posing $u_{\pm}=v_{\pm}=\SL^{\pm}_{\lambda}\phi$, by 
$(-A^{\pm}_\max+\lambda)u_{\pm}=0$ and by \eqref{jump}, one gets
\begin{align*}
0=&\sum_{1\le i,j\le n}\langle
a_{ij}\partial_{i}\SL_{\lambda}\phi,\partial_{j} \SL_{\lambda}\phi\rangle_{L^2(\RE^{n})} +\langle
(V+\lambda)\SL_{\lambda}\phi,\SL_{\lambda}\phi\rangle_{L^2(\RE^{n})}+\langle[\hat\gamma_1]
\SL_{\lambda}\phi,\gamma_0\SL_{\lambda}\phi\rangle_{-\frac12,\frac12}\\
\ge&c_{\circ}\|\nabla\SL_{\lambda}\phi\|^{2}_{L^2(\RE^{n})} +
(\lambda-\max(V_{\text{neg}}))\,\|\SL_{\lambda}\phi\|^{2}_{L^2(\RE^{n})}-\langle\phi,\gamma_0\SL_{z}\phi\rangle_{-\frac12,\frac12}\,.
\end{align*}
Therefore, setting $\kappa_{\circ}:=\min\{c_{\circ}, \lambda-\max(V_{\text{neg}})\}$, one obtains
$$
\langle\phi,\gamma_0\SL_{\lambda}\phi\rangle_{-\frac12,\frac12}\ge \kappa_{\circ}\,\|\SL_{\lambda}\phi\|^{2}_{H^{1}(\RE^{n})}\,.
$$
By $[\hat\gamma_1]
\SL_{\lambda}\phi=-\phi$, by \eqref{4.3} and by the continuous embedding $H^{1}(\Omega_{\pm})\hookrightarrow H^{-1}(\Omega_{\pm})$ 
one gets 
\begin{align*}
\|\phi\|^{2}_{H^{-\frac12}(\Gamma)}
\le& c\left(\|\SL_{\lambda}^{-}\phi\|^{2}_{H^{1}(\Omega_{-})}+\lambda^{2}\|\SL^{-}_{\lambda}\phi\|^{2}_{H^{-1}(\Omega_{-})}+\|\SL_{\lambda}^{+}\phi\|^{2}_{H^{1}(\Omega_{+})}+\lambda^{2}\|\SL^{+}_{\lambda}\phi\|^{2}_{H^{-1}(\Omega)}\right)\\
\le &c\,(1+\lambda^{2})\|\SL_{\lambda}\phi\|^{2}_{H^{1}(\RE^{n})} 
\end{align*}
and so \eqref{coercive1} follows by posing $c_{0}=\kappa_{\circ}(c\,(1+\lambda^{2}))^{-1}$. \par
The proof of \eqref{coercive2} proceeds along the same lines by inserting $u_{\pm}=v_{\pm}=\DL^{\pm}_{\lambda}\varphi$ in the Green's formula above: in this case one obtains
$$
-\langle\hat\gamma_1\DL_{\lambda}\varphi,\varphi\rangle_{-\frac12,\frac12}\ge \kappa_{0}\,\|\DL_{\lambda}\varphi\|^{2}_{H^{1}(\Omega_{-})\oplus H^{1}(\Omega_{+})}\,.
$$
By $[\gamma_0]
\DL_{\lambda}\varphi=\varphi$, denoting by $N_{\pm}$ the norm of  $\gamma_{0}^{\pm}\in\B(H^{1}(\Omega_{\pm}),H^{\frac12}(\Gamma))$, one obtains 
\begin{align*}
\|\varphi\|^{2}_{H^{\frac12}(\Gamma)}
\le N^{2}_{-}\|\DL_{\lambda}^{-}\varphi\|^{2}_{H^{1}(\Omega_{-})}+N^{2}_{+}\|\DL^{+}_{\lambda}\varphi\|^{2}_{H^{1}(\Omega_{+})}\le \max\{N^{2}_{-},N^{2}_{+}\} \|\DL_{\lambda}\varphi\|^{2}_{H^{1}(\Omega_{-})\oplus H^{1}(\Omega_{+})}\,,
\end{align*}
and so \eqref{coercive1} follows by posing $c_{1}=\kappa_{\circ}(\max\{N^{2}_{-},N^{2}_{+}\})^{-1}$.
\end{proof}
We conclude this section providing results about the mapping properties  of $\hat\gamma_{0}\SL_{\lambda}$ and $\hat\gamma_{1}\DL_{\lambda}$.
\begin{lemma}\label{regglob} Let $\lambda>\max(V_{\text{neg}})$ and $\Omega$ be of class $\C^{k,1}$, $k\ge 1$. Then
\be\label{regglob1}
\hat\gamma_{0}\SL_{\lambda}\phi\in H^{s+\frac12}(\Gamma)\quad\iff\quad \phi\in H^{s-\frac12}(\Gamma)\,,\quad |s|\le k-1\,,
\ee
\be\label{regglob2}
\hat\gamma_{1}\DL_{\lambda}\phi\in H^{s-\frac12}(\Gamma)\quad\iff\quad \phi\in H^{s+\frac12}(\Gamma)\,,
\quad \quad |s|\le k\,.
\ee
\end{lemma}
\begin{proof} By \eqref{trace1} and \eqref{trace2} we only need to prove the $\Rightarrow$  implications. By \cite[Theorem 7.17]{McLe}, both the maps $\hat\gamma_{0}\SL_{\lambda}$ and $\hat\gamma_{1}\DL_{\lambda}$ are Fredholm with index zero and $s$-independent kernel. By \eqref{coercive1} and \eqref{coercive2} such maps are injective and therefore bijective. Thus, by \eqref{trace1}, \eqref{trace2} and by the inverse mapping theorem, $(\hat\gamma_{0}\SL_{\lambda})^{-1}\in\B(H^{s+\frac12}(\Gamma),H^{s-\frac12}(\Gamma))$ and    
$(\hat\gamma_{1}\DL_{\lambda})^{-1}\in\B(H^{s-\frac12}(\Gamma),H^{s+\frac12}(\Gamma))$.
\end{proof}

\end{subsection}
\end{section}
\begin{section}{Self-adjoint realizations of singular perturbations supported on hypersurfaces.}
Let $A:H^{2}(\RE^{n})\subset L^{2}(\RE^{n})\to L^{2}(\RE^{n})$ be the elliptic, self-adjoint operator defined in Subsection \ref{layers}, i.e. 
$$
Au:=\sum_{1\le i,j\le n}\partial_{x_i}(a_{ij}\partial_{x_j}u)-Vu\,,
$$
with $a_{ij}, \partial_{x_{i}}a_{ij}, V\in\C_{b}^{\infty}(\RE^{n})$, the symmetric matrix $\underline a\equiv[a_{ij}(x)]$ satisfying \eqref{matrice}.\par  
Given $\Omega$ open, bounded and of class $\C^{1,1}$, posing 
$$
\tau:H^{2}(\RE^{n})\to H^{\frac{3}{2}}(  \Gamma)\oplus H^{\frac{1}{2}}(  \Gamma) \,,\quad
\tau u:=\gamma_{0}u\oplus \gamma_{1}u\,,
$$ 
one has  
\begin{lemma}\label{tau} The map $\tau$ is bounded, surjective and $\ker(\tau)$ is dense in $L^{2}(\RE^{n})$.
\end{lemma}
\begin{proof} Let $u\in H^{2}(\RE^{n})$. Then, by $\|u\|^{2}_{H^{2}(\RE^{n})}=\|u|\Omega_{-}\|^{2}_{H^{2}(\Omega_{-})}+\|u|\Omega_{+}\|^{2}_{H^{2}(\Omega_{+})}$, $\tau$ is bounded since both $\gamma_{0}^{\pm}$ and $\gamma^{\pm}_{1}$ are bounded; $\ker(\tau)$ is dense since it contains the dense set $\C^{\infty}_{\comp}(\RE^{n}\backslash\Gamma)$. By \cite[Remark 2.5.1.2]{Grisv}, both 
$$
\tau^{\pm}: H^{2}(\Omega_{\pm})\to H^{\frac{3}{2}}(  \Gamma)\oplus H^{\frac{1}{2}}(  \Gamma) \,,\quad
\tau^{\pm} u_{\pm}:=\gamma_{0}^{\pm}u_{\pm}\oplus \gamma_{1}^{\pm}u_{\pm}
$$
are surjective. Given $\phi\oplus\varphi\in H^{\frac{3}{2}}(  \Gamma)\oplus H^{\frac{1}{2}}(  \Gamma)$, let $u_{\pm}\in H^{2}(\Omega_{\pm})$ such that $\tau^{\pm}u_{\pm}=\phi\oplus\varphi$. Let $\t u_{\pm}\in H^{2}(\RE^{n})$ be an extensions of $u_{\pm}$. Then $\tau\left(\frac12\,(\t u_{-}+\t u_{+})\right)=\phi\oplus\varphi$ and so $\tau$ is surjective.
\end{proof}
By the definition of the map $\tau$ we have that 
$$
S:=A|\ker(\tau)=A^{\min}_{-}\oplus A^{\min}_{+}\,,
$$
where
$$
A_{\pm}^{\min}:=A|H^{2}_{0}(\Omega_{\pm})\,,\quad 
H^{2}_{0}(\Omega_{\pm}):=\{u_{\pm}\in H^{2}(\Omega_{\pm}):\gamma_{0}^{\pm}u_{\pm}=\gamma_{1}^{\pm}u_{\pm}=0\}\,.
$$
Since it is known that $(A^{\min}_{\pm})^{*}=A^{\max}_{\pm}$, one obtains 
\be\label{S*}
\dom(S^*)=\dom(A^{\max}_{-})\oplus \dom(A^{\max}_{+})\,,\qquad 
S^*=A^{\max}_{-}\oplus A^{\max}_{+}\,.
\ee
Thus we have the well-defined bounded (w.r.t. the graph norm in $\dom(S^*)$) maps 
$$
\hat\gamma_{0},\,[\hat\gamma_{0}]\in \B(\dom(S^{*}),H^{-\frac12}(\Gamma))\,,\quad
\hat\gamma_{1},\,[\hat\gamma_{1}]\in \B(\dom(S^{*}),H^{-\frac32}(\Gamma))\,.
$$
By Lemma \ref{tau} we can apply the results of Section \ref{Pre1} to $A$ and so find all self-adjoint extensions of 
the closed symmetric operator $A|\ker(\tau)=A^{\min}_{-}\oplus A^{\min}_{+}$. 
To this end we need to determine the operator $G_{z}$ and $M_{z}$ (see definitions \eqref{gz} and \eqref{weyl}). By \eqref{simple} and \eqref{double} this is immediate:
\be\label{Gz}
G_{z}: H^{-\frac32}(\Gamma)\oplus H^{-\frac12}(\Gamma)\to L^{2}(\RE^{n})\,,\quad G_{z}(\phi\oplus\varphi):=\SL_{z}\phi+\DL_{z}\varphi
\ee
and 
$$
M_{z}:H^{-\frac32}(\Gamma)\oplus H^{-\frac12}(\Gamma)\to 
H^{\frac32}(\Gamma)\oplus H^{\frac12}(\Gamma)\,,
$$
\be\label{Mz}
M_{z}:=\left[\,\begin{matrix}\gamma_{0}(\SL  -\SL_{z})&\gamma_{0}(\DL  -\DL_{z})\\
\gamma_{1}(\SL  -\SL_{z})&\gamma_{1}(\DL  -\DL_{z})
\end{matrix}\,\right]\,,
\ee
where  
$$\SL  :=\SL_{\lambda_{\circ}}\,,\qquad\DL  :=\DL_{\lambda_{\circ}}\,,\qquad \lambda_{\circ}>\max(V_{\text{neg}})\,.$$   
Next lemma provides a representation of $A^{\max}_{-}\oplus A^{\max}_{+}$ and of its domain. Before giving the precise statement we need some definitions. The distribution $\delta_{\Gamma}\in\mathscr{D}'(\RE^{n})$ is defined as usual by 
$$
\forall u\in\C^{\infty}_{\comp}(\RE^{n})\,,\quad (\delta_{\Gamma},u)=
\int_{\Gamma}u(x)\,d\sigma_{\Gamma}(x)\,.
$$
Given $f\in H^{-s}(\Gamma)$, we then define $f\delta_{\Gamma}\in\mathscr{D}'(\RE^{n})$ and $f\partial_{\underline a}\delta_{\Gamma}\in\mathscr{D}'(\RE^{n})$ by
$$
\forall u\in\C^{\infty}_{\comp}(\RE^{n})\,,\quad (f\delta_{\Gamma},u)=\langle \bar f,u|\Gamma\rangle_{-s,s}\,
$$
and 
$$
\forall u\in\C^{\infty}_{\comp}(\RE^{n})\,,\quad(f\partial_{\underline  a}\delta_{\Gamma},u):=-\sum_{1\le i,j\le n}(f\nu_{i}\delta_{\Gamma},a_{ij}\partial_{x_{j}}u)
\,.
$$
Notice that if $\Omega$ is of class $\C^{1,1}$ then $\nu$ is Lipschitz continuous and so the product $f\nu$ is a well-defined vector in $H^{-r}(\Gamma)$, $r=
\min\{1,s\}$. 
\begin{lemma}\label{L2} 
\begin{align*}
\dom(A^{\max}_{-})\oplus \dom(A^{\max}_{+})
=&
\{u=u_{\circ}+\SL  \phi+\DL  \varphi\,,\ u_{\circ}\in H^{2}(\RE^{n})\,,\ \phi\oplus\varphi\in H^{-\frac32}(\Gamma)\oplus H^{-\frac12}(\Gamma)\}\\
\equiv&\{u=u_{\circ}-\SL  [\hat\gamma_{1}]u+\DL  [\hat\gamma_{0}]u\,,\ u_{\circ}\in H^{2}(\RE^{n})\}\,,
\end{align*}
and
\be\label{aggiunto}
(A^{\max}_{-}\oplus A^{\max}_{+})u=
Au-[\hat\gamma_{1}]u\,
\delta_\Gamma-[\hat\gamma_{0}]u\,\partial_{\underline{a}} \delta_{\Gamma}\,.
\ee
\end{lemma} 
\begin{proof}
By \eqref{S*}, Lemma \ref{adj} and \eqref{Gz} one has 
\begin{align*}
\dom(A^{\max}_{-})\oplus \dom(A^{\max}_{+})
=
\{u=u_{\circ}+G (\phi\oplus\varphi)\,,\ u_{\circ}\in H^{2}(\RE^{n})\,,\ \phi\oplus\varphi\in H^{-\frac32}(\Gamma)\oplus H^{-\frac12}(\Gamma)\}\,.
\end{align*}
Since $[\gamma_{0}]u_{\circ}=[\gamma_{1}]u_{\circ}=0$, the proof of the first statement follows by using the jump relations \eqref{jump}. As regards the second statement,  in the case $A=\Delta$ the proof has been given in \cite[Theorem 3.1]{DPP} (be aware that there the jumps of the trace maps have been defined with opposite signs).  The proof in the more general case discussed here proceeds along the same lines and is left to the reader. 
\end{proof}
\begin{remark} 
By $(A^{\max}_{-}\oplus A^{\max}_{+})G_{z}(\phi\oplus\varphi)=zG_{z}(\phi\oplus\varphi)$ (see \eqref{bvp}) and by \eqref{aggiunto}, one has $AG_{z}(\phi\oplus\varphi)(x)=zG_{z}(\phi\oplus\varphi)(x)$ for any $x\in \RE^{n}\backslash\Gamma$. Hence, by elliptic regularity (see e.g. \cite[Theorem 6.4]{McLe}), setting $\Gamma^{\epsilon}:=\{x\in\RE^{n}:\text{\rm dist}(x,\Gamma)\le \epsilon\}$, $\epsilon>0$, one has 
\be\label{reg2}
\forall (\phi\oplus\varphi)\in H^{-\frac32}(\Gamma)\oplus H^{-\frac12}(\Gamma)\,,\qquad G_{z}(\phi\oplus\varphi)\in \bigcap_{s>0}H^{s}(\RE^{n}\backslash\Gamma^{\epsilon})\,.
\ee  
Similarly, if $\supp(\phi)\subseteq\overline\Sigma$ and $\supp(\varphi)\subseteq\overline\Sigma$, $\Sigma\subset\Gamma$ relatively open, then
by \eqref{aggiunto}, one has $AG_{z}(\phi\oplus\varphi)(x)=zG_{z}(\phi\oplus\varphi)(x)$ for any $x\in \RE^{n}\backslash\overline\Sigma$. Hence, by \cite[Theorem 6.4]{McLe}, setting $\Sigma^{\epsilon}:=\{x\in\RE^{n}:\text{\rm dist}(x,\Sigma)\le \epsilon\}$, $\epsilon>0$, 
\be\label{reg3}
\forall (\phi\oplus\varphi)\in H_{\overline\Sigma}^{-\frac32}(\Gamma)\oplus H_{\overline\Sigma}^{-\frac12}(\Gamma)\,,\qquad G_{z}(\phi\oplus\varphi)\in \bigcap_{s>0}H^{s}(\RE^{n}\backslash\Sigma^{\epsilon})\,.
\ee  
\end{remark}
From now on  $$\Pi:H^{\frac32}(\Gamma)\oplus H^{\frac12}(\Gamma)\to H^{\frac32}(\Gamma)\oplus H^{\frac12}(\Gamma)\,,$$ denotes an orthogonal projector, 
$$\Pi':H^{-\frac32}(\Gamma)\oplus H^{-\frac12}(\Gamma)\to H^{-\frac32}(\Gamma)\oplus H^{-\frac12}(\Gamma)\,,$$
denotes the orthogonal projector defined as the dual of $\Pi$, so that $\Pi'=(\Lambda^{3}\oplus\Lambda)\Pi(\Lambda^{-3}\oplus\Lambda^{-1})$, and $$\Theta:\dom(\Theta)\subseteq\ran(\Pi')\to\ran(\Pi)$$ denotes a self-adjoint operator. By Theorem \ref{ext} and Lemma \ref{L2}, one readily obtains all self-adjoint extension of $S$:
\begin{theorem}\label{T1} Any self-adjoint extension of $A^{\min}_{-}\oplus A^{\min}_{+}$ is of the kind $A_{\Pi,\Theta}$, where 
$$A_{\Pi,\Theta}:\dom(A_{\Pi,\Theta})\subseteq L^{2}(\RE^{n})\to L^{2}(\RE^{n})\,,\qquad A_{\Pi,\Theta}:=(A^{\max}_{-}\oplus A^{\max}_{+})|\dom(A_{\Pi,\Theta})\,,
$$ 
\begin{align*}
&\dom(A_{\Pi,\Theta}):=\{u\in \dom(A^{\max}_{-})\oplus\dom(A^{\max}_{+}):
(-[\hat\gamma_{1}]u)\oplus [\hat\gamma_{0}]u\in\D(\Theta)\,,\\
& \Pi (
\gamma_{0}(u+\SL  [\hat\gamma_{1}]u-\DL  [\hat\gamma_{0}]u)\oplus \gamma_{1}(u+\SL  [\hat\gamma_{1}]u-\DL  [\hat\gamma_{0}]u))=\Theta ((-[\hat\gamma_{1}]u)\oplus [\hat\gamma_{0}]u)\}\,.
\end{align*}
The set $$\text{$Z_{\Pi,\Theta}:=\{z\in\rho(A): \Theta+\Pi  M_{z}\Pi'\ \text{has a bounded inverse}\}$}$$ is not void; in particular $\CO\backslash\RE\subseteq Z_{\Pi,\Theta}\subseteq \rho(A_{\Pi,\Theta})$ and for any $z\in Z_{\Pi,\Theta}$ the resolvent of $A_{\Pi,\Theta}$ is given by 
\begin{equation}\label{krein1}
(-A_{\Pi,\Theta}+z)^{-1}u
=
(-A+z)^{-1}u+G_{z}\Pi' (\Theta+\Pi M_{z}\Pi')^{-1}\Pi 
(\gamma_{0}((-A+z)^{-1}u)\oplus\gamma_{1}((-A+z)^{-1}u))\,,
\end{equation}
where $G_{z}$ and $M_{z}$ are defined in \eqref{Gz} and \eqref{Mz} respectively.
\end{theorem}
\begin{remark} Let us notice that the $\Pi'$'s appearing in formula \eqref{krein1} act there as the inclusion map $\Pi':\ran(\Pi')\to H^{-3/2}(\Gamma)\oplus H^{-1/2}(\Gamma)$. This means that one does not need to know $\Pi'$ explicitly: it suffices to know the subspace $\ran(\Pi')=\ran(\Pi)'$.
\end{remark}
\begin{remark} Let us notice that the self-adjoint extension $A$ corresponds to the choice $\Pi=0$. By Lemma \ref{L2}, the choice $\Pi=\Pi_{1}\oplus 0$ gives $[\hat\gamma_{0}]u=0$ and so produces self-adjoint extension (''$\delta$-type'' interactions) of the kind $Au-[\hat\gamma_{1}]u\,\delta_{\Gamma}$ while the choice $\Pi=0\oplus \Pi_{2}$ gives $[\hat\gamma_{1}]u=0$ and so produces self-adjoint extension (''$\delta'$-type'' interactions) of the kind $Au-[\hat\gamma_{0}]u\,\partial_{\underline a}\delta_{\Gamma}$; different $\Pi$'s  give combinations of $\delta$ and $\delta'$ interactions.
\end{remark}
\begin{remark}\label{regularity}
Let $u=u_{\circ}+\SL  \phi+\DL  \varphi$ be in 
$\D(A_{\Pi,\Theta})$ and  suppose that 
\be\label{reg1}
\dom(\Theta)\subseteq H^{s_{1}}(\Gamma)\oplus H^{s_{2}}(\Gamma)\,,\quad s_{1}>-\frac32\,,\ s_{2}>-\frac12\,.
\ee
Then, by \eqref{reg2} and the mapping properties of single and double layer operators (see \cite[Corollary 6.14]{McLe}),    
 $$\SL  \phi+\DL  \varphi
\in 
H^{s}(\Omega_{-})\oplus H^{s}(\Omega_{+})\,,\quad s=\min\left\{s_{1}+\frac32,s_{2}+\frac12\right\}
$$ 
and so  
$$\dom(A_{\Pi,\Theta})\subseteq H^{s_{\circ}}(\RE^{n}\backslash\Gamma)\,,\quad 
s_{\circ}=\min\{2,s\}\,.
$$
In particular, $\dom(A_{\Pi,\Theta})\subseteq H^{2}(\RE^{n}\backslash\Gamma)$ whenever $s\ge 2$. The relation \eqref{reg3} suggests that such a kind of regularity could hold on a larger set whenever $\supp(\phi)\cup\supp(\varphi)\subseteq\overline\Sigma$, $\Sigma\subset\Gamma$. However, as the next result shows, one has to exclude a neighborhood of the interface $\partial\Sigma$: 
\end{remark}
\begin{lemma}\label{regularity2} Let 
\begin{align*}
D_{1,1}:=&\{\phi\in H^{-\frac32}_{\overline\Sigma}(\Gamma):\phi|\Sigma \in H^{\frac12}(\Sigma)\}\,,
\qquad\qquad\quad
D_{2,1}:=
\{\varphi\in H^{-\frac12}_{\overline\Sigma}(\Gamma):\varphi|\Sigma \in H^{\frac32}(\Sigma)\}\,,
\\
D_{1,2}:=&\{\phi\in H^{-\frac32}_{\overline\Sigma}(\Gamma):(\hat\gamma_{0}\SL\phi)|\Sigma\in H^{\frac32}(\Sigma)\}\,,
\qquad
D_{2,2}:=
\{\varphi\in H^{-\frac12}_{\overline\Sigma}(\Gamma):(\hat\gamma_{1}\DL\varphi)|\Sigma\in H^{\frac12}(\Sigma)\}\,.
\end{align*}
If 
$$
\dom(\Theta)\subseteq D_{1,i}\oplus D_{2,j}\,,\quad i,j\in\{1,2\}\,,
$$
then 
$$
\dom(A_{\Pi,\Theta})\subseteq H^{2}(\RE^{n}\backslash(\overline\Sigma\cup(\partial\Sigma)^{\epsilon}))\,,
$$
where
$$
(\partial\Sigma)^{\epsilon}:=\{x\in\RE^{n}:\text{\rm dist}(x,\partial\Sigma)\le \epsilon\}\,,\quad \epsilon>0\,.
$$
\end{lemma}
\begin{proof} If $\phi\in D_{1,1}$ and $\varphi\in D_{2,1}$, by \cite[Theorem 6.13]{McLe}, one has $\SL\phi\in H^{2}(\Omega_{1}^{\pm})$ and $\DL\varphi\in H^{2}(\Omega_{1}^{\pm})$, where $\Omega_{1}^{\pm}:=\Omega_{1}\cap\Omega_{\pm}$ and $\Omega_{1}$ of class $\C^{1,1}$ such that $\Omega_{1}\cap\Gamma$ is strictly contained in $\Sigma$. 
\par
If $\phi\in D_{1,2}$ and $\varphi\in D_{2,2}$, by \cite[Theorem 7.16]{McLe}, $\phi\in H^{\frac12}(\Sigma_{1})$ and $\varphi\in H^{\frac32}(\Sigma_{1})$, for any relatively open $\Sigma_{1}$ such that $\overline\Sigma_{1}\subset\Sigma$. Let now take $\Omega_{2}$ of class $\C^{1,1}$ such that $\Omega_{2}\cap\Gamma$ is strictly contained in $\Sigma_{1}$. Then, by \cite[Theorem 6.13]{McLe}, one has $\SL\phi\in H^{2}(\Omega_{2}^{\pm})$ and $\DL\varphi\in H^{2}(\Omega_{2}^{\pm})$, where $\Omega_{2}^{\pm}:=\Omega_{2}\cap\Omega_{\pm}$. \par 
The final statement then follows from \eqref{reg3}.
\end{proof}
In the case $\dom(A_{\Pi,\Theta})\subseteq H^{2}(\RE^{n}\backslash\Gamma)$, the statements contained in Theorem \ref{T1}  simplify and one has the following  
\begin{corollary}\label{C1} Suppose that $\D(\Theta)\subseteq H^{\frac12}(\Gamma)\oplus H^{\frac32}(\Gamma)$ and define  
$$
B_{\Theta}:\dom(\Theta)\subseteq \ran(\Pi')\to\ran(\Pi)\,,\quad B_{\Theta}:=\Theta +
\Pi B   \Pi'
\,,
$$
$$
B   :H^{\frac12}(\Gamma)\oplus H^{\frac32}(\Gamma)\to H^{\frac32}(\Gamma)\oplus H^{\frac12}(\Gamma)\,,\quad 
B   :=\left[\,\begin{matrix}\gamma_{0}\SL  &\gamma_0\DL  \\
\gamma_{1}\SL  &\gamma_{1}\DL  
\end{matrix}\,\right]\,.
$$
Then 
\begin{align*}
\D(A_{\Pi,\Theta})
=
\{u\in H^{2}(\RE^{n}\backslash\Gamma): (-[\gamma_{1}]u)\oplus 
[\gamma_{0}]u\in\D(\Theta)\,,\  
\Pi (\gamma_{0}u\oplus\gamma_{1}u)=B_{\Theta}((-[\gamma_{1}]u)\oplus[\gamma_{0}]u) \}
\end{align*}
and
\be\label{krein2}
(-A_{\Pi,\Theta}+z)^{-1}u\nonumber
=
(-A+z)^{-1}u+
G_{z}\Pi'(B_{\Theta}-\Pi M^{\circ}_{z}\Pi')^{-1}\Pi (\gamma_{0}(-A+z)^{-1}u\oplus\gamma_{1}(-A+z)^{-1}u)\,,
\ee
where 
$$
M^{\circ}_{z}:=\left[\,\begin{matrix}\gamma_{0}\SL_{z}&\gamma_0\DL_{z}\\
\gamma_{1}\SL_{z}&\gamma_{1}\DL_{z}
\end{matrix}\,\right]\,.
$$
\end{corollary}
\begin{proof} By Remark \ref{regularity}, 
$u\in H^{2}(\Omega_{-})\oplus H^{2}(\Omega_{+})$.
Thus $\Pi(\gamma_{0}u\oplus\gamma_{1}u)$ is well defined and, by the definition of $\dom(A_{\Pi,\Theta})$ given in Theorem \ref{T1}, 
\begin{align*}
&\Theta (\phi\oplus \varphi)
=\Pi (\gamma_{0}(u-\SL  \phi-\DL  \varphi)\oplus \gamma_{1}(u-\SL  \phi-\DL  \varphi))\\
=&\Pi(\gamma_{0}u\oplus\gamma_{1}u)-\Pi (\gamma_{0}(\SL  \phi+\DL  \varphi)\oplus \gamma_{1}(\SL  \phi+\DL  \varphi))\,.
\end{align*}
The proof is then concluded by the identity $\Theta+\Pi M_{z}\Pi'=B_{\Theta}-\Pi M_{z}^{\circ}\Pi'$.
\end{proof}
Let us recall some definitions:  ${\mathfrak S}_{\infty}(H_{1},H_{2})$ (${\mathfrak S}_{\infty}(H):={\mathfrak S}_{\infty}(H,H)$), denotes the operator ideal of compact operators on the Hilbert space $H_{1}$ to the Hilbert space $H_{2}$; ${\mathfrak S}_{p,\infty}(H_{1},H_{2})$ (${\mathfrak S}_{p,\infty}(H):={\mathfrak S}_{p,\infty}(H,H)$), $p>0$, denote the operator ideals of compact operators $T$ on the Hilbert space $H_{1}$ to the Hilbert space $H_{2}$ such that $s_{k}(T)=O(k^{-1/p})$, where the the singular values $s_{k}(T)$ are defined as the eigenvalues of the non-negative compact operator $(T^{*}T)^{\frac12}$. One has $T_{2}T_{1}\in {\mathfrak S}_{p,\infty}(H_{1},H_{2})$ whenever $T_{1}\in {\mathfrak S}_{p_{1},\infty}(H_{1},H_{0})$, $T_{2}\in {\mathfrak S}_{p_{2},\infty}(H_{0},H_{2})$ and $\frac1p=\frac1{p_{1}}+\frac1{p_{2}}$. Notice that if $p<q$ then ${\mathfrak S}_{p,\infty}(H_{1},H_{2})\subset{\mathfrak S}_{q}(H_{1},H_{2})$, where ${\mathfrak S}_{q}(H_{1},H_{2})$ denotes the Schatten-von Neumann ideal of compact operator with $q$-summable singular values; in particular $T\in {\mathfrak S}_{p,\infty}(H_{1},H_{2})$ is trace class whenever $p<1$.
\par
\begin{lemma}\label{ess}
Let $\Theta$ satisfy \eqref{reg1}. Then
$$
\sigma_{ess}(A_{\Pi,\Theta})=\sigma_{ess}(A)\,.
$$
\end{lemma}
\begin{proof} 
Here we follow the same kind of reasonings as in the proof of \cite[Lemma 4.7]{BLL1}. The operator  $(\Theta+\Pi M_{z}\Pi')^{-1}\in \B(\ran(\Pi),\ran(\Pi'))$ is closed as an operator from $\ran(\Pi)$ into $H^{-\frac32}(\Gamma)\oplus H^{-\frac12}(\Gamma)$, hence, since $\ran((\Theta+\Pi M_{z}\Pi')^{-1})=\dom(\Theta)\subseteq H^{s_{1}}(\Gamma)\oplus H^{s_{2}}(\Gamma)$, it is also
closed as an operator from $\ran(\Pi)$ into $H^{s_{1}}(\Gamma)\oplus H^{s_{2}}(\Gamma)$, which implies that it belongs to 
$\B(\ran(\Pi),H^{s_{1}}(\Gamma)\oplus H^{s_{2}}(\Gamma))$. Denoting 
 by $E_{s,r}$ the compact embedding of $H^{s}(\Gamma)$ into $H^{r}(\Gamma)$, $r<s$, by 
$$(\Theta+\Pi M_{z}\Pi')^{-1}=(E_{s_{1},-\frac32}\oplus E_{s_{2},-\frac12})(\Theta+\Pi M_{z}\Pi')^{-1}\,,
$$ 
one gets $(\Theta+\Pi M_{z}\Pi')^{-1}\in{\mathfrak S}_\infty(\ran(\Pi),\ran(\Pi'))$. Thus, by \eqref{krein1}, the resolvent difference $(-A_{\Pi,\Theta}+z)^{-1}-(-A+z)^{-1}$ belongs to ${\mathfrak S}_\infty(L^{2}(\RE^{n}))$ and so $\sigma_{ess}(A_{\Pi,\Theta})=\sigma_{ess}(A)$ by Weyl's theorem on the preservation of the essential spectrum under compact resolvent perturbations (see e.g. \cite[Theorem 8.12]{Sch}). \par 
\end{proof}

The next result applies to all the self-adjoint extensions given in Corollary \ref{C1} (and so to all the operators considered in Section 5). In the proof we follow the same arguments as in \cite{M} and \cite{BLL}.
\begin{theorem}\label{teo-schatten} Suppose that $\Gamma$ is smooth. If $\D(\Theta)\subseteq H^{s_{1}}(\Gamma)\oplus H^{s_{2}}(\Gamma)$ with $$s=\min\left\{s_{1}+\frac32,s_{2}+\frac12\right\}\ge 2\,,
$$ 
then for any integer $k\ge 1$ and for any $z\in \rho(A)\cap\rho(A_{\Pi,\Theta})$ one has
\be\label{schatten}
(-A_{\Pi,\Theta}+z)^{-k}-(-A+z)^{-k}\in {\mathfrak S}_{\frac{n-1}{2(k-1)+s},\infty}(L^{2}(\RE^{n}))\,.
\ee
\end{theorem}
\begin{proof} Given $z_{\circ}\in\CO\backslash\RE$, Re$(z_{\circ})\ge \lambda_{\circ}$, let $$
B:=(\tau(-A+\bar z_{\circ})^{-1})^{'}\Pi'(\Theta+\Pi M_{z_{\circ}}\Pi')^{-1}
$$
and
$$
C:=\Pi\tau(-A+z_{\circ})^{-1}\,,
$$
so that, by \eqref{krein1},
$$
(-A_{\Pi,\Theta}+z_{\circ})^{-1}-(-A+z_{\circ})^{-1}=BC\,.
$$ 
By \eqref{m,m+2}, one has $(-A+z_{\circ})^{-(m+1)}\in\B(L^{2}(\RE^{n}),H^{2m+2}(\RE^{n}))$ and so $\tau(-A+z_{\circ})^{-(m+1)}$ has range in $H^{2m+\frac32}(\Gamma)\oplus H^{2m+\frac12}(\Gamma)$. Then, by \cite[Lemma 4.7]{BLL1}, 
$$
(-A+z_{\circ})^{-(m+1)}\in {\mathfrak S}_{\frac{n-1}{2m},\infty}(L^{2}(\RE^{n}),H^{\frac32}(\Gamma)\oplus H^{\frac12}(\Gamma))$$ and so
$$ 
C(-A+z_{\circ})^{-m}\in {\mathfrak S}_{\frac{n-1}{2m},\infty}(L^{2}(\RE^{n}),\ran(\Pi))\,.
$$
By duality, that also gives 
\be\label{tag}
(-A+z_{\circ})^{-m}(\tau(-A+\bar z_{\circ})^{-1})'\Pi=(\Pi\tau(-A+\bar z_{\circ})^{-(m+1)})'\in {\mathfrak S}_{\frac{n-1}{2m},\infty}(\ran(\Pi'),L^{2}(\RE^{n}))\,.
\ee
Since $(\Theta+\Pi M_{z_{\circ}}\Pi')^{-1}\in\B(\ran(\Pi),\ran(\Pi'))$ and its range is contained in $H^{s_{1}}(\Gamma)\oplus H^{s_{2}}(\Gamma)$, by \cite[Lemma 4.7]{BLL1}, one gets 
$$
(\Theta+\Pi M_{z_{\circ}}\Pi')^{-1}\in {\mathfrak S}_{\frac{n-1}s,\infty}(\ran(\Pi), H^{-\frac32}(\Gamma)\oplus H^{-\frac12}(\Gamma))\,,
$$
and so, by \eqref{tag},
$$
(-A+z_{\circ})^{-m}B\in {\mathfrak S}_{\frac{n-1}{2m+s},\infty}(\ran(\Pi),L^{2}(\RE^{n}))\,.
$$
The proof is then concluded by \cite[Lemma 2.3]{BLL}.  
\end{proof}
Next we provide a version of Theorem \ref{teo-schatten} which applies to the case where the self-adjoint operator $\Theta$ is defined through the associated sesquilinear form.  In particular the next results apply to all the operators considered in Section 6. 
\begin{theorem}\label{teo-schatten2} 
Suppose that $\Gamma$ is smooth. Let $\t f$ be the sesquilinear form  associated to the self-adjoint operator in $\ran(\Pi)$ defined by $\t\Theta:=\Theta(\Lambda^{3}\oplus\Lambda)$;  if $\D(\t f)\subseteq H^{ s_{1}}(\Gamma)\oplus H^{s_{2}}(\Gamma)$ with $$s_{*}=\min\left\{s_{1}-\frac32, s_{2}-\frac12\right\}\ge 1\,,
$$ 
then for any integer $k\ge 1$ and for any $z\in \rho(A)\cap\rho(A_{\Pi,\Theta})$ one has
\be\label{schatten2}
(-A_{\Pi,\Theta}+z)^{-k}-(-A+z)^{-k}\in {\mathfrak S}_{\frac{n-1}{2(k-1)+2 s_{*}},\infty}(L^{2}(\RE^{n}))\,.
\ee
\end{theorem}
\begin{proof} According to our assumptions, $$\dom(\Theta)=(\Lambda^{3}\oplus\Lambda)\dom(\t\Theta)\subseteq (\Lambda^{3}\oplus\Lambda)\dom(\t f)\subseteq H^{-\frac12}(\Gamma)\oplus H^{\frac12}(\Gamma)\,. $$ 
Therefore hypothesis \eqref{reg1} holds and so, by Lemma \ref{ess}, $\sigma_{ess}(A_{\Pi,\Theta})=\sigma_{ess}(A)\subseteq(-\infty,\|V_{\text{\rm neg}}\|_{\infty}]$. Thus there exists $\lambda\in \rho(A)\cap\rho(A_{\Pi,\Theta})\cap\RE$ and so, by \eqref{krein1},  the  operator $\t\Theta+\Pi \t M_{\lambda}\Pi$ is self-adjoint and has a bounded inverse $(\t\Theta+\Pi \t M_{\lambda}\Pi)^{-1}\in\B(\ran(\Pi))$ (here $\t M_{z}:=M_{z}(\Lambda^{3}\oplus\Lambda)$). Let 
$$
\t\Theta+\Pi \t M_{\lambda}\Pi=\t U\,|\t\Theta+\Pi \t M_{\lambda}\Pi|
$$ be the polar decomposition of $\t\Theta+\Pi \t M_{\lambda}\Pi$ (see e.g. \cite[Section 7, Chapter VI]{K}). Since $\t\Theta+\Pi \t M_{\lambda}\Pi$ is self-adjoint and injective, $\t U$ is self-adjoint and unitary. Then 
$$(\t\Theta+\Pi \t M_{\lambda}\Pi)^{-1}=
|\t\Theta+\Pi \t M_{\lambda}\Pi|^{-1}\t U=|\t\Theta+\Pi \t M_{\lambda}\Pi|^{-\frac12}|\t\Theta+\Pi \t M_{\lambda}\Pi|^{-\frac12}\t U\,.
$$ 
Since $|\t\Theta+\Pi \t M_{\lambda}\Pi|^{-\frac12}$ is bounded and commutes with $\t\Theta+\Pi \t M_{\lambda}\Pi$, by \cite[Lemma 2.37, Chapter VI]{K}, it commutes with $\t U$. Therefore 
$$
(\t\Theta+\Pi \t M_{\lambda}\Pi)^{-1}=|\t\Theta+\Pi \t M_{\lambda}\Pi|^{-\frac12}\t U\,|\t\Theta+\Pi \t M_{\lambda}\Pi|^{-\frac12}
$$
and so, by \eqref{krein1},
$$
(-A_{\Pi,\Theta}+\lambda)^{-1}-(-A+\lambda)^{-1}=BC\,,
$$
where
$$
B=(\tau(-A+\lambda)^{-1})^{*}\Pi\,|\t\Theta+\Pi \t M_{\lambda}\Pi|^{-\frac12}\,,\quad C=\t UB^{*}\,.
$$
Since $\Pi \t M_{\lambda}\Pi$ is bounded, $\dom(\t f)=\dom(|\t\Theta|^{\frac12})=\dom(|\t\Theta+\Pi \t M_{\lambda}\Pi|^{\frac12})$ and so $$\ran(|\t\Theta+\Pi \t M_{\lambda}\Pi|^{-\frac12})\subseteq H^{s_{1}}(\Gamma)\oplus 
H^{s_{2}}(\Gamma)\,.
$$ 
Therefore, by \cite[Lemma 4.7]{BLL1}, one gets  
$$
|\t\Theta+\Pi \t M_{\lambda}\Pi|^{-\frac12}\in {\mathfrak S}_{\frac{n-1}{s_{*}},\infty}(\ran(\Pi), H^{\frac32}(\Gamma)\oplus H^{\frac12}(\Gamma))\,.
$$
Then, by \eqref{tag}, 
$$
(-A+\lambda)^{-m}B\in {\mathfrak S}_{\frac{n-1}{2m+s_{*}},\infty}(\ran(\Pi),L^{2}(\RE^{n}))\,.
$$
and 
$$
C(-A+\lambda)^{-m}\in {\mathfrak S}_{\frac{n-1}{2m+s_{*}},\infty}(L^{2}(\RE^{n}),\ran(\Pi))\,.
$$
 The proof is then concluded by \cite[Lemma 2.3]{BLL}.  
 \end{proof}

\begin{corollary}\label{spectrum} Let $\Gamma$ be smooth and let $\Theta$ satisfy the same hypotheses as either in Theorem \ref{teo-schatten} or in Theorem \ref{teo-schatten2}. Then
$$
\sigma_{ac}(A_{\Pi,\Theta})=\sigma_{ac}(A)
$$
and the wave operators 
$$W_{\pm}(A_{\Pi,\Theta},A)= \text{\rm s-}\lim_{t\to\pm\infty}e^{-itA_{\Pi,\Theta}}e^{itA}P_{ac}(A)\,, 
$$
$$W_{\pm}(A,A_{\Pi,\Theta})=\text{\rm s-}\lim_{t\to\pm\infty}e^{-itA}e^{itA_{\Pi,\Theta}}P_{ac}(A_{\Pi,\Theta})$$ 
exist and are complete, i.e. the limits exist everywhere and the ranges coincide with with the absolutely continuous subspaces.
\end{corollary}
\begin{proof} By either Theorem \ref{teo-schatten} or Theorem \ref{teo-schatten2}, for $k$ large enough, the resolvent difference $(-A_{\Pi,\Theta}+\lambda)^{-k}-(-A+\lambda)^{-k}$ is trace class; so, by the Birman-Kato criterion (see \cite[Theorem 4.8, Chapter X]{K}), one obtains the existence and completeness of the 
wave operators; thus $\sigma_{ac}(A_{\Pi,\Theta})=\sigma_{ac}(A)$.
\end{proof}
\begin{remark} In the case $A\le 0$ and $A_{\Pi,\Theta}\le 0$, by Corollary \ref{spectrum} and \cite[Sections 8 and 9]{K2}, one also gets the existence and completeness of wave operators for the pairs of wave equations $\partial^{2}_{tt} u=A_{\Pi,\Theta}u$ and $\partial^{2}_{tt}  u=Au$. 
\end{remark}
\begin{remark} Under additional hypotheses on the behavior at infinity of the coefficients of $A$, the spectral results in Lemma \ref{ess} and Corollary \ref{spectrum} can be specified. Let us suppose that 
\be\label{aij1}
a_{ij}(x)=a_{ij}^{\circ}+b_{ij}(x)\,,\quad b_{ij}(x)=O(1/\|x\|^{\delta})\,,
\ee
\be\label{aij2}
\partial_{x_{i}}b_{ij}(x)=O(1/\|x\|^{\delta})\,,\quad V(x)=O(1/\|x\|^{\delta})\,,
\ee
for some $\delta>1$, as $\|x\|\to+\infty$. Let $A_{\circ}$ be the differential operator with constant coefficients
$$
A_{\circ}:H^{2}(\RE^{n})\subset L^{2}(\RE^{n})\to L^{2}(\RE^{n})\,,\quad A_{\circ}u=\sum_{1\le i,j\le n}a^{\circ}_{ij}\partial^{2}_{x_ix_{j}}u\,.
$$
One has $\sigma(A_{\circ})=\sigma_{ac}(A_{\circ})=\sigma_{ess}(A_{\circ})=(-\infty,0]$. 
Since, by \cite[Theorem 5.3]{Bal},  
$\sigma_{ess}(A)=\sigma_{ess}(A_{\circ})$ and, by \cite[Theorem 2.1]{H1}, \cite[Chapter XIV]{H2}, $\sigma_{ac}(A)=\sigma_{ac}(A_{\circ})$, Lemma \ref{ess} and Corollary \ref{spectrum} give  
$$
\sigma_{ess}(A_{\Pi,\Theta})=\sigma_{ac}(A_{\Pi,\Theta})=(-\infty,0]\,.
$$
\end{remark}

\end{section}
\begin{section}{Applications: boundary conditions on $\Gamma$.}
Using the scheme provided by Theorem \ref{T1} and Corollary \ref{C1}, we next give the construction of some standard models 
of elliptic operators with boundary conditions on $\Gamma$, the boundary of a bounded domain $\Omega$ of class $\C^{1,1}$.
\begin{subsection}{Dirichlet boundary conditions.}\label{dirichlet} 
Let us consider the self-adjoint extension $A_{D}$ corresponding to Dirichlet boundary conditions on the whole $\Gamma$; it is given by the direct sum $A_{D}=A^{D}_{{-}}\oplus A^{D}_{{+}}$, where 
the self-adjoint operators $A^{D}_{\pm}$ are defined by $A^{D}_{\pm}:=A|\dom(A^{D}_{\pm})$, 
$\dom(A^{D}_{\pm})=\{u_{\pm}\in H^{2}(\Omega_{\pm}):\gamma_{0}^{\pm}u_{\pm}=0\}$. 
Since 
\begin{align*}
\dom(A^{D}_{-})\oplus\dom(A^{D}_+)=&\{u=u_{-}\oplus u_{+}\in H^{2}(\Omega_{-})\oplus H^2(\Omega_{+}): [\gamma_{0}]u=0\,,\gamma_{0}u=0\}\\
=&\{u\in H^{1}(\RE^{n})\cap H^{2}(\RE^{n}\backslash\Gamma):\gamma_{0}u=0\}\,,
\end{align*}
that corresponds, in Corollary \ref{C1}, to the choice  $\Pi=\Pi_{1}$, where $\Pi_{1}(\phi\oplus\varphi):=\phi\oplus 0$, and 
$B_{\Theta}=0$, i.e. $\Theta(\phi\oplus\varphi):=(-\Theta_{D}\phi)\oplus 0$, where $\Theta_{D}$ is the (necessarily self-adjoint, by Lemma \ref{bt}) operator
\be\label{ThetaD}
\Theta_{D}=\gamma_{0}\SL  :H^{\frac12}(\Gamma)\subseteq H^{-\frac32}(\Gamma)\to H^{\frac32}(\Gamma)\,.
\ee
Thus
$$
(A^{D}_{{-}}\oplus A^{D}_{{+}})u=Au-[\gamma_{1}]u\,\delta_{\Gamma}
$$
and, for any $z\in \rho(A)\cap\rho(A^{D}_{-})\cap\rho(A^{D}_{+})$,
\begin{align}\label{krein3}
(-(A^{D}_{{-}}\oplus A^{D}_{{+}})+z)^{-1}=&
(-A+z)^{-1}-\SL_{z}(\gamma_{0}\SL_{z})^{-1}\gamma_{0}(-A+z)^{-1}
\,.
\end{align}
Let $z\in \rho(A)\cap\rho(A^{D}_{-})\cap\rho(A^{D}_{+})$, so that, by \eqref{krein3}, $(\gamma_{0}\SL_{z})^{-1}\in\B(H^{\frac32}(\Gamma),H^{\frac12}(\Gamma))$. Given $\varphi\in H^{\frac32}(\Gamma)$, let us define $\phi:=(\gamma_{0}\SL_{z})^{-1}\varphi$ 
and $u_{\pm}:=\SL^{\pm}_{z}\phi$.  Then $A^{\max}_{\pm}u_{\pm}=z\,u_{\pm}$ and 
$\gamma^{\pm}_{0}u_{\pm}=\varphi$ and so one gets $u_{\pm}=K_{z}^{\pm}\varphi$, where $K_{z}^{\pm}\in\B(H^{s}(\Gamma),\dom(A^{\max}_{\pm}))$, $s\ge -\frac12$, is the Poisson operator which solves the Dirichlet boundary value problem
\be
\begin{cases}
(A^{\max}_{\pm}-z)K^{\pm}_{z}\psi=0&\\
\hat\gamma^{\pm}_0K^{\pm}_{z}\psi=\psi\,.
\end{cases}
\ee
To $K^{\pm}_{z}$ one associates the Dirichlet-to-Neumann operator $P_{z}^{\pm}\in \B(H^{s}(\Gamma),H^{s-1}(\Gamma))$, $s\ge -\frac12$, defined by $P_{z}^{\pm}:=\hat\gamma_{1}^{\pm}K_{z}^{\pm}$. Thus, since $[\gamma_{1}]\SL_{z}\phi=-\phi$, one has 
\be\label{DtN}
\forall z\in \rho(A)\cap\rho(A^{D}_{-})\cap\rho(A^{D}_{+})\,,\qquad (\gamma_{0}\SL_{z})^{-1}=P_{z}^{-}-P_{z}^{+}\,.
\ee
Therefore, by \eqref{krein3},
$$
(-(A^{D}_{{-}}\oplus A^{D}_{{+}})+z)^{-1}=
(-A+z)^{-1}-\SL_{z}(P_{z}^{-}-P^{+}_{z})\gamma_{0}(-A+z)^{-1}
\,.
$$
\end{subsection}  
\begin{subsection}{Neumann boundary conditions.}
Let us consider the self-adjoint extension $A_{N}$ corresponding to Neumann boundary conditions on the whole $\Gamma$; it is given by the direct sum $A_{N}=A^{N}_{{-}}\oplus A^{N}_{{+}}$, where 
the self-adjoint operators $A^{N}_{\pm}$ are defined by $A^{N}_{\pm}:=A|\dom(A^{N}_{\pm})$, 
$\dom(A^{N}_{\pm})=\{u_{\pm}\in H^{2}(\Omega_{\pm}):\gamma_{1}^{\pm}u_{\pm}=0\}$. 
Since 
$$
\dom(A^{N}_{-})\oplus\dom(A^{N}_+)=\{u\in H^{2}(\RE^{n}\backslash\Gamma): [\gamma_{1}]u=0\,,\gamma_{1}u=0\}\,,
$$
that corresponds, in Corollary \ref{C1}, to the choice  $\Pi=\Pi_{2}$, where $\Pi_{2}(\phi\oplus\varphi):=0\oplus\varphi$, and 
$B_{\Theta}=0$, i.e. $\Theta(\phi\oplus\varphi):=0\oplus (-\Theta_{N}\phi)$, where $\Theta_{N}$ is the (necessarily self-adjoint, by Lemma \ref{bt}) operator
\be\label{ThetaN}
\Theta_{N}=\gamma_{1}\DL  :H^{\frac32}(\Gamma)\subseteq H^{-\frac12}(\Gamma)\to H^{\frac12}(\Gamma)\,.
\ee
Thus
$$
(A^{N}_{{-}}\oplus A^{N}_{{+}})u=Au-[\gamma_{0}]u\,\partial_{\underline a}\delta_{\Gamma}\,,
$$
and, for any $z\in \rho(A)\cap\rho(A^{N}_{-})\cap\rho(A^{N}_{+})$
\begin{align}\label{krein4}
(-(A^{N}_{{-}}\oplus A^{N}_{{+}})+z)^{-1}=&
(-A+z)^{-1}-\DL_{z}(\gamma_{1}\DL_{z})^{-1}\gamma_{1}(-A+z)^{-1}
\,.
\end{align}
Let $z\in \rho(A)\cap\rho(A^{N}_{-})\cap\rho(A^{N}_{+})$, so that, by \eqref{krein4}, $(\gamma_{1}\DL_{z})^{-1}\in\B(H^{\frac12}(\Gamma),H^{-\frac12}(\Gamma))$. Given $\phi\in H^{\frac12}(\Gamma)$, let us define $\varphi:=(\gamma_{1}\DL_{z})^{-1}\phi$ 
and $u_{\pm}:=\SL^{\pm}_{z}\varphi$.  Then $A^{\max}_{\pm}u_{\pm}=z\,u_{\pm}$ and 
$\gamma^{\pm}_{1}u_{\pm}=\phi$ and so one gets $u_{\pm}=\t K_{z}^{\pm}\phi$, where $\t K_{z}^{\pm}\in\B(H^{s}(\Gamma),\dom(A^{\max}_{\pm}))$, $s\ge -\frac32$,  solves the 
boundary value problem
$$
\begin{cases}
(A^{\max}_{\pm}-z) \t K^{\pm}_{z}\psi=0&\\
\hat\gamma^{\pm}_1\t K^{\pm}_{z}\psi=\psi\,.
\end{cases}
$$
To $\t K^{\pm}_{z}$ one associates the Neumann-to-Dirichlet operator $Q_{z}^{\pm}\in \B(H^{s}(\Gamma),H^{s+1}(\Gamma))$, $s\ge -\frac32$, defined by $Q_{z}^{\pm}:=\hat\gamma_{0}^{\pm}\t K_{z}^{\pm}$. Thus, since $[\gamma_{0}]\DL_{z}\varphi=\varphi$, one has 
\be\label{NtD}
\forall z\in \rho(A)\cap\rho(A^{N}_{-})\cap\rho(A^{N}_{+})\,,\qquad (\gamma_{1}\DL_{z})^{-1}=Q_{z}^{+}-Q_{z}^{-}\,.
\ee
Therefore, by \eqref{krein4},
$$
(-(A^{N}_{{-}}\oplus A^{N}_{{+}})+z)^{-1}=(-A+z)^{-1}+\DL_{z}(Q_{z}^{-}-Q^{+}_{z})\gamma_{1}(-A+z)^{-1}
\,.
$$
\end{subsection}  
\begin{subsection}{Robin boundary conditions.}\label{robin} Let us consider the linear operator $A_{R}$ corresponding to Robin boundary conditions on the whole $\Gamma$; it is given by the direct sum $A_{R}=A^{R}_{{-}}\oplus A^{R}_{{+}}$, where 
 $$A^{R}_{\pm}:=A|\dom(A^{R}_{\pm})\,,\qquad 
\dom(A^{R}_{\pm})=\{u_{\pm}\in \dom(A_{\pm}^{\max}):\gamma_{1}^{\pm}u_{\pm}= b_{\pm}\,\gamma_{0}^{\pm}u_{\pm}\}\,.
$$ 
We suppose that $b_{\pm}\in M^{\frac12}(\Gamma)$ and that the functions $b_{\pm}$ are real-valued. Hence the operators $A^R_{\pm}$ are self-adjoint and $\dom(A^{\pm}_{R})\subseteq H^{2}(\Omega_{\pm})$ (use e.g \cite[Theorem 11]{GrubbTo}). In case $b_{+}(x)\not=b_{-}(x)$ for a.e. $x\in\Gamma$, the domain of $A^{R}_{-}\oplus A^{R}_{+}$ represents as
\begin{align}\label{AR}
&\dom(A^{R}_{-}\oplus A^{R}_{+})=\{u\in H^{2}(\RE^{n}\backslash\Gamma):\gamma_{1}^{\pm}u_{\pm}=b_{\pm}\,\gamma_{0}^{\pm}u_{\pm}\}
\nonumber
\\
=&\left\{u\in H^{2}(\RE^{n}\backslash\Gamma):(b_{+}-b_{-})\gamma_0u
=
[\gamma_{1}]u-\frac12(b_{+}+b_{-})[\gamma_{0}]u\,,\ 
(b_{+}-b_{-}) \gamma_{1}u
=
\frac12(b_{+}+b_{-})[\gamma_{1}]u-b_{+}b_{-}[\gamma_{0}]u
\right\}\,.
\end{align} 
Then, according to Corollary \ref{C1}, the self-adjoint operator $A^{R}_{-}\oplus A^{R}_{+}$  corresponds to the choice  $\Pi=1$ and $B_{\Theta}=B_{R}$, where
\be\label{BR}
B_{R}=-\frac{1}{[b]}\left[\,\begin{matrix}1  &\langle b\rangle  \\
\langle b\rangle  &{b_{+}b_{-}}
\end{matrix}\,\right]\,,
\quad
 \langle b\rangle:=\frac12(b_{+}+b_{-})\,,\quad \quad [b]:=b_{+}-b_{-}
\,,
\ee
provided that the operator 
\be\label{ThetaR}
\Theta=-\Theta_{R}\,,\quad \Theta_{R}:=
\left[\,\begin{matrix}1/[b]+\gamma_0\SL  &\langle b\rangle/[b]+\gamma_0\DL  \\
\langle b\rangle/[b]+\gamma_{1}\SL  &b_{+}b_{-}/[b]+\gamma_1\DL 
\end{matrix}\,\right]
\ee
is self-adjoint. This follows by the next lemma:
\begin{lemma}\label{LTR} If  $b_{\pm}\in M^{\frac32}(\Gamma)$ are real valued and $1/[b]\in L^{\infty}(\Gamma)$ , then
$$
\Theta_{R}:H^{\frac32}(\Gamma)\times H^{\frac32}(\Gamma)\subset H^{-\frac32}(\Gamma)\oplus H^{-\frac12}(\Gamma)\to H^{\frac32}(\Gamma)\oplus H^{\frac12}(\Gamma)
$$
 is self-adjoint.
\end{lemma} 
\begin{proof} $\Theta_{R}$ is a well-defined linear operator by \eqref{trace1} and \eqref{trace2}. By \eqref{ThetaR} and \eqref{AR}, setting $u_{\circ}:=u+\SL[\gamma_{1}]u-\DL[\gamma_{0}]u$, one has
\begin{align*}
&\{u\in H^{2}(\RE^{n}\backslash\Gamma):(-[\gamma_{1}]u)\oplus[\gamma_{0}]u\in \D(\Theta_{R})\,,\ \gamma_0u_{\circ}
\oplus\gamma_{1}u_{\circ}=-\Theta_{R}((-[\gamma_{1}]u)\oplus[\gamma_{0}]u)\}
\\=&\left\{u\in H^{2}(\RE^{n}\backslash\Gamma):\gamma_0u=
([\gamma_{1}]u-\langle b\rangle[\gamma_{0}]u)/[b]\,,\  
\gamma_{1}u=(\langle b\rangle[\gamma_{1}]u-b_{+}b_{-}[\gamma_{0}]u)/[b]
\right\}\\
=&\D(A_{-}^{R}\oplus A_{+}^{R})\,.
\end{align*}
Since $A_{-}^{R}\oplus A_{+}^{R}$ is self-adjoint, $\Theta_{R}$ is self-adjoint by Lemma \ref{bt}.
\end{proof} 
\begin{remark}
By Green's formula and Ehrling's lemma (here we procede as in the proof of \cite[Proposition 3.15]{BLL}{}) one has, for $\epsilon>0$ sufficiently small and $\frac12<s<1$,
\begin{align*}
\langle (-A_{\pm}^{R}+\lambda_{\circ})u_{\pm},u_{\pm}\rangle_{L^{2}(\Omega_{\pm})}
\ge &c_{\circ}\|\nabla u_{\pm}\|^{2}_{L^{2}(\Omega_{\pm})}-\|b_{\pm}\|_{L^{\infty}(\Gamma)}\|\gamma_{0}^{\pm}u_{\pm}\|^{2}_{L^{	2}(\Gamma)} \\
\ge &c_{\circ}\|\nabla u_{\pm}\|^{2}_{L^{2}(\Omega_{\pm})}-\|b_{\pm}\|_{L^{\infty}(\Gamma)}\|u_{\pm}\|^{2}_{H^{s}(\Omega_{\pm})}\\
\ge &c_{\circ}\|\nabla u_{\pm}\|^{2}_{L^{2}(\Omega_{\pm})}-\epsilon\, \|b_{\pm}\|_{L^{\infty}(\Gamma)}\|u_{\pm}\|^{2}_{H^{1}(\Omega_{\pm})}-c^{\pm}_{\epsilon}\|b_{\pm}\|_{L^{\infty}(\Gamma)}\|u_{\pm}\|^{2}_{L^{2}(\Omega_{\pm})}\\
\ge &-\kappa_{\epsilon}^{\pm}\|u_{\pm}\|^{2}_{L^{2}(\Omega_{\pm})}\,.
\end{align*}
Thus $A_{-}^{R}\oplus A_{+}^{R}$ is semibounded.
\end{remark}
According to Lemma \ref{LTR}, the Corollary \ref{C1} applies and we get 
$$
(A^{R}_{{-}}\oplus A^{R}_{{+}})u=Au-\frac4{[b]}\,\left((\langle b\rangle\,\gamma_{1}u-b_+b_{-}\gamma_{0}u)\,\delta_{\Gamma}+(\gamma_{1}u-\langle b\rangle\,\gamma_{0}u)\,\partial_{\underline a}\delta_{\Gamma}\right)\,.
$$
Moroever, for any $z\in \rho(A)\cap\rho(A^{R}_{-})\cap\rho(A^{R}_{+})$,
\begin{align*}
(-(A^{R}_{{-}}\oplus A^{R}_{{+}})+z)^{-1}u
=
(-A+z)^{-1}u
-G_{z} \left[\,\begin{matrix}1/{[b]}+\gamma_{0}\SL_{z}&\langle b\rangle/[b]+\gamma_0\DL_{z}\\
\langle b\rangle/[b]+\gamma_{1}\SL_{z}&b_{+}b_{-}/[b]+\gamma_{1}\DL_{z}
\end{matrix}\,\right]^{-1}\ 
\left[\,\begin{matrix}\gamma_{0}(-A+z)^{-1}u\\
\gamma_{1}(-A+z)^{-1}u\end{matrix}\,\right],
\end{align*}
where $G_{z}$ is defined in \eqref{Gz}. Let us notice that the case in which one has the same Robin boundary conditions on both sides of $\Gamma$ corresponds to the choice $b_{+}=b=-b_{-}$. Thus in this case one has 
$$
(A^{R}_{{-}}\oplus A^{R}_{{+}})u=Au-2b\,\gamma_{0}u\,\delta_{\Gamma}-(2/b)\,\gamma_{1}u\,\partial_{\underline a}\delta_{\Gamma}
$$
and
\begin{align*}
(-(A^{R}_{{-}}\oplus A^{R}_{{+}})+z)^{-1}u
=
(-A+z)^{-1}u
-G_{z} \left[\,\begin{matrix}1/{(2b)}+\gamma_{0}\SL_{z}&\gamma_0\DL_{z}\\
\gamma_{1}\SL_{z}&-b/2+\gamma_{1}\DL_{z}
\end{matrix}\,\right]^{-1}\ 
\left[\,\begin{matrix}\gamma_{0}(-A+z)^{-1}u\\
\gamma_{1}(-A+z)^{-1}u\end{matrix}\,\right].
\end{align*}
\end{subsection}  
\begin{subsection}{$\delta$-interactions.}\label{delta} 
Let  $\Pi(\phi\oplus\varphi)=\Pi_{1}(\phi\oplus\varphi):=\phi\oplus 0$ and $\Theta(\phi\oplus\varphi)=(-\Theta_{\alpha,D})\phi\oplus 0$, where  $\Theta_{\alpha,D}:=1/\alpha+\Theta_{D}=1/\alpha+\gamma_{0}\SL$ is the compression to $\ran(\Pi_{1})$ of $\Theta_{R}$ 
(here we consider the case $b_{+}=-b_{-}=\alpha/2$). This gives the boundary condition ${\alpha}\gamma_{0}u=[\gamma_{1}]u$ and so one obtains  the self-adjoint extensions usually called ''$\delta$-interactions on $\Gamma\,$'' (see \cite{BLL}, \cite{BEK} and references therein). In order to apply Corollary \ref{C1} we need the following
\begin{lemma}\label{lemmadelta}
If $\alpha\in M^{\frac32}(\Gamma)$ is real valued and $1/\alpha\in L^{\infty}(\Gamma)$, then $\Theta_{\alpha,D}:H^{\frac32}(\Gamma)\subseteq H^{-\frac32}(\Gamma)\to 
H^{\frac32}(\Gamma)$ is self-adjoint.
\end{lemma}
\begin{proof} By $1/\alpha\in M^{\frac32}(\Gamma)$, the linear operator $(1/\alpha):H^{\frac32}(\Gamma)\subset H^{-\frac32}(\Gamma)\to H^{\frac32}(\Gamma)$ is well-defined; by $\alpha\in M^{\frac32}(\Gamma)$ it is not difficult to check that it is self-adjoint. By \eqref{trace1}, $\gamma_{0}\SL  \alpha\in \B(H^{\frac12}(\Gamma),H^{\frac32}(\Gamma))$. Thus, by \eqref{rst}, for any $\epsilon>0$,
\begin{align*}
\|\gamma_{0}\SL  \phi\|_{H^{\frac32}(\Gamma)}=&
\|\gamma_{0}\SL  \alpha(1/\alpha)\phi\|_{H^{\frac32}(\Gamma)}
\le
c\,\|(1/\alpha)\phi\|_{H^{\frac12}(\Gamma)}\\
\le &c\,\epsilon\, 
\|(1/\alpha)\phi\|_{H^{\frac32}(\Gamma)}+{c}_{\epsilon}\,\|(1/\alpha)\phi\|_{H^{-\frac32}(\Gamma)}\\
\le & 
c\,\left(\epsilon\, 
\|(1/\alpha)\phi\|_{H^{\frac32}(\Gamma)}+{c}_{\epsilon}\,\|\phi\|_{H^{-\frac32}(\Gamma)}\,\right)
\end{align*}
and so the self-adjoint operator $\gamma_{0}\SL  :H^{\frac12}(\Gamma)\subset H^{-\frac32}(\Gamma)\to H^{\frac32}(\Gamma)$ is infinitesimally $(1/\alpha)\,$-bounded. The proof is then concluded by \cite[Corollary 1]{HK}.
\end{proof}
Therefore, by Corollary \ref{C1} one gets the  self-adjoint extension
$$
A_{\alpha,\delta}\,u=Au-{\alpha}\gamma_{0}u\,\delta_{\Gamma}\,,
$$
$$
\dom(A_{\alpha,\delta}):=\{u\in H^{1}(\RE^{n})\cap H^{2}(\RE^{n}\backslash\Gamma):{\alpha}\gamma_{0}u=[\gamma_{1}]u\}\,.
$$
By \eqref{krein2} and \eqref{DtN}, its resolvent is given by 
\begin{align*}
&(-A_{\alpha,\delta}+z)^{-1}=
(-A+z)^{-1}-\SL_{z}((1/\alpha)+\gamma_{0}\SL_{z})^{-1}\gamma_{0}(-A+z)^{-1}
\\=&(-A+z)^{-1}-\SL_{z}(P_{z}^{-}-P^{+}_{z})(\alpha+P_{z}^{-}-P^{+}_{z})^{-1}\alpha\gamma_{0}(-A+z)^{-1}
\,. 
\end{align*}
\end{subsection}  

\begin{subsection}{$\delta'$-interactions.}\label{deltaprimo} 
Let  $\Pi(\phi\oplus\varphi)=\Pi_{2}(\phi\oplus\varphi):=0\oplus\varphi$ and $\Theta(\phi\oplus\varphi)=0\oplus(-\Theta_{\beta,N}\varphi)$, where $\Theta_{\beta,N}:=-1/\beta+\Theta_{N}=-1/\beta+\gamma_{1}\DL$ is the compression to $\ran(\Pi_{2})$ of $\Theta_{R}$ 
(here we consider the case $b_{+}=-b_{-}=2/\beta$). This gives the boundary condition ${\beta}\gamma_{1}u=[\gamma_{0}]u$ and so one obtains  the self-adjoint extensions usually called ''$\delta'$-interactions on $\Gamma\,$'' (see \cite{BLL}and references therein). In order to apply Corollary \ref{C1} we need the following
\begin{lemma}\label{lemmadeltaprimo}
If $\beta\in M^{\frac12}(\Gamma)$ is real valued and $1/\beta\in L^{\infty}(\Gamma)$, then  $\Theta_{\beta,N}:H^{\frac32}(\Gamma)\subseteq H^{-\frac12}(\Gamma)\to 
H^{\frac12}(\Gamma)$ is self-adjoint.
\end{lemma}
\begin{proof} 
By \eqref{coercive2}, one has $(\hat\gamma_{1}\DL  )^{-1}\in\B(H^{-\frac12}(\Gamma),H^{\frac12}(\Gamma))$. Thus, by \eqref{rst} and \eqref{trace2}, one gets, for any $\epsilon>0$,
\begin{align*}
\|(1/\beta)\varphi\|_{H^{\frac12}(\Gamma)}=&
\|(1/\beta)(\hat\gamma_{1}\DL  )^{-1}\hat\gamma_{1}\DL  \varphi\|_{H^{\frac12}(\Gamma)}
\le
c\,\|\hat\gamma_{1}\DL  \varphi\|_{H^{-\frac12}(\Gamma)}\\
\le &c\,\epsilon\, 
\|\hat\gamma_{1}\DL  \varphi\|_{H^{\frac12}(\Gamma)}+{c}_{\epsilon}\,\|\hat\gamma_{1}\DL   \varphi\|_{H^{-\frac32}(\Gamma)}\\
\le & c\left(\epsilon\, 
\|\hat\gamma_{1}\DL   \varphi\|_{H^{\frac12}(\Gamma)}+{c}_{\epsilon}\,\|\varphi\|_{H^{-\frac12}(\Gamma)}\,\right)
\end{align*}
and so the operator $(1/\beta):H^{\frac32}(\Gamma)\subseteq H^{-\frac12}(\Gamma)\to H^{\frac12}(\Gamma)$ is infinitesimally $\hat\gamma_{1}\DL\,  $-bounded. 
Since $\hat\gamma_{1}\DL  :H^{\frac32}(\Gamma)\subseteq H^{-\frac12}(\Gamma)\to H^{\frac12}(\Gamma)$ is self-adjoint, the proof is then concluded by \cite[Corollary 1]{HK}.
\end{proof}
Therefore, by Corollary \ref{C1} one gets the  self-adjoint extension
$$
A_{\beta,\delta'  }u=Au-{\beta}\gamma_{1}u\,\partial_{\underline a}\delta_{\Gamma}\,,
$$
$$
\dom(A_{\beta,\delta'  }):=\{u\in H^{2}(\RE^{n}\backslash\Gamma):[\gamma_{1}]u=0\,,\ {\beta}\gamma_{1}u=[\gamma_{0}]u\}\,.
$$
By \eqref{krein2} and \eqref{NtD}, its resolvent is given by
\begin{align*}
&(-A_{\beta,\delta'}+z)^{-1}=
(-A+z)^{-1}+\DL_{z}((1/\beta)-\gamma_{1}\DL_{z})^{-1}\gamma_{1}(-A+z)^{-1}
\\=&(-A+z)^{-1}+\DL_{z}(Q_{z}^{+}-Q^{-}_{z})(-\beta+Q_{z}^{+}-Q^{-}_{z})^{-1}\beta\gamma_{1}(-A+z)^{-1}
\,.
\end{align*}
\end{subsection}  
\end{section}
\begin{section}{Applications: boundary conditions on $\Sigma\subset\Gamma$.}
Using the scheme provided by Theorem \ref{T1}, we next give the construction of some models 
of elliptic operators with boundary conditions on a relatively open part $\Sigma\subset\Gamma$ of class $\C^{0,1}$. In such cases we need more regularity hypotheses on $\Omega$ (with respect to the ones used in the previous section); these are needed in the proofs, in order that Sobolev spaces of appropriate order can be properly defined.  \par
\begin{subsection}{Dirichlet boundary conditions.}\label{dir-loc} Given $\Sigma\subset \Gamma$ relatively open of class $\C^{0,1}$,  we denote by $\Pi_{\Sigma}$ the orthogonal projector in the Hilbert space $H^{\frac32}(\Gamma)$ such that $\ran(\Pi_{\Sigma})=H^{\frac32}_{\Sigma^{c}}(\Gamma)^{\perp}$. By \eqref{perp}, $\ran(\Pi_{\Sigma}')=H^{-\frac32}_{\overline\Sigma}(\Gamma)$, where $\Pi'_{\Sigma}=\Lambda^{3}\Pi_{\Sigma}\Lambda^{-3}$ is the dual projection. In the following, we also use the identifications $H^{\frac32}_{\Sigma^{c}}(\Gamma)^{\perp}\simeq H^{\frac32}(\Sigma)$ and $H^{\frac32}(\Sigma)'\simeq H^{-\frac32}_{\overline\Sigma}(\Gamma)$. We denote the orthogonal projection from  $H^{\frac32}(\Gamma)$ onto $H^{\frac32}(\Sigma)$ by $R_{\Sigma}$; it can be identified with the restriction map $R_{\Sigma}:H^{s}(\Gamma)\to H^{s}(\Sigma)$, $R_{\Sigma}\phi:=\phi|\Sigma$ (see the end of Subsection \ref{sobolev}).
\begin{theorem}\label{compr} Let $\Omega$ be of class $\C^{3,1}$, then  
 $$
\Theta_{D,\Sigma}: \dom(\Theta_{D,\Sigma})\subseteq H^{-\frac32}_{\overline\Sigma}(\Gamma)\to H^{\frac32}(\Sigma)\,,\quad \Theta_{D,\Sigma}\phi:= R_{\Sigma}\gamma_{0}\SL \Pi_{\Sigma}'\phi\equiv (\gamma_{0}\SL \phi)|\Sigma\,,
$$
$$
\dom(\Theta_{D,\Sigma}):=\{\phi\in H^{-\frac12}_{\overline\Sigma }(\Gamma):(\gamma_{0}\SL  \phi)|\Sigma\in H^{\frac32}(\Sigma)\}
$$
is self-adjoint.
\end{theorem}
\begin{proof} Let $f_{D}$ be the densely defined sesquilinear form in the Hilbert space $H^{\frac32}(\Gamma)$ 
$$f_{D}: H^{5/2}(\Gamma)\times H^{5/2}(\Gamma)\subseteq H^{\frac32}(\Gamma)\times H^{\frac32}(\Gamma)\to \RE$$  
$$
f_{D}(\phi_{1},\phi_{2}):=\langle \Lambda^{3}\phi_{1},\gamma_{0}\SL  \Lambda^{3}\phi_{2}
\rangle_{-\frac12,\frac12}=\langle \gamma_{0}\SL  \Lambda^{3}\phi_{1},\Lambda^{3}\phi_{2}
\rangle_{\frac12,-\frac12}\,.
$$ 
By \eqref{coercive1},
$$ 
f_{D}(\phi,\phi)\ge c_{0} \,\|\Lambda^{3}\phi\|^{2}_{-\frac12}=c_{0}\,\|\phi\|^{2}_{H^{5/2}(\Gamma)}
$$
and so $f_{D}$ is strictly positive and closed. 
Since, for any $\phi_{1}\in H^{7/2}(\Gamma)$ and for any $\phi_{2}\in H^{5/2}(\Gamma)$,
$$ 
f_{D}(\phi_{1},\phi_{2})=\langle \Lambda^{\frac32}\gamma_{0}\SL  \Lambda^{3}\phi_{1}
,\Lambda^{\frac32}\phi_{2},\rangle_{L^{2}(\Gamma)}=\langle \gamma_{0}\SL  \Lambda^{3}\phi_1,\phi_{2}\rangle_{H^{\frac32}(\Gamma)}\,,
$$
$f_{D}$ is the sesquilinear form associated with the  
self-adjoint operator in $H^{\frac32}(\Gamma)$ defined by 
$$\t\Theta_{D}:=\gamma_{0}\SL  \Lambda^{3}:H^{7/2}(\Gamma)\subseteq H^{\frac32}(\Gamma)\to H^{\frac32}(\Gamma)\,.
$$
The domain $$\dom(f_{D})\cap\ran(\Pi_{\Sigma})=H^{5/2}(\Gamma)\cap H^{\frac32}_{\Sigma^{c}}(\Gamma)^{\perp}=H^{5/2}(\Gamma)\cap\Lambda^{-3}H^{-\frac32}_{\overline\Sigma}(\Gamma)=\Lambda^{-3}
H^{-\frac12}_{\overline\Sigma }(\Gamma)$$ is dense in $\Lambda^{-3}H^{-\frac32}_{\overline\Sigma }(\Gamma)$, and we can use Lemma \ref{lemma-compr} to determine  the positive self-adjoint operator  $\t\Theta_{D,\Sigma}$ in $\ran(\Pi_{\Sigma})$ associated to the restriction of $f_{D}$ to $\ran(\Pi_{\Sigma})$. Then $$\Theta_{D,\Sigma}:=
{\rm U}_{\Sigma}\t\Theta_{D,\Sigma}\Lambda^{-3}\,,\quad \dom(\Theta_{D,\Sigma}):=\Lambda^{3}\dom(\t\Theta_{D,\Sigma})\,,
$$
is self-adjoint, where the map ${\rm U}_{\Sigma}(\Pi_{\Sigma}\phi)=\phi|\Sigma$ provides the unitary isomorphism $H^{\frac32}_{\Sigma^{c} }(\Gamma)^{\perp}\simeq H^{\frac32}(\Sigma)$.   To conclude the proof we need to determine the operator  $\breve\Theta_{D}:=(\t\Theta_{D})\breve{\,}$ and the subspace $K_{\Sigma}:=\fk_{\Pi_{\Sigma}}$ (we refer to the Appendix for the notations).
Let $H_{D}$ be the Hilbert space given by $\dom(f_{D})=H^{5/2}(\Gamma)$ endowed with the 
scalar product  $\langle\phi_{1},\phi_{2}\rangle_{D}:=f_{D}(\phi_{1},\phi_{2})$; let $H_{D}'$ denote  its dual space.
Since
$$
f_{D}(\phi_{1},\phi_{2})=\langle \gamma_{0}\SL  \Lambda^{3}\phi_{1},\Lambda^{3}\phi_{2}\rangle_{\frac12,-\frac12}=
\langle \Lambda^{\frac12}\gamma_{0}\SL  \Lambda^{3}\phi_{1},\Lambda^{5/2}\phi_{2}
\rangle_{L^{2}(\Gamma)}\,,
$$
one has 
$$ H_D^{\prime}=H^{\frac12}(\Gamma)\,,\quad 
\langle \varphi,\phi\rangle_{H'_{D},H_{D}}=\langle \Lambda^{\frac12}\varphi,\Lambda^{5/2}\phi\rangle_{L^{2}(\Gamma)}
$$ 
and
$$
\breve\Theta _{D}:H^{5/2}(\Gamma)\to H^{\frac12}(\Gamma)\,,\quad \breve\Theta _{D}=\gamma_{0}\SL  \Lambda^{3}\,.
$$
Moreover
\begin{align*}
K_{\Sigma}=&\{\varphi\in H^{\frac12}(\Gamma):\forall\phi\in \Lambda^{-3}
H^{-\frac12}_{\overline\Sigma }(\Gamma)\,,\quad\langle \Lambda^{\frac12}\varphi,\Lambda^{5/2}\phi\rangle_{L^{2}(\Gamma)}=0\}\\
=&\{\varphi\in H^{\frac12}(\Gamma):\forall\phi\in H^{-\frac12}_{\overline\Sigma }(\Gamma)\,,\ \langle\Lambda^{-\frac12}\phi,\Lambda^{\frac12}\varphi\rangle_{L^{2}(\Gamma)}=0\}\\
=&\{\varphi\in H^{\frac12}(\Gamma):\forall\phi\in H^{-\frac12}_{\overline\Sigma }(\Gamma)\,,\ \langle\phi,\varphi\rangle_{-\frac12,\frac12}=0\}\\
=&
H^{\frac12}_{\Sigma^{c}}(\Gamma)\,.
\end{align*}
Therefore, by Lemma \ref{lemma-compr},  $\t\Theta_{D,\Sigma}$ is self-adjoint on the domain 
$$
\dom(\t\Theta_{D,\Sigma}):=\{\phi\in \Lambda^{-3}H^{-\frac12}_{\overline\Sigma }(\Gamma):\exists\t\phi\in\ran(\Pi_{\Sigma})\ \text{s.t.}\, \gamma_{0}\SL  \Lambda^{3}\phi-\t\phi\in H^{\frac12}_{\Sigma^{c}}(\Gamma)\}
$$
and $\t\Theta_{D,\Sigma}\phi=\t\phi$. Then ${\rm U}_{\Sigma}(\t\Theta_{D,\Sigma}\Lambda^{-3}\phi)=\t\phi|\Sigma=(\hat\gamma_{0}\SL\phi)|\Sigma$ and 
\begin{align*}
\dom(\Theta_{D,\Sigma})=&\{\phi\in H^{-\frac12}_{\overline\Sigma }(\Gamma):\exists\t\phi\in\ran(\Pi_{\Sigma})\ \text{s.t.}\, \gamma_{0}\SL \phi-\t\phi\in H^{\frac12}_{\Sigma^{c}}(\Gamma)\}\\
\subseteq &  \{\phi\in H^{-\frac12}_{\overline\Sigma }(\Gamma):(\gamma_{0}\SL \phi)|\Sigma\in H^{\frac32}(\Sigma)\}\\
\subseteq & \{\phi\in H^{-\frac12}_{\overline\Sigma }(\Gamma):\exists\t\phi\in\ran(\Pi_{\Sigma})\ \text{s.t.}\, (\gamma_{0}\SL \phi)|\Sigma=\t\phi|\Sigma\}\\
=&\{\phi\in H^{-\frac12}_{\overline\Sigma }(\Gamma):\exists\t\phi\in\ran(\Pi_{\Sigma})\ \text{s.t.}\, \gamma_{0}\SL \phi-\t\phi\in H^{\frac12}_{\Sigma^{c}}(\Gamma)\}=\dom(\Theta_{D,\Sigma})\,.
\end{align*}
\end{proof}
\begin{corollary} The linear operator in $L^{2}(\RE^{n})$ defined by $A_{D,\Sigma}:=(A_{{-}}^{\max}\oplus A_{{+}}^{\max})|\dom(A_{D,\Sigma})$, where 
\begin{align*}
\D(A_{D,\Sigma})
=&
\{u\in H^{1}(\RE^{n})\cap(\dom(A_{-}^{\max})\oplus\dom(A_{+}^{\max})): 
[\hat\gamma_{1}]u\in \dom(\Theta_{D,\Sigma})\,,\ (\gamma_{0}u)|\Sigma=0\}
\\
\subseteq &\{u\in  H^{1}(\RE^{n})\cap(\dom(A_{-}^{\max})\oplus\dom(A_{+}^{\max})): (\gamma^{\pm}_{0}u)|\Sigma=0\,,\ ([\hat\gamma_{1}]u)|\overline\Sigma^{c}=0\}
\,,
\end{align*}
is self-adjoint and its resolvent is given by 
\begin{align*}
(-A_{D,\Sigma}+z)^{-1}=&
(-A+z)^{-1}-\SL_{z}\Pi_{\Sigma}'(R_{\Sigma}\gamma_{0}\SL_{z}\Pi_{\Sigma}')^{-1}R_{\Sigma}\gamma_{0}(-A+z)^{-1}
\,,
\end{align*}
where $R_{\Sigma}$ is the restriction operator $R_{\Sigma}\phi=\phi|\Sigma$ and  $\Pi_{\Sigma}'$ acts there as the inclusion map $\Pi_{\Sigma}':H^{-\frac32}_{\overline\Sigma}(\Gamma)\to H^{-\frac32}(\Gamma)$. 
\end{corollary}
\begin{proof}
By Theorem \ref{compr} and Theorem \ref{T1},  taking $\Pi(\phi\oplus\varphi)=\Pi_{\Sigma}\phi\oplus 0$ and $\Theta(\phi\oplus\varphi)=(-{\rm U}_{\Sigma}^{-1}\Theta_{D,\Sigma}\phi)\oplus 0$ one gets the self-adjoint extension $(A_{{-}}^{\max}\oplus A_{{+}}^{\max})|\dom(A_{D,\Sigma})$ with domain (contained in $H^{1}(\RE^{n}\backslash\Gamma)$  by $\dom(\Theta_{D,\Sigma})\subseteq H^{-\frac12}(\Gamma)$ and Remark \ref{regularity}),
\begin{align*}
&
\D(A_{D,\Sigma})\\
=&\{u\in H^{1}(\RE^{n}\backslash\Gamma)\cap(\dom(A_{-}^{\max})\oplus\dom(A_{+}^{\max})): 
[\gamma_{0}]u=0\,,\ 
[\hat\gamma_{1}]u\in \dom(\Theta_{D,\Sigma})\,,\ \Pi_{\Sigma}\gamma_{0}(u+\SL[\hat\gamma_{1}]u)={\rm U}_{\Sigma}^{-1}\Theta_{D,\Sigma}[\hat\gamma_{1}]u
\}
\\
=&\{u\in H^{1}(\RE^{n})\cap(\dom(A_{-}^{\max})\oplus\dom(A_{+}^{\max})): 
[\hat\gamma_{1}]u\in \dom(\Theta_{D,\Sigma})\,,\ 
(\gamma_{0}u)|\Sigma+(\gamma_{0}\SL[\hat\gamma_{1}]u)|\Sigma=
(\gamma_{0}\SL[\hat\gamma_{1}]u)|\Sigma\}
\\
=&\{u\in H^{1}(\RE^{n})\cap(\dom(A_{-}^{\max})\oplus\dom(A_{+}^{\max})): 
[\hat\gamma_{1}]u\in \dom(\Theta_{D,\Sigma})\,,\ (\gamma_{0}u)|\Sigma=0\}
\,.
\end{align*}
The formula giving $(-A_{D,\Sigma}+z)^{-1}$ is consequence of \eqref{krein1}, since 
$$(-{\rm U}_{\Sigma}^{-1}\Theta_{D,\Sigma}+\Pi_{\Sigma}\gamma_{0}(\SL-\SL_{z})\Pi_{\Sigma}')^{-1}\Pi_{\Sigma}=
(-\Theta_{D,\Sigma}+R_{\Sigma}\gamma_{0}(\SL-\SL_{z})\Pi_{\Sigma}')^{-1}R_{\Sigma}
$$ 
and 
$$
-\Theta_{D,\Sigma}\phi+R_{\Sigma}\gamma_{0}(\SL-\SL_{z})\phi=-\gamma_{0}\SL\phi|\Sigma
+\gamma_{0}\SL\phi|\Sigma-\gamma_{0}\SL_{z}\phi|\Sigma=-\gamma_{0}\SL_{z}\phi|\Sigma\,,
$$
\end{proof}

\begin{remark} Since $\supp([\gamma_{1}]u)\subseteq\overline\Sigma$ for any $u\in\dom(A_{D,\Sigma})$, one has 
$[\hat\gamma_{1}]u\,\delta_{\Gamma}=[\hat\gamma_{1}]u\,\delta_{\overline\Sigma}$; thus
$$A_{D,\Sigma}u=Au-[\hat\gamma_{1}]u\,\delta_{\overline\Sigma}$$ and so 
$(A_{D,\Sigma}u)|{\overline\Sigma}^{c}=(Au)|\overline\Sigma^{c}$. This also shows that $A_{D,\Sigma}$ is a self-adjoint extension of the symmetric operator $A|\C^{\infty}_{\comp}(\RE^{n}\backslash\overline\Sigma)$. Hence it
depends only on $\Sigma$ and not on the whole $\Gamma$: one would obtain the same operator by considering any other bounded domain $\Omega_{\circ}$ with boundary $\Gamma_{\!\circ }$ such that $\Sigma\subset\Gamma_{\!\circ}$. 
\end{remark}
\begin{remark}
According to the definition of  $\dom(\Theta_{D,\Sigma})$, the Lemma \ref{regularity2} applies and one gets 
$$
\dom(A_{D,\Sigma})\subseteq H^{2}(\RE^{n}\backslash(\overline\Sigma\cup(\partial\Sigma)^{\epsilon}))\,.
$$
\end{remark}
\begin{remark} By the results contained in the proof of Theorem \ref{compr}, for the domain of the sesquilinear form $f_{D,\Sigma}$ associated to the self-adjoint operator $\t\Theta_{D,\Sigma}$ one has the relation $\dom(f_{D,\Sigma})\subseteq H^{5/2}(\Gamma)$. Therefore Theorem \ref{teo-schatten2} applies and so, by Corollary \ref{spectrum},  $\sigma_{ac}(A_{D,\Sigma})=\sigma_{ac}(A)$ and the wave operators $W_{\pm}(A_{D,\Sigma},A)$, $W_{\pm}(A,A_{D,\Sigma})$ exist and are complete. 
\end{remark}
\begin{remark} In case $n=3$ and both $\Gamma$ and $\Sigma$ are smooth, by \cite[Theorem 2.4]{CS}, one has that 
\be\label{CS1}
(\hat\gamma_{0}\SL\phi)|\Sigma\in H^{s+1}(\Sigma)\iff \phi\in H^{s}_{\overline\Sigma}(\Gamma)\,,\qquad-1< s<0\,.
\ee 
Therefore $\dom(\Theta_{D,\Sigma})\subseteq H^{s}(\Gamma)$, $-\frac 12\le s<0$, and so, by Remark \ref{regularity} and \eqref{H2}, 
$$
\D(A_{D,\Sigma})\subseteq H^{\frac32-}(\RE^{3})\,.
$$
\end{remark}
\end{subsection}  
\begin{subsection}{Neumann boundary conditions.}\label{neu-loc}
Given $\Sigma\subset \Gamma$ relatively open of class $\C^{0,1}$,  here we denote by $\Pi_{\Sigma}$ the orthogonal projector in the Hilbert space $H^{\frac12}(\Gamma)$ such that $\ran(\Pi_{\Sigma})=H^{\frac12}_{\Sigma^{c}}(\Gamma)^{\perp}$. By \eqref{perp}, $\ran(\Pi_{\Sigma}')=H^{-\frac12}_{\overline\Sigma}(\Gamma)$, where $\Pi'_{\Sigma}=\Lambda\Pi_{\Sigma}\Lambda^{-1}$ is the dual projection. In the following, we also use the identifications $H^{\frac12}_{\Sigma^{c}}(\Gamma)^{\perp}\simeq H^{\frac12}(\Sigma)$ and $H^{\frac12}(\Sigma)'\simeq H^{-\frac12}_{\overline\Sigma}(\Gamma)$. We denote the orthogonal projection from  $H^{\frac12}(\Gamma)$ onto $H^{\frac12}(\Sigma)$ by $R_{\Sigma}$; it can be identified with the restriction map $R_{\Sigma}:H^{s}(\Gamma)\to H^{s}(\Sigma)$, $R_{\Sigma}\phi:=\phi|\Sigma$ (see the end of Subsection \ref{sobolev}).\par 
Similarly to Example \ref{dir-loc} one has the following
\begin{theorem}\label{comprN} Let $\Omega $ be of class $\C^{2,1}$, then 
$$\Theta_{N,\Sigma}:\dom(\Theta_{N,\Sigma})\subseteq H^{-\frac12}_{\overline\Sigma}(\Gamma)\to H^{\frac12}(\Sigma)\,,\quad \Theta_{N,\Sigma}:=R_{\Sigma}\hat\gamma_{1}\DL\Pi'_{\Sigma}\varphi\equiv (\hat\gamma_{1}\DL\varphi)|\Sigma\,,  
$$
$$
\dom(\Theta_{N,\Sigma}):=\{\varphi\in H^{\frac12}_{\overline\Sigma}(\Gamma):(\hat\gamma_{1}\DL\varphi)|\Sigma\in H^{\frac12}(\Sigma)\}\,,
$$
is self-adjoint.
\end{theorem}
\begin{proof} Let $f_{N}$ be the densely defined sesquilinear form in the Hilbert space $H^{\frac12}(\Gamma)$ 
$$f_{N}: H^{\frac32}(\Gamma)\times H^{\frac32}(\Gamma)\subseteq H^{\frac12}(\Gamma)\times H^{\frac12}(\Gamma)\to \RE$$  
$$
f_{N}(\varphi_{1},\varphi_{2}):=
\langle \hat\gamma_{1}\DL  \Lambda\varphi_{1},\Lambda\varphi_{2}
\rangle_{-\frac12,\frac12}\,.
$$ 
By \eqref{coercive2},
$$ 
-f_{N}(\varphi,\varphi)\ge c_{1} \,\|\Lambda\varphi\|^{2}_{\frac12}=c_{1}\,\|\varphi\|^{2}_{H^{\frac32}(\Gamma)}
$$
and so $f_{N}$ is strictly negative and closed. Since, for any $\varphi_{1}\in H^{5/2}(\Gamma)$ and for any $\phi_{2}\in H^{\frac32}(\Gamma)$,
$$ 
f_{N}(\varphi_{1},\varphi_{2})=\langle \Lambda^{\frac12}\hat\gamma_{1}\DL  \Lambda\varphi_{1}
,\Lambda^{\frac12}\varphi_{2},\rangle_{L^{2}(\Gamma)}=\langle \hat\gamma_{1}\DL  \Lambda\varphi_1,\varphi_{2}\rangle_{H^{\frac12}(\Gamma)}\,,
$$
$f_{N}$ is the sesquilinear form associated with the self-adjoint operator 
$$
\t\Theta_{N}:=\hat\gamma_{1}\DL  \Lambda:H^{5/2}(\Gamma)\subset H^{\frac12}(\Gamma)\to H^{\frac12}(\Gamma)\,.
$$
Since  
$$
\dom(f_{N})\cap\ran(\Pi_{\Sigma})=
H^{\frac32}(\Gamma)\cap H^{\frac12}_{\Sigma^{c}}(\Gamma)^{\perp}=H^{\frac32}(\Gamma)\cap\Lambda^{-1}H^{-\frac12}_{\overline\Sigma}(\Gamma)=\Lambda^{-1}
H^{\frac12}_{\overline\Sigma }(\Gamma)
$$ 
is dense in $\Lambda^{-1} H^{-\frac12}_{\overline\Sigma }(\Gamma)$, we can use  Lemma \ref{lemma-compr} to determine the positive self-adjoint operator $\t\Theta_{N,\Sigma}$ in $\ran(\Pi_{\Sigma})$ associated to the restriction of $f_{N}$ to $\ran(\Pi_{\Sigma}) $. Then
$$
\Theta_{N,\Sigma}:=\text{\rm U}_{\Sigma}\t\Theta_{N,\Sigma}\Lambda^{-1}\,,\qquad \dom(\Theta_{N,\Sigma}):=\Lambda\dom(\t\Theta_{N,\Sigma})\,,
$$
is self-adjoint, where the map $\text{\rm U}_{\Sigma}(\Pi_{\Sigma})\varphi=\varphi|\Sigma$ provides the unitary isomorphism $H^{\frac12}_{\Sigma^{c}}(\Gamma)^{\perp}\simeq H^{\frac12}(\Sigma)$. To conclude the proof we need to determine the operator $\breve\Theta_{N}:=(\t\Theta_{N})\breve{\,}$ and the subspace $K_{\Sigma}:=\fk_{\Pi_{\Sigma}}$ (we refer to the Appendix for the notations). 
Let $H_{N}$ be the Hilbert space given by $\dom(f_{N})=H^{\frac32}(\Gamma)$ endowed with the 
scalar product  $\langle\varphi_{1},\varphi_{2}\rangle_{N}:=-f_{N}(\varphi_{1},\varphi_{2})$; let $H_{N}'$ denote its dual space.
Since
$$
f_{N}(\varphi_{1},\varphi_{2})=\langle \hat\gamma_{1}\DL  \Lambda\varphi_{1},\Lambda\varphi_{2}\rangle_{-\frac12,\frac12}=
\langle \Lambda^{-\frac12}\hat\gamma_{1}\DL  \Lambda\varphi_{1},\Lambda^{\frac32}\varphi_{2}
\rangle_{L^{2}(\Gamma)}\,,
$$
one has 
$$H_N^{\prime}=H^{-\frac12}(\Gamma)\,,\quad 
\langle \phi,\varphi\rangle_{H'_{N},H_{N}}=\langle \Lambda^{-\frac12}\phi,\Lambda^{\frac32}\varphi\rangle_{L^{2}(\Gamma)}
$$ 
and
$$
\breve\Theta _{N}: H^{\frac32}(\Gamma)\to H^{-\frac12}(\Gamma)\,,\quad \breve\Theta _{N}=\hat\gamma_{1}\DL  \Lambda\,. 
$$
Moreover
\begin{align*}
K_{\Sigma}=&\{\phi\in H^{-\frac12}(\Gamma):\forall\varphi\in \Lambda^{-1}
H^{\frac12}_{\overline\Sigma }(\Gamma)\,,\quad\langle \Lambda^{-\frac12}\phi,\Lambda^{\frac32}\varphi\rangle_{L^{2}(\Gamma)}=0\}\\
=&\{\phi\in H^{-\frac12}(\Gamma):\forall\varphi\in H^{\frac12}_{\overline\Sigma }(\Gamma)\,,\ \langle\Lambda^{-\frac12}\phi,\Lambda^{\frac12}\varphi\rangle_{L^{2}(\Gamma)}=0\}\\
=&\{\phi\in H^{-\frac12}(\Gamma):\forall\varphi\in H^{\frac12}_{\overline\Sigma }(\Gamma)\,,\ \langle\phi,\varphi\rangle_{-\frac12,\frac12}=0\}\\
=&
H^{-\frac12}_{\Sigma^{c}}(\Gamma)\,.
\end{align*}
Therefore, by Lemma 5.1, $\t\Theta_{N,\Sigma}$ is self-adjoint on the domain 
$$
\dom(\t\Theta_{N,\Sigma}):=\{\varphi\in \Lambda^{-1}H^{\frac12}_{\overline\Sigma }(\Gamma):\exists\t\varphi\in\ran(\Pi_{\Sigma})\ \text{s.t.}\, \hat\gamma_{1}\DL  \Lambda\varphi-\t\varphi\in H^{-\frac12}_{\Sigma^{c}}(\Gamma)\}
$$
and $\t\Theta_{N,\Sigma}\varphi:=\t\varphi$. Then $\U_{\Sigma}\t\Theta_{N,\Sigma}\Lambda^{-1}\varphi=\t\varphi|\Sigma=(\hat\gamma_{1}\DL \varphi)|\Sigma$ and
\begin{align*}
\dom(\t\Theta_{N,\Sigma})=&\{\varphi\in H^{\frac12}_{\overline\Sigma }(\Gamma):\exists\t\varphi\in\ran(\Pi_{\Sigma})\ \text{s.t.}\, \hat\gamma_{1}\DL  \varphi-\t\varphi\in H^{-\frac12}_{\Sigma^{c}}(\Gamma)\}\\
\subseteq&\{\varphi\in H^{\frac12}_{\overline\Sigma }(\Gamma):(\hat\gamma_{1}\DL  \varphi)|\Sigma\in H^{\frac12}(\Sigma)\}\\
\subseteq&\{\varphi\in H^{\frac12}_{\overline\Sigma }(\Gamma):\exists\t\varphi\in\ran(\Pi_{\Sigma})\ \text{s.t.}\, (\hat\gamma_{1}\DL  \varphi)|\Sigma=\t\varphi|\Sigma\}\\
=&\{\varphi\in H^{\frac12}_{\overline\Sigma }(\Gamma):\exists\t\varphi\in\ran(\Pi_{\Sigma})\ \text{s.t.}\, \hat\gamma_{1}\DL  \varphi-\t\varphi\in H^{-\frac12}_{\Sigma^{c}}(\Gamma)\}=
\dom(\t\Theta_{N,\Sigma})
\end{align*}
\end{proof}
\begin{corollary} The linear operator in $L^{2}(\RE^{n})$ defined by $A_{N,\Sigma}:=(A_{{-}}^{\max}\oplus A_{{+}}^{\max})|\dom(A_{N,\Sigma})$, where 
\begin{align*}
\D(A_{N,\Sigma})
=&\{u\in H^{1}(\RE^{n}\backslash\overline\Sigma)\cap(\dom(A^{\max}_{-})\oplus\dom(A^{\max}_{+})): [\gamma_{0}]u\in \dom(\Theta_{N,\Sigma})\,,\  
[\hat\gamma_{1}]u=0\,,\  (\hat\gamma_{1}u)|\Sigma=0\}
\\
\subseteq &
\{u\in H^{1}(\RE^{n}\backslash\overline\Sigma)\cap(\dom(A^{\max}_{-})\oplus\dom(A^{\max}_{+})):  (\hat\gamma^{\pm}_{1}u)|\Sigma=0\,,\ ([\hat\gamma_{1}]u)|\overline\Sigma^{c}=0\}
\,,
\end{align*}
is self-adjoint and its resolvent is given by 
\begin{align*}
(-A_{N,\Sigma}+z)^{-1}=
(-A+z)^{-1}-\DL_{z}\Pi_{\Sigma}'(R_{\Sigma}\hat\gamma_{1}\DL_{z}\Pi_{\Sigma}')^{-1}R_{\Sigma}\gamma_{1}(-A+z)^{-1}
\,,
\end{align*}
where $R_{\Sigma}$ is the restriction operator $R_{\Sigma}\varphi=\varphi|\Sigma$ and $\Pi_{\Sigma}'$ acts there as the inclusion map $\Pi_{\Sigma}':H^{-\frac12}_{\overline\Sigma}(\Gamma)\to H^{-\frac12}(\Gamma)$. 
\end{corollary}
\begin{proof}
By Theorem \ref{compr} and Theorem \ref{T1}, taking $\Pi(\phi\oplus\varphi)=0\oplus\Pi_{\Sigma}\varphi$ and $\Theta(\phi\oplus\varphi)=0\oplus (-\U^{-1}_{\Sigma}\Theta_{N,\Sigma}\varphi)$, one gets the self-adjoint extension $(A_{{-}}^{\max}\oplus A_{{+}}^{\max})|\dom(A_{N,\Sigma})$ with domain (contained in $H^{1}(\RE^{n}\backslash\Gamma)$ by $\dom(\Theta_{N,\Sigma})\subseteq H^{\frac12}(\Gamma)$ and Remark \ref{regularity}), 
\begin{align*}
&\dom(A_{N,\Sigma})\\
=&\{u\in H^{1}(\RE^{n}\backslash\Gamma)\cap(\dom(A^{\max}_{-})\oplus\dom(A^{\max}_{+})): [\gamma_{0}]u\in \dom(\Theta_{N,\Sigma})\,,\ 
[\hat\gamma_{1}]u=0\,,\  \Pi_{\Sigma}\gamma_{1}(u-\DL[\gamma_{0}u])=-\U^{-1}_{\Sigma}\Theta_{N,\Sigma}[\gamma_{0}]u\}
\\
=&\{u\in H^{1}(\RE^{n}\backslash\overline\Sigma)\cap(\dom(A^{\max}_{-})\oplus\dom(A^{\max}_{+})): [\gamma_{0}]u\in \dom(\Theta_{N,\Sigma})\,,\ 
[\hat\gamma_{1}]u=0\,,\  (\hat\gamma_{1}u)|\Sigma-(\hat\gamma_{1}\DL[\gamma_{0}u])|\Sigma=-(\hat\gamma_{1}\DL[\gamma_{0}u])|\Sigma\}\\
=&\{u\in H^{1}(\RE^{n}\backslash\overline\Sigma)\cap(\dom(A^{\max}_{-})\oplus\dom(A^{\max}_{+})): [\gamma_{0}]u\in \dom(\Theta_{N,\Sigma})\,,\  
[\hat\gamma_{1}]u=0\,,\  (\hat\gamma_{1}u)|\Sigma=0\}
\,.
\end{align*}
The formula giving $(-A_{N,\Sigma} + z)^{-1}$ is consequence of \eqref{krein1}, since
$$
(-\U_{\Sigma}^{-1}\Theta_{N,\Sigma}+\Pi_{\Sigma}\gamma_{1}(\DL-\DL_{z})\Pi_{\Sigma}')^{-1}\Pi_{\Sigma}=
(-\Theta_{N,\Sigma}+R_{\Sigma}\gamma_{1}(\DL-\DL_{z})\Pi_{\Sigma}')^{-1}R_{\Sigma}
$$
and
$$
-\Theta_{N,\Sigma}\varphi+R_{\Sigma}\gamma_{1}(\DL-\DL_{z})\varphi=
-(\hat\gamma_{1}\DL\varphi)|\Sigma+
(\hat\gamma_{1}\DL\varphi)|\Sigma-(\hat\gamma_{1}\DL_{z}\varphi)|\Sigma
=-(\hat\gamma_{1}\DL_{z}\varphi)|\Sigma\,.
$$
\end{proof}
\begin{remark}
Since $\supp([\gamma_{0}]u)\subseteq\overline\Sigma$ for any $u\in\dom(A_{N,\Sigma})$, one has 
$[\gamma_{0}]u\,\partial_{\underline a}\delta_{\Gamma}=
[\gamma_{0}]u\,\partial_{\underline a}\delta_{\overline\Sigma}$; thus
$$A_{N,\Sigma}u=Au-[\gamma_{0}]u\,\partial_{\underline a}\delta_{\overline\Sigma}$$ and so 
$(A_{N,\Sigma}u)|\overline\Sigma^{c}=(Au)|\overline\Sigma^{c}$. This also shows that $A_{N,\Sigma}$ is a self-adjoint extension of the symmetric operator $A|\C^{\infty}_{\comp}(\RE^{n}\backslash\overline\Sigma)$. Hence it depends only on $\Sigma$ and not on the whole $\Gamma$: one would obtain the same operator by considering any other bounded domain $\Omega_{\circ}$ with boundary $\Gamma_{\!\circ }$ such that $\Sigma\subset\Gamma_{\!\circ}$. 
\end{remark}
\begin{remark}
According to the definition of  $\dom(\Theta_{N,\Sigma})$, the Lemma \ref{regularity2} applies and one gets 
$$
\dom(A_{N,\Sigma})\subseteq H^{2}(\RE^{n}\backslash(\overline\Sigma\cup(\partial\Sigma)^{\epsilon}))\,.
$$
\end{remark}
\begin{remark} By the results contained in the proof of Theorem \ref{comprN}, for the domain of the sesquilinear form $f_{N,\Sigma}$ associated to the self-adjoint operator $\t\Theta_{N,\Sigma}$ one has the relation $\dom(f_{N,\Sigma})\subseteq H^{\frac32}(\Gamma)$. Therefore Theorem \ref{teo-schatten2} applies and so, by Corollary \ref{spectrum},  $\sigma_{ac}(A_{N,\Sigma})=\sigma_{ac}(A)$ and the wave operators $W_{\pm}(A_{N,\Sigma},A)$, $W_{\pm}(A,A_{N,\Sigma})$ exist and are complete. 
\end{remark}
\begin{remark} In case $n=3$ and both $\Gamma$ and $\Sigma$ are smooth, by \cite[Theorem 2.4]{CS}, one has that 
\be\label{CS2}
(\hat\gamma_{1}\DL\phi)|\Sigma\in H^{s-1}(\Sigma)\iff \phi\in H^{s}_{\overline\Sigma}(\Gamma)\,,\qquad 0< s<1\,,
\ee 
Therefore $\dom(\Theta_{N,\Sigma})\subseteq H^{s}(\Gamma)$, $\frac 12\le s<1$, and so, by Remark \ref{regularity} and \eqref{H4}, 
$$
\D(A_{N,\Sigma})\subseteq H^{\frac32-}(\RE^{3}\backslash\overline\Sigma)\,.
$$
\end{remark}
\end{subsection}  
\begin{subsection}{Robin boundary conditions.}\label{robin-loc}  
Given $\Sigma\subset \Gamma$ relatively open of class $\C^{0,1}$,  here we denote by $\Pi_{\Sigma}^{\oplus}$ the orthogonal projector in the Hilbert space $H^{\frac32}(\Gamma)\oplus H^{\frac12}(\Gamma)$ such that $\ran(\Pi^{\oplus}_{\Sigma})=H^{\frac32}_{\Sigma^{c}}(\Gamma)^{\perp}\oplus H^{\frac12}_{\Sigma^{c}}(\Gamma)^{\perp}$. By \eqref{perp}, $\ran((\Pi^{\oplus}_{\Sigma})')=\Lambda^{-3}H^{-\frac32}_{\overline\Sigma}(\Gamma)\oplus \Lambda^{-1}H^{-\frac12}_{\overline\Sigma}(\Gamma)$, where $(\Pi^{\oplus}_{\Sigma})'=(\Lambda^{3}\oplus\Lambda)\Pi^{\oplus}_{\Sigma}(\Lambda^{-3}\oplus\Lambda^{-1})$ is the dual projection. In the following, we also use the identifications $H^{\frac32}_{\Sigma^{c}}(\Gamma)^{\perp}\oplus H^{\frac12}_{\Sigma^{c}}(\Gamma)^{\perp}\simeq H^{\frac32}(\Sigma)\oplus H^{\frac12}(\Sigma)$ and $(H^{\frac32}(\Sigma)\oplus H^{\frac12}(\Sigma))'\simeq H^{-\frac32}_{\overline\Sigma}(\Gamma)\oplus  H^{-\frac12}_{\overline\Sigma}(\Gamma)$. We denote the orthogonal projection from  $H^{\frac32}(\Gamma)\oplus H^{\frac12}(\Gamma)$ onto $H^{\frac32}(\Sigma)\oplus H^{\frac12}(\Sigma)$  by $R^{\oplus}_{\Sigma}$; it can be identified with the restriction map $R^{\oplus}_{\Sigma}:H^{s_{1}}(\Gamma)\oplus H^{s_{2}}(\Gamma)\to 
H^{s_{1}}(\Sigma)\oplus H^{s_{2}}(\Sigma)$, $R^{\oplus}_{\Sigma}(\phi\oplus\varphi):=(\phi|\Sigma)\oplus(\varphi|\Sigma)$ (see the end of Subsection \ref{sobolev}).
\begin{theorem}\label{comprR} Let $\Omega$ be of class $\C^{4,1}$, and let $b_{\pm}\in M^{\frac32}(\Gamma)$ be real valued and $1/[b]\in L^{\infty}(\Gamma)$ be negative. Then
$$
\Theta_{R,\Sigma}:
\dom(\Theta_{R,\Sigma})\subseteq 
H^{-\frac32}_{\overline\Sigma}(\Gamma)\oplus H^{-\frac12}_{\overline\Sigma}(\Gamma)\to H^{\frac32}(\Sigma)\oplus H^{\frac12}(\Sigma)\,,
$$
\begin{align*}
&\Theta_{R,\Sigma}(\phi,\varphi)
:=R^{\oplus}_{\Sigma}\left[\begin{matrix}1/[b]+\gamma_{0}\SL &\langle b\rangle/[b]+\hat\gamma_{0}\DL\\
\langle b\rangle/[b]+\hat\gamma_{1}\SL &b_{+}b_{-}/[b]+\hat\gamma_{1}\DL\end{matrix}\right] (\Pi^{\oplus}_{\Sigma})'(\phi,\varphi)\\
\equiv&
(((1/[b]+\gamma_{0}\SL)\phi+(\langle b\rangle/[b]+\hat\gamma_{0}\DL)\varphi)|\Sigma)\oplus(((\langle b\rangle/[b]+\hat\gamma_{1}\SL)\phi+(b_{+}b_{-}/[b]+\hat\gamma_{1}\DL)\varphi)|\Sigma)
\,,
\end{align*}
\begin{align*}
\dom(\Theta_{R,\Sigma}):=&\big\{(\phi,\varphi)\in L^{2}_{\overline\Sigma }(\Gamma)\times H^{\frac12}_{\overline\Sigma }(\Gamma):
((1/[b]+\gamma_{0}\SL)\phi+(\langle b\rangle/[b]+\hat\gamma_{0}\DL)\varphi)|\Sigma\in H^{\frac32}(\Sigma)\,,\\
&((\langle b\rangle/[b]+\hat\gamma_{1}\SL)\phi
+(b_{+}b_{-}/[b]+\hat\gamma_{1}\DL)\varphi)|\Sigma\in H^{\frac12}(\Sigma)\big\}
\end{align*}
is self-adjoint.
\end{theorem}
\begin{proof} Let $\t\Theta_{R}$ be the self-adjoint operator in $H^{\frac32}(\Gamma)\oplus H^{\frac12}(\Gamma)$  
$$\t\Theta_{R}:=\Theta_{R}(\Lambda^{3}\oplus\Lambda):H^{9/2}(\Gamma)\times H^{5/2}(\Gamma)\subseteq H^{\frac32}(\Gamma)\oplus H^{\frac12}(\Gamma)\to H^{\frac32}(\Gamma)\oplus H^{\frac12}(\Gamma)\,.
$$
Since $\gamma_{0}\SL\,\Lambda^{3}$ is infinitesimally $(1/[b])\Lambda^3$-bounded (proceed as in  the proof of Lemma \ref{lemmadelta}) and $b_{+}b_{-}/[b]\Lambda$ is infinitesimally $(\hat\gamma_{1}\SL\,\Lambda)$-bounded (proceed as in  the proof of Lemma \ref{lemmadeltaprimo}), by $[b]<0$ and \eqref{coercive2} respectively, both $(1/[b]+\gamma_{0}\SL)\Lambda^{3}$ and $(b_{+}b_{-}/[b]+\hat\gamma_{1}\DL)\Lambda$ are upper bounded. Let $\lambda_{*}$ be a common strict upper bound; by the Frobenius-Schur factorization of the block operator matrix $\t\Theta_{R}$ (see e.g. \cite[Theorem 2.2.18]{Tre}),  there exists $\lambda_{R}>\lambda_{*}$ such that 
$-\t\Theta_{R}+\lambda_{R}>0$ whenever 
\be\label{FS}
-(1/[b]+\gamma_{0}\SL)\Lambda^{3}+\lambda_{R}> (\langle b\rangle/[b]+\hat\gamma_{0}\DL)\Lambda
(-(b_{+}b_{-}/[b]+\hat\gamma_{1}\DL)\Lambda+\lambda_{R})^{-1}
(\langle b\rangle/[b]+\hat\gamma_{1}\SL)\Lambda^{3}\,.
\ee
Since 
$$
\langle (-(1/[b]+\gamma_{0}\SL)\Lambda^{3}+\lambda_{R})\phi,\phi\rangle_{H^{\frac32}(\Gamma)}\ge   (\lambda_{R}-\lambda_{*})\,\|\phi\|^{2}_{H^{\frac32}(\Gamma)}
$$
and
\begin{align*}
&\langle(\langle b\rangle/[b]+\hat\gamma_{0}\DL)\Lambda
(-(b_{+}b_{-}/[b]+\hat\gamma_{1}\DL)\Lambda+\lambda_{R})^{-1}
(\langle b\rangle/[b]+\hat\gamma_{1}\SL)\Lambda^{3}\phi,\phi\rangle_{H^{\frac32}(\Gamma)}
\\
=&\langle
(-(b_{+}b_{-}/[b]+\hat\gamma_{1}\DL)\Lambda+\lambda_{R})^{-1}
(\langle b\rangle/[b]+\hat\gamma_{1}\SL)\Lambda^{3}\phi,(\langle b\rangle/[b]+\hat\gamma_{1}\SL)\Lambda^{3}\phi\rangle_{H^{\frac12}(\Gamma)}
\\
\le &\|
(-(b_{+}b_{-}/[b]+\hat\gamma_{1}\DL)\Lambda+\lambda_{*})^{-1}\|_{H^{-\frac12}(\Gamma),H^{\frac32}(\Gamma)}
\|(\langle b\rangle/[b]+\hat\gamma_{1}\SL)\Lambda^{\frac12}\|^{2}_{L^{2}(\Gamma),H^{-\frac12}(\Gamma)}
\|\phi\|^{2}_{H^{\frac32}(\Gamma)}\,,
\end{align*}
the operator inequality \eqref{FS} holds true by taking $\lambda_{R}$ sufficiently large. Let $f_{R}$ be the densely defined, semibounded, closed sesquilinear form associated with $\t\Theta_{R}$, 
i.e.
$$
f_{R}:(H^{3}(\Gamma)\times H^{\frac32}(\Gamma))\times (H^{3}(\Gamma)\times H^{\frac32}(\Gamma))
\subset
(H^{\frac32}(\Gamma)\oplus H^{\frac12}(\Gamma))\times (H^{\frac32}(\Gamma)\oplus H^{\frac12}(\Gamma))\to\RE\,,
$$
\begin{align*}
f_{R}((\phi_{1},\varphi_{1}),(\phi_{2},\varphi_{2})):=&
\langle(1/[b]+\hat\gamma_{0}\SL)\Lambda^{3}\phi_{1}+(\langle b\rangle/[b]+\hat\gamma_{0}\DL)\Lambda\varphi_{1},\Lambda^{3}\phi_{2}\rangle_{L^{2}(\Gamma)}\\
+&\langle(\langle b\rangle/[b]+\hat\gamma_{1}\SL)\Lambda^{3}\phi_{1}+(b_{+}b_{-}/[b]+\hat\gamma_{1}\DL)\Lambda\varphi_{1},\Lambda\varphi_{2}\rangle_{-\frac12,\frac12}.
\end{align*}
Since 
$H^{3}(\Gamma)\cap H^{\frac32}_{\Sigma^{c}}(\Gamma)^{\perp}=H^{3}(\Gamma)\cap\Lambda^{-3}H^{-\frac32}_{\overline\Sigma}(\Gamma)=\Lambda^{-3}
L^{2}_{\overline\Sigma }(\Gamma)$ is dense in $\Lambda^{-3} H^{-\frac32}_{\overline\Sigma }(\Gamma)$ 
and
$H^{\frac32}(\Gamma)\cap H^{\frac12}_{\Sigma^{c}}(\Gamma)^{\perp}=H^{\frac32}(\Gamma)\cap\Lambda^{-1}H^{-\frac12}_{\overline\Sigma}(\Gamma)=\Lambda^{-1}
H^{\frac12}_{\overline\Sigma }(\Gamma)$ is dense in $\Lambda^{-1} H^{-\frac12}_{\overline\Sigma }(\Gamma)$, we can use  Lemma \ref{lemma-compr} to determine the semibounded   self-adjoint operator $\t\Theta_{R,\Sigma}$ in $\ran(\Pi_{\Sigma}^{\oplus})$ associated to the restriction of $f_{R}$ to $\ran(\Pi^{\oplus}_{\Sigma})$. Then 
$$
\Theta_{R,\Sigma}:=\U_{\Sigma}^{\oplus}\t\Theta_{R,\Sigma}(\Lambda^{-3}\oplus\Lambda^{-1})\,,\quad \dom(\Theta_{R,\Sigma}):=(\Lambda^{3}\oplus\Lambda)\dom(\t\Theta_{R,\Sigma})\,,
$$
is self-adjoint, where the map $\U^{\oplus}_{\Sigma}\Pi^{\oplus}_{\Sigma}(\phi\oplus\varphi)=(\phi|\Sigma)\oplus(\varphi|\Sigma)$ provides the unitary isomorphism $(H^{\frac32}_{\Sigma^{c}}(\Gamma)^{\perp}\oplus H^{\frac12}_{\Sigma^{c}}(\Gamma)^{\perp})\simeq (H^{\frac32}(\Sigma)\oplus H^{\frac12}(\Sigma))$. To conclude the proof  we need to determine the operator $\breve\Theta_{R}:=(\t\Theta_{R})\breve{\,}$ and the subspace $K_{\Sigma}:=\fk_{\Pi^{\oplus}_{\Sigma}}$ (we refer to the Appendix for the notations). Let $H_{R}$ be the Hilbert space given by $\dom(f_{R})=H^{3}(\Gamma)\times H^{\frac32}(\Gamma)$ endowed with the 
scalar product  $$\langle(\phi_{1},\varphi_{1}),(\phi_{2},\varphi_{2})\rangle_{R}:=(-f_{R}+\lambda_{R})((\phi_{1},\varphi_{1}),(\phi_{2},\varphi_{2}))
\,.
$$ 
Let $H_{R}'$ denote its dual space.
Since
\begin{align*}
&f_{R}((\phi_{1},\varphi_{1}),(\phi_{2},\varphi_{2}))=
\langle(1/[b]+\hat\gamma_{0}\SL)\Lambda^{3}\phi_{1}
+(\langle b\rangle/[b]+\hat\gamma_{0}\DL)\Lambda\varphi_{1},\Lambda^{3}\phi_{2}\rangle_{L^{2}(\Gamma)}
\\&+\langle\Lambda^{-\frac12}((\langle b\rangle/[b]+\hat\gamma_{1}\SL)\Lambda^{3}\phi_{1}
+(b_{+}b_{-}/[b]+\hat\gamma_{1}\DL))\Lambda\varphi_{1},\Lambda^{\frac32}\varphi_{2}\rangle_{L^{2}(\Gamma)}\,,
\end{align*}
one has 
$$H_{R}'=L^{2}(\Gamma)\times H^{-\frac12}(\Gamma)\,,\quad \langle (\phi',\varphi'),(\phi,\varphi)\rangle_{H'_{R},H_{R}}=\langle \phi',\Lambda^{3}\phi\rangle_{L^{2}(\Gamma)}+\langle \Lambda^{-\frac12}\varphi',\Lambda^{\frac32}\varphi\rangle_{L^{2}(\Gamma)}\,,
$$ 
and
$$
\breve\Theta_{R}:H^{3}(\Gamma)\times H^{\frac32}(\Gamma)\to L^{2}(\Gamma)
\times H^{-\frac12}(\Gamma)\,,
$$
$$
\breve\Theta_{R}(\phi,\varphi)=
\big((1/[b]+\hat\gamma_{0}\SL)\Lambda^{3}\phi+(\langle b\rangle/[b]+\hat\gamma_{0}\DL)\Lambda\varphi,
(\langle b\rangle/[b]+\hat\gamma_{1}\SL)\Lambda^{3}\phi+(b_{+}b_{-}/[b]+\hat\gamma_{1}\DL)\Lambda\varphi\big).
$$
Moreover
\begin{align*}
K_{\Sigma}
=&
\{(\phi',\varphi')\in L^{2}(\Gamma)\times H^{-\frac12}(\Gamma):\forall(\phi,\varphi)\in\Lambda^{-3}L^{2}_{\overline\Sigma}(\Gamma)\times\Lambda^{-1}H^{\frac12}_{\overline\Sigma}(\Gamma)\,,\ \langle(\varphi',\phi'),(\phi,\varphi)\rangle_{H_{R}',H_{R}}=0\}\\
=&\{(\phi',\varphi')\in L^{2}(\Gamma)\times H^{-\frac12}(\Gamma):\forall\,(\phi,\varphi)\in L^{2}_{\overline\Sigma }(\Gamma)\times H^{\frac12}_{\overline\Sigma }(\Gamma)\,,\ 
\langle\phi',
\phi\rangle_{L^{2}(\Gamma)}+
\langle\varphi',
\varphi\rangle_{-\frac12,\frac12}
=0\}\\
=&L^{2}_{\Sigma^{c}}(\Gamma)\times 
H^{-\frac12}_{\Sigma^{c}}(\Gamma)\,.
\end{align*}
Therefore, by Lemma \ref{lemma-compr}, $\t\Theta_{R,\Sigma}$ is self-adjoint on the domain
\begin{align*}
\dom(\t\Theta_{R,\Sigma})
:=&\big\{(\phi,\varphi)\in \Lambda^{-3}L^{2}_{\overline\Sigma }(\Gamma)\times\Lambda^{-1}H^{\frac12}_{\overline\Sigma }(\Gamma):\exists\,\t\phi\oplus\t\varphi\in\ran(\Pi^{\oplus}_{\Sigma})\ \text{s.t.}\,\\
& (1/[b]+\hat\gamma_{0}\SL)\Lambda^{3}\phi+(\langle b\rangle/[b]+\hat\gamma_{0}\DL)\Lambda\varphi-\t\phi\in L^{2}_{\Sigma^{c}}(\Gamma)\ \text{and}\\
&(\langle b\rangle/[b]+\hat\gamma_{1}\SL)\Lambda^{3}\phi+(b_{+}b_{-}/[b]+\hat\gamma_{1}\DL)\Lambda\varphi-\t\varphi
\in H^{-\frac12}_{\Sigma^{c}}(\Gamma)\big\}
\end{align*}
and $\t\Theta_{R,\Sigma}(\phi,\varphi):=(\t\phi,\t\varphi)$. Then
\begin{align*}
&\U^{\oplus}_{\Sigma}\t\Theta_{R,\Sigma}(\Lambda^{-3}\phi,\Lambda^{-1}\varphi)=(\t\phi|\Sigma)\oplus(\t\varphi|\Sigma)\\
=&
(((1/[b]+\hat\gamma_{0}\SL)\phi+(\langle b\rangle/[b]+\hat\gamma_{0}\DL)\varphi)|\Sigma)\oplus(((\langle b\rangle/[b]+\hat\gamma_{1}\SL)\phi+(b_{+}b_{-}/[b]+\hat\gamma_{1}\DL)\varphi)|\Sigma)
\end{align*}
and
\begin{align*}
&\dom(\Theta_{R,\Sigma})\\
=&\big\{(\phi,\varphi)\in L^{2}_{\overline\Sigma }(\Gamma)\times H^{\frac12}_{\overline\Sigma }(\Gamma):\exists\,\t\phi\oplus\t\varphi\in\ran(\Pi^{\oplus}_{\Sigma})\ \text{s.t.}\,\\ 
& (1/[b]+\hat\gamma_{0}\SL)\phi+(\langle b\rangle/[b]+\hat\gamma_{0}\DL)\varphi-\t\phi\in L^{2}_{\Sigma^{c}}(\Gamma)\ \text{and}\\
&(\langle b\rangle/[b]+\hat\gamma_{1}\SL)\phi
+(b_{+}b_{-}/[b]+\hat\gamma_{1}\DL)\varphi-\t\varphi
\in H^{-\frac12}_{\Sigma^{c}}(\Gamma)\big\}\\
\subseteq &
\big\{(\phi,\varphi)\in L^{2}_{\overline\Sigma }(\Gamma)\times H^{\frac12}_{\overline\Sigma }(\Gamma):
((1/[b]+\hat\gamma_{0}\SL)\phi+(\langle b\rangle/[b]+\hat\gamma_{0}\DL)\varphi)|\Sigma\in H^{\frac32}(\Sigma)\ \text{and}\\
&((\langle b\rangle/[b]+\hat\gamma_{1}\SL)\phi
+(b_{+}b_{-}/[b]+\hat\gamma_{1}\DL)\varphi)|\Sigma\in H^{\frac12}(\Sigma)\big\}\\
\subseteq 
&\big\{(\phi,\varphi)\in L^{2}_{\overline\Sigma }(\Gamma)\times H^{\frac12}_{\overline\Sigma }(\Gamma):\exists\,\t\phi\oplus\t\varphi\in\ran(\Pi^{\oplus}_{\Sigma})\ \text{s.t.}\,\\ 
& ((1/[b]+\hat\gamma_{0}\SL)\phi+(\langle b\rangle/[b]+\hat\gamma_{0}\DL)\varphi)|\Sigma=\t\phi|\Sigma\ \text{and}\\
&((\langle b\rangle/[b]+\hat\gamma_{1}\SL)\phi
+(b_{+}b_{-}/[b]+\hat\gamma_{1}\DL)\varphi)|\Sigma=\t\varphi|\Sigma
\big\}\\
=&\big\{(\phi,\varphi)\in L^{2}_{\overline\Sigma }(\Gamma)\times H^{\frac12}_{\overline\Sigma }(\Gamma):\exists\,\t\phi\oplus\t\varphi\in\ran(\Pi^{\oplus}_{\Sigma})\ \text{s.t.}\,\\ 
& (1/[b]+\hat\gamma_{0}\SL)\phi+(\langle b\rangle/[b]+\hat\gamma_{0}\DL)\varphi-\t\phi\in L^{2}_{\Sigma^{c}}(\Gamma)\ \text{and}\\
&(\langle b\rangle/[b]+\hat\gamma_{1}\SL)\phi
+(b_{+}b_{-}/[b]+\hat\gamma_{1}\DL)\varphi-\t\varphi
\in H^{-\frac12}_{\Sigma^{c}}(\Gamma)\big\}=\dom(\Theta_{R,\Sigma})\,.
 \end{align*}
\end{proof}
\begin{corollary} The linear operator in $L^{2}(\RE^{n})$ defined by $A_{R,\Sigma}:=(A_{{-}}^{\max}\oplus A_{{+}}^{\max})|\dom(A_{R,\Sigma})$, where 
\begin{align*}
\D(A_{R,\Sigma})
=&\{u\in H^{1}(\RE^{n}\backslash\overline\Sigma)\cap(\dom(A^{\max}_{-})\oplus\dom(A^{\max}_{+})): 
(-[\hat\gamma_{1}]u)\oplus [\gamma_{0}]u\in \dom(\Theta_{R,\Sigma})\,,\ (\gamma_{1}^\pm u-b_{\pm}\gamma_{0}^{\pm}u)|\Sigma=0\}
\\
\subseteq &\{u\in H^{1}(\RE^{n}\backslash\overline\Sigma)\cap(\dom(A^{\max}_{-})\oplus\dom(A^{\max}_{+})): (\gamma_{1}^\pm u-b_{\pm}\gamma_{0}^{\pm}u)|\Sigma=0\,,\ 
([\hat\gamma_{1}]u)|\overline\Sigma^{c}=0\}
\,,
\end{align*}
is self-adjoint and its resolvent is given by
\begin{align*}
(-A_{R,\Sigma}+z)^{-1}u
=
(-A+z)^{-1}u
-G_{z}(\Pi^{\oplus}_{\Sigma})'\left(R ^{\oplus}_{\Sigma}\left[\,\begin{matrix}1/{[b]}+\gamma_{0}\SL_{z}&\langle b\rangle/[b]+\gamma_0\DL_{z}\\
\langle b\rangle/[b]+\gamma_{1}\SL_{z}&b_{+}b_{-}/[b]+\gamma_{1}\DL_{z}
\end{matrix}\,\right](\Pi^{\oplus}_{\Sigma})'\right)^{\!\! -1}\!\! R^{\oplus}_{\Sigma}
\left[\,\begin{matrix}\gamma_{0}(-A+z)^{-1}u\\
\gamma_{1}(-A+z)^{-1}u\end{matrix}\,\right],
\end{align*}
where $R^{\oplus}_{\Sigma}$ is the restriction operator $R^{\oplus}_{\Sigma}(\phi\oplus\varphi)=(\phi|\Sigma)\oplus(\varphi|\Sigma)$,
$(\Pi^{\oplus}_{\Sigma})'$ acts there as the inclusion map $(\Pi^{\oplus}_{\Sigma})':
H_{\overline\Sigma}^{-\frac32}(\Gamma)\oplus 
H_{\overline\Sigma}^{-\frac12}(\Gamma)\to H^{-\frac32}(\Gamma)\oplus 
H^{-\frac12}(\Gamma)$ and $G_{z}$ is defined in \eqref{Gz}.
\end{corollary}
\begin{proof} By Theorem \ref{compr} and Theorem \ref{T1}, taking $\Pi=\Pi^{\oplus}_{\Sigma}$ and $\Theta=-(\U^{\oplus}_{\Sigma})^{-1}\Theta_{R,\Sigma}$, one gets the self-adjoint extension $(A_{{-}}^{\max}\oplus A_{{+}}^{\max})|\dom(A_{R,\Sigma})$ with domain (contained in $H^{1}(\RE^{n}\backslash\Gamma)$ by $\dom(\Theta_{R,\Sigma})\subseteq 
L^{2}(\Gamma)\times H^{\frac12}(\Gamma)$ and Remark \ref{regularity}), 
\begin{align*}
&\dom(A_{R,\Sigma})\\
=&\{u\in H^{1}(\RE^{n}\backslash\Gamma)\cap(\dom(A^{\max}_{-})\oplus\dom(A^{\max}_{+})): (-[\hat\gamma_{1}]u)\oplus [\gamma_{0}]u\in \dom(\Theta_{R,\Sigma})\,,\\ 
&\Pi^{\oplus}_{\Sigma}(\gamma_{0}(u+\SL[\hat\gamma_{1}u]-\DL[\gamma_{0}u])\oplus \gamma_{1}(u+\SL[\hat\gamma_{1}u]-\DL[\gamma_{0}u]))=-(\U^{\oplus}_{\Sigma}\Theta_{R,\Sigma})^{-1}((-[\hat\gamma_{1}]u)\oplus [\gamma_{0}]u)\}
\\
=&\{u\in H^{1}(\RE^{n}\backslash\overline\Sigma)\cap(\dom(A^{\max}_{-})\oplus\dom(A^{\max}_{+})): (-[\hat\gamma_{1}]u)\oplus [\gamma_{0}]u\in \dom(\Theta_{R,\Sigma})\,,\\ 
&(\gamma_{0}u)|\Sigma+(\gamma_{0}\SL[\hat\gamma_{1}u])|\Sigma-(\gamma_{0}\DL[\gamma_{0}u])|\Sigma=((1/[b]+\hat\gamma_{0}\SL)[\hat\gamma_{1}]u-(\langle b\rangle/[b]+\hat\gamma_{0}\DL)[\gamma_{0}]u)|\Sigma\\
&(\hat\gamma_{1}u)|\Sigma+(\hat\gamma_{1}\SL[\hat\gamma_{1}u])|\Sigma-(\hat\gamma_{1}\DL[\gamma_{0}u])|\Sigma=((\langle b\rangle/[b]+\hat\gamma_{1}\SL)[\hat\gamma_{1}]u-(b_{+}b_{-}/[b]+\hat\gamma_{1}\DL)[\gamma_{0}]u)|\Sigma\}\\
=&\{u\in H^{1}(\RE^{n}\backslash\overline\Sigma)\cap(\dom(A^{\max}_{-})\oplus\dom(A^{\max}_{+})): (-[\hat\gamma_{1}]u)\oplus [\gamma_{0}]u\in \dom(\Theta_{R,\Sigma})\,,\\ 
&([b]\gamma_{0}u-[\hat\gamma_{1}]u+\langle b\rangle[\gamma_{0}]u)|\Sigma=0\,,\ ([b]\gamma_{1}u-\langle b\rangle[\gamma_{1}]u+b_{+}b_{-}[\gamma_{0}]u)|\Sigma=0\}\\
=&\{u\in H^{1}(\RE^{n}\backslash\overline\Sigma)\cap(\dom(A^{\max}_{-})\oplus\dom(A^{\max}_{+})): \\
&(-[\hat\gamma_{1}]u)\oplus [\gamma_{0}]u\in \dom(\Theta_{R,\Sigma})\,,\ (\gamma_{1}^\pm u-b_{\pm}\gamma_{0}^{\pm}u)|\Sigma=0\}
\,.
\end{align*}
The formula giving $(-A_{R,\Sigma} + z)^{-1}$ is consequence of \eqref{krein1}, since
\begin{align*}
&
\left(-\U^{\oplus}_{\Sigma}\Theta_{R,\Sigma}+\Pi^{\oplus}_{\Sigma}\left[\,\begin{matrix}\gamma_{0}(\SL-\SL_{z})&\gamma_0(\DL-\DL_{z})\\
\gamma_{1}(\SL-\SL_{z})&\gamma_{1}(\DL-\DL_{z})
\end{matrix}\,\right](\Pi^{\oplus}_{\Sigma})'\right)^{\!\! -1} \Pi^{\oplus}_{\Sigma}
\\
=
&
\left(-\Theta_{R,\Sigma}+R^{\oplus}_{\Sigma}\left[\,\begin{matrix}\gamma_{0}(\SL-\SL_{z})&\gamma_0(\DL-\DL_{z})\\
\gamma_{1}(\SL-\SL_{z})&\gamma_{1}(\DL-\DL_{z})
\end{matrix}\,\right](\Pi^{\oplus}_{\Sigma})'\right)^{\!\! -1} R^{\oplus}_{\Sigma}
\end{align*}
and
\begin{align*}
-\Theta_{R,\Sigma}+R^{\oplus}_{\Sigma}\left[\,\begin{matrix}\gamma_{0}(\SL-\SL_{z})&\gamma_0(\DL-\DL_{z})&\\
\gamma_{1}(\SL-\SL_{z})&\gamma_{1}(\DL-\DL_{z})
\end{matrix}\,\right](\Pi^{\oplus}_{\Sigma})'
=R^{\oplus}_{\Sigma}\left[\,\begin{matrix}1/{[b]}+\gamma_{0}\SL_{z}&\langle b\rangle/[b]+\gamma_0\DL_{z}\\
\langle b\rangle/[b]+\gamma_{1}\SL_{z}&b_{+}b_{-}/[b]+\gamma_{1}\DL_{z}
\end{matrix}\,\right](\Pi^{\oplus}_{\Sigma})'&\,.
\end{align*}
\end{proof}
\begin{remark}
Since $\supp([\gamma_{0}]u)\subseteq\overline\Sigma$ and $\supp([\gamma_{1}]u)\subseteq\overline\Sigma$, for any $u\in\dom(A_{R,\Sigma})$, one has 
\begin{align*}
A_{R,\Sigma}u=Au-[\gamma_{1}]u\,\delta_{\overline\Sigma}-[\gamma_{0}]u\,\partial_{\underline a}\delta_{\overline\Sigma}
=
Au-\frac4{[b]}\,\left((\langle b\rangle\,\gamma_{1}u-b_+b_{-}\gamma_{0}u)\,\delta_{\overline\Sigma}+(\gamma_{1}u-\langle b\rangle\,\gamma_{0}u)\,\partial_{\underline a}\delta_{\overline\Sigma}\right)
\end{align*}
and so 
$(A_{R,\Sigma}u)|\overline\Sigma^{c}=(Au)|\overline\Sigma^{c}$. This also shows that $A_{R,\Sigma}$ is a self-adjoint extension of the symmetric operator $A|\C^{\infty}_{\comp}(\RE^{n}\backslash\overline\Sigma)$. Hence it depends only on $\Sigma$ and $b_{\pm}|\Sigma$ and not on the whole $\Gamma$: one would obtain the same operator by considering any other bounded domain $\Omega_{\circ}$ with boundary $\Gamma_{\!\circ }$ such that $\Sigma\subset\Gamma_{\!\circ}$. 
\end{remark}
\begin{remark} By the results contained in the proof of Theorem \ref{comprR}, for the domain of the sesquilinear form $f_{R,\Sigma}$ associated to the self-adjoint operator $\t\Theta_{R,\Sigma}$ one has the relation $\dom(f_{R,\Sigma})\subseteq H^{3}(\Gamma)\oplus H^{\frac32}(\Gamma)$. Therefore Theorem \ref{teo-schatten2} applies and so, by Corollary \ref{spectrum},  $\sigma_{ac}(A_{R,\Sigma})=\sigma_{ac}(A)$ and the wave operators $W_{\pm}(A_{R,\Sigma},A)$, $W_{\pm}(A,A_{R,\Sigma})$ exist and are complete. 
\end{remark}
\begin{remark} Let $\phi\oplus\varphi\in\dom(\Theta_{R,\Sigma})$. Then, by 
$$((1/[b]+\gamma_{0}\SL)\phi+(\langle b\rangle/[b]+\hat\gamma_{0}\DL)\varphi)|\Sigma\in H^{\frac32}(\Sigma)\,,
$$
there follows $\phi|\Sigma\in H^{\frac12}(\Sigma)\subseteq H_{\overline\Sigma}^{s}(\Gamma)$, $s<\frac12$. Hence, by
$$((\langle b\rangle/[b]+\hat\gamma_{1}\SL)\phi
+(b_{+}b_{-}/[b]+\hat\gamma_{1}\DL)\varphi)|\Sigma\in H^{\frac12}(\Sigma)\,,
$$
one gets $(\hat\gamma_{1}\DL\varphi)|\Sigma\in H^{s}(\Sigma)$, $s<\frac12$. Using \eqref{reg3} and the same arguments as in the proof of Lemma \ref{regularity2}, one obtains 
$$
\dom(A_{R,\Sigma})\subseteq H^{2-}(\RE^{n}\backslash(\overline\Sigma\cup(\partial\Sigma)^{\epsilon}))\cap H^{2}(\RE^{n}\backslash\Sigma^{\epsilon})\,.
$$
\end{remark}
\begin{remark} Suppose $n=3$ and both $\Gamma$ and $\Sigma$ are smooth. By Theorem \ref{comprR}, we have
$$
((\langle b\rangle/[b]+\hat\gamma_{1}\SL)\phi
+(b_{+}b_{-}/[b]+\hat\gamma_{1}\DL)\varphi)|\Sigma\in H^{\frac12}(\Sigma)
$$
with $\phi\oplus\varphi\in\dom(\Theta_{R,\Sigma})\subseteq L^{2}(\Gamma)\oplus H^{\frac12}(\Gamma)$. This yields $(\hat\gamma_{1}\DL\varphi)|\Sigma\in L^{2}(\Sigma)\subseteq H^{s}(\Sigma)$, $s<0$ and so, according to \eqref{CS2}, we get $\varphi\in H^{s}_{\overline\Sigma}(\Gamma)$, $\frac12\le s<1$. 
Therefore $\dom(\Theta_{R,\Sigma})\subseteq L^{2}(\Gamma)\oplus H^{s}(\Gamma)$, $\frac 12\le s<1$, and so, by Remark \ref{regularity} and \eqref{H4}, 
$$
\D(A_{R,\Sigma})\subseteq H^{\frac32-}(\RE^{3}\backslash\overline\Sigma)\,.
$$
\end{remark}

\end{subsection}  
\begin{subsection}{$\delta$-interactions.}\label{delta-loc} Given $\Sigma\subset \Gamma$ relatively open of class $\C^{0,1}$,  we denote by $\Pi_{\Sigma}$ the orthogonal projector in the Hilbert space $H^{\frac32}(\Gamma)$ such that $\ran(\Pi_{\Sigma})=H^{\frac32}_{\Sigma^{c}}(\Gamma)^{\perp}\simeq H^{\frac32}(\Sigma)$. Here $\Pi_{\Sigma}'$, $R_{\Sigma}$ and $\U_{\Sigma}$ denote the same operators 
as in Subsection \ref{dir-loc}.
\begin{theorem}\label{compr-delta} Let $\Omega$ be of class $\C^{4,1}$,  let $\alpha\in M^{\frac32}(\Gamma)$ be real-valued and let $1/\alpha\in L^{\infty}(\Gamma)$ have constant sign on (each connected component of) $\Gamma$. Then 
$$\Theta_{\alpha,D,\Sigma}:\dom(\Theta_{\alpha,D,\Sigma})
\subseteq H^{-\frac32}_{\overline\Sigma}(\Gamma)\to H^{\frac32}(\Sigma)\,,$$
$$ 
\Theta_{\alpha,D,\Sigma}\phi:=R_{\Sigma}(1/\alpha+\gamma_{0}\SL)\Pi_{\Sigma}'\phi\equiv((1/\alpha+\gamma_{0}\SL)\phi)|\Sigma\,,
$$
$$
\dom(\Theta_{\alpha,D,\Sigma}):=\{\phi\in H^{-\frac12}_{\overline\Sigma }(\Gamma): ((1/\alpha+\gamma_{0}\SL)\phi)|\Sigma\in H^{\frac32}(\Sigma)\}\,,
$$
is self-adjoint.
\end{theorem}
\begin{proof} Since $\gamma_{0}\SL$ is infinitesimally $1/\alpha$-bounded (see the proof of Lemma \ref{lemmadelta}), the self-adjoint operator in $H^{\frac32}(\Gamma)$  
$$
\t\Theta_{\alpha,D}:=\Theta_{\alpha,D}\Lambda^{3}:H^{9/2}(\Gamma)\subseteq H^{\frac32}(\Gamma)\to H^{\frac32}(\Gamma)\,,
$$
is semibounded: it is strictly positive whenever $1/\alpha\ge 0$ and upper bounded whenever $1/\alpha\le 0$. In the following we suppose $1/\alpha\ge 0$, the case $1/\alpha\le 0$ being treated in a similar way. Let $f_{\alpha,D}$ be the densely defined, strictly positive, closed sesquilinear form associated with $\t\Theta_{\alpha,D}$,
i.e.
$$
f_{\alpha,D}:H^{3}(\Gamma)\times H^{3}(\Gamma)
\subseteq
H^{\frac32}(\Gamma)\times H^{\frac32}(\Gamma)\to\RE\,,
$$
\begin{align*}
f_{\alpha,D}(\phi_{1}\phi_{2}):=\langle(1/\alpha+\gamma_{0}\SL)\Lambda^{3}\phi_{1},\Lambda^{3}\phi_{2}\rangle_{L^{2}(\Gamma)}\,.
\end{align*}
Since 
$H^{3}(\Gamma)\cap H^{\frac32}_{\Sigma^{c}}(\Gamma)^{\perp}=H^{3}(\Gamma)\cap\Lambda^{-3}H^{-\frac32}_{\overline\Sigma}(\Gamma)=\Lambda^{-3}
L^{2}_{\overline\Sigma }(\Gamma)$ is dense in $\Lambda^{-3} H^{-\frac32}_{\overline\Sigma }(\Gamma)$, we can use  Lemma \ref{lemma-compr} to determine the semibounded   self-adjoint operator $\t\Theta_{\alpha,D,\Sigma}$ in $\ran(\Pi_{\Sigma})$ associated to the restriction of $f_{\alpha,D}$ to $\ran(\Pi_{\Sigma})$. Then 
$$
\Theta_{\alpha,D,\Sigma}:=\U_{\Sigma}\t\Theta_{\alpha,D,\Sigma}\Lambda^{-3}\,,\quad \dom(\Theta_{\alpha,D,\Sigma}):=\Lambda^{3}\dom(\t\Theta_{\alpha,D,\Sigma})\,,
$$
is self-adjoint. To conclude the proof we need to determine the operator  $\breve\Theta_{\alpha,D}:=(\t\Theta_{\alpha,D})\breve{\,}$ and the subspace $K_{\Sigma}:=\fk_{\Pi_{\Sigma}}$ (we refer to the Appendix for the notations). Let $H_{\alpha,D}$ be the Hilbert space given by $\dom(f_{\alpha,D})=H^{3}(\Gamma)$ endowed with the 
scalar product  $$\langle\phi_{1},\phi_{2}\rangle_{\alpha,D}:=f_{\alpha,D}(\phi_{1},\phi_{2})
\,.
$$ 
Let $H_{\alpha,D}'$ denote its dual space.
Since
\begin{align*}
f_{\alpha,D}(\phi_{1},\phi_{2})=
\langle(1/\alpha+\gamma_{0}\SL)\Lambda^{3}\phi_{1}
,\Lambda^{3}\phi_{2}\rangle_{L^{2}(\Gamma)}
\,,
\end{align*}
one has 
$$H_{\alpha,D}'=L^{2}(\Gamma)\,,\quad \langle \varphi,\phi\rangle_{H'_{\alpha,D},H_{\alpha,D}}=\langle \varphi,\Lambda^{3}\phi\rangle_{L^{2}(\Gamma)}\,,
$$ 
and
$$
\breve\Theta_{\alpha,D}:H^{3}(\Gamma)\to L^{2}(\Gamma)\,,
\quad
\breve\Theta_{\alpha,D}=(1/\alpha+\gamma_{0}\SL)\Lambda^{3}\,.
$$
Moreover
\begin{align*}
K_{\Sigma}
=&
\{\varphi\in L^{2}(\Gamma):\forall\phi\in\Lambda^{-3}L^{2}_{\overline\Sigma}(\Gamma)\,,\ 
\langle \varphi,\Lambda^{3}\phi\rangle_{L^{2}(\Gamma)}=0\}\\
=&\{\varphi\in L^{2}(\Gamma):\forall\phi\in L^{2}_{\overline\Sigma }(\Gamma)\,,\ 
\langle\varphi,\phi\rangle_{L^{2}(\Gamma)}=0\}\\
=&L^{2}_{\Sigma^{c}}(\Gamma)\,.
\end{align*}
Therefore, by Lemma 5.1, $\t\Theta_{\alpha,D,\Sigma}$ is self-adjoint on the domain
\begin{align*}
\dom(\t\Theta_{\alpha,D,\Sigma}):=\{\phi\in \Lambda^{-3}L^{2}_{\overline\Sigma }(\Gamma):\exists\,\t\phi\in\ran(\Pi_{\Sigma})\ \text{s.t.}\ (1/\alpha+\gamma_{0}\SL)\Lambda^{3}\phi-\t\phi\in L^{2}_{\Sigma^{c}}(\Gamma)\}
\end{align*}
and $\t\Theta_{\alpha,D,\Sigma}\phi:=\t\phi$. Then
$$
\U_{\Sigma}\t\Theta_{\alpha,D,\Sigma}\Lambda^{-3}\phi=\t\phi|\Sigma=
((1/\alpha+\gamma_{0}\SL)\phi)|\Sigma
$$
and
\begin{align*}
\dom(\Theta_{\alpha,D,\Sigma})
=&\{\phi\in L^{2}_{\overline\Sigma }(\Gamma):\exists\,\t\phi\in\ran(\Pi_{\Sigma})\ \text{s.t.}\ (1/\alpha+\gamma_{0}\SL)\phi-\t\phi\in L^{2}_{\Sigma^{c}}(\Gamma)\}\\
\subseteq&\{\phi\in L^{2}_{\overline\Sigma }(\Gamma):((1/\alpha+\gamma_{0}\SL)\phi)|\Sigma\in H^{\frac32}(\Sigma)\}\\
\subseteq&\{\phi\in L^{2}_{\overline\Sigma }(\Gamma):\exists\,\t\phi\in\ran(\Pi_{\Sigma})\ \text{s.t.}\ ((1/\alpha+\gamma_{0}\SL)\phi)|\Sigma=\t\phi|\Sigma\}\\
=&\{\phi\in L^{2}_{\overline\Sigma }(\Gamma):\exists\,\t\phi\in\ran(\Pi_{\Sigma})\ \text{s.t.}\ (1/\alpha+\gamma_{0}\SL)\phi-\t\phi\in L^{2}_{\Sigma^{c}}(\Gamma)\}=\dom(\Theta_{\alpha,D,\Sigma})\,.
\end{align*}
\end{proof}
\begin{remark}\label{6.17}
Let $\phi\in\dom(\Theta_{\alpha,D,\Sigma})\subseteq L^{2}_{\overline\Sigma}(\Gamma)$. Then $\gamma_{0}\SL\phi\in H^{1}(\Gamma)$ and so, by  $((1/\alpha+\gamma_{0}\SL)\phi)|\Sigma\in H^{\frac32}(\Sigma)$, one gets $(\phi/\alpha)|\Sigma\in H^{1}(\Sigma)$. Thus, since $H^{s}(\Sigma)=H^{s}_{\overline\Sigma}(\Gamma)$ for any $s\in [0,\frac12)$, in conclusion one has $\dom(\Theta_{\alpha,D,\Sigma})\subseteq H^{s}(\Gamma)$ for any $s\in [0,\frac12)$.
\end{remark}
\begin{corollary} The linear operator in $L^{2}(\RE^{n})$ defined by $A_{\alpha,\delta,\Sigma}:=(A_{{-}}^{\max}\oplus A_{{+}}^{\max})|\dom(A_{\alpha,\delta,\Sigma})$ with domain
\begin{align*}
\D(A_{\alpha,\delta,\Sigma})
=&\{u\in H^{2-}(\RE^{n}\backslash\overline\Sigma)\cap (\dom(A_{-}^{\max})\oplus\dom(A_{+}^{\max})): [\gamma_{1}]u\in \dom(\Theta_{\alpha,D,\Sigma})\,,\ (\alpha\gamma_{0}u-[\gamma_{1}]u)|\Sigma=0\}\\
\subseteq& \{u\in H^{2-}(\RE^{n}\backslash\overline\Sigma)\cap (\dom(A_{-}^{\max})\oplus\dom(A_{+}^{\max})): (\alpha\gamma_{0}u-[\gamma_{1}]u)|\Sigma=0
\}\,,
\end{align*}
is self-adjoint and its resolvent is given by 
\begin{align*}
(-A_{\alpha,\delta,\Sigma}+z)^{-1}u
=
(-A+z)^{-1}-\SL_{z}\Pi_{\Sigma}'(R_{\Sigma}(1+\alpha\gamma_{0}\SL_{z})\Pi_{\Sigma}')^{-1}R_{\Sigma}\alpha\gamma_{0}(-A+z)^{-1}\,,
\end{align*}
where $R_{\Sigma}$ is the restriction operator $R_{\Sigma}\phi=\phi|\Sigma$ and 
$\Pi'_{\Sigma}$ acts there as the inclusion map  
$\Pi'_{\Sigma}:H_{\overline\Sigma}^{-\frac32}(\Gamma)\to H^{-\frac32}(\Gamma)$.
\end{corollary}
\begin{proof}
By Theorem \ref{compr-delta} and Theorem \ref{T1},  taking $\Pi(\phi\oplus\varphi)=\Pi_{\Sigma}\phi\oplus 0$ and $\Theta(\phi\oplus\varphi)=(-\U_{\Sigma}^{-1}\Theta_{\alpha,D,\Sigma}\phi)\oplus 0$, one gets the self-adjoint extension $(A_{{-}}^{\max}\oplus A_{{+}}^{\max})|\dom(A_{\alpha,\delta,\Sigma})$ with domain (contained in $H^{s}(\RE^{n}\backslash\Gamma)$, $s<2$, by $\dom(\Theta_{\alpha,D,\Sigma})\subseteq H^{s}(\Gamma)$, $s<\frac12$, and by Remark \ref{regularity}) 
\begin{align*}
&\D(A_{\alpha,\delta,\Sigma})\\
=&\{u\in H^{2-}(\RE^{n}\backslash\Gamma)\cap (\dom(A_{-}^{\max})\oplus\dom(A_{+}^{\max})): [\gamma_{1}]u\in \dom(\Theta_{\alpha,D,\Sigma})\,,\ 
\Pi_{\Sigma}(\gamma_{0}(u+\SL[\gamma_{1}u]))=U^{-1}_{\Sigma}\Theta_{\alpha,D,\Sigma}[\gamma_{1}u]
\}\\
=&\{u\in H^{2-}(\RE^{n}\backslash\Gamma)\cap (\dom(A_{-}^{\max})\oplus\dom(A_{+}^{\max})): [\gamma_{1}]u\in \dom(\Theta_{\alpha,D,\Sigma})\,,\ 
(\gamma_{0}u)|\Sigma+(\SL[\gamma_{1}u])|\Sigma=((1/\alpha+\gamma_{0}\SL)[\hat\gamma_{1}u])|\Sigma
\}\\
=&\{u\in H^{2-}(\RE^{n}\backslash\overline\Sigma)\cap (\dom(A_{-}^{\max})\oplus\dom(A_{+}^{\max})): [\gamma_{1}]u\in \dom(\Theta_{\alpha,D,\Sigma})\,,\ 
(\alpha\gamma_{0}u-[\gamma_{1}]u)|\Sigma=0\}\\
\subseteq& \{u\in H^{s}(\RE^{n}\backslash\overline\Sigma)\cap (\dom(A_{-}^{\max})\oplus\dom(A_{+}^{\max})): (\alpha\gamma_{0}u-[\gamma_{1}]u)|\Sigma=0\}\,,
\end{align*}
The formula giving  $(-A_{\alpha,\delta,\Sigma}+z)^{-1}$ is consequence of \eqref{krein1}, since   
$$
(-\U_{\Sigma}\Theta_{\alpha,D,\Sigma}+\Pi_{\Sigma}\gamma_{0}(\SL-\SL_{z})\Pi_{\Sigma}')^{-1}\Pi_{\Sigma}
=(-\Theta_{\alpha,D,\Sigma}+R_{\Sigma}\gamma_{0}(\SL-\SL_{z})\Pi_{\Sigma}')^{-1}R_{\Sigma}
$$
and
\begin{align*}
-\Theta_{\alpha,D,\Sigma}\phi+R_{\Sigma}\gamma_{0}(\SL-\SL_{z})\phi
=&
-((1/\alpha+\gamma_{0}\SL)\phi)|\Sigma+(\gamma_{0}\SL\phi)|\Sigma
-(\gamma_{0}\SL_{z}\phi)|\Sigma\\
=&-((1/\alpha+\gamma_{0}\SL_{z})\phi)|\Sigma=
-(1/(\alpha|\Sigma))(1+\alpha\gamma_{0}\SL_{z})\phi)|\Sigma\,.
\end{align*}
\end{proof}
\begin{remark}
Since $\supp([\gamma_{0}]u)\subseteq\overline\Sigma$, for any $u\in\dom(A_{\alpha,\delta,\Sigma})$, one has 
\begin{align*}
&A_{\alpha,\delta,\Sigma}\,u=Au-\alpha\gamma_{0}u\,\delta_{\overline\Sigma}
\end{align*}
and so 
$(A_{\alpha,\delta,\Sigma}\,u)|\overline\Sigma^{c}=(Au)|\overline\Sigma^{c}$. This also shows that $A_{\alpha,\delta,\Sigma}$ is a self-adjoint extension of the symmetric operator $A|\C^{\infty}_{\comp}(\RE^{n}\backslash\overline\Sigma)$. Hence it depends only on $\Sigma$ and $\alpha|\Sigma$ and not on the whole $\Gamma$: one would obtain the same operator by considering any other bounded domain $\Omega_{\circ}$ with boundary $\Gamma_{\!\circ }$ such that $\Sigma\subset\Gamma_{\!\circ}$. 
\end{remark}
\begin{remark}
According to Remark \ref{6.17}, $\phi|\Sigma\in H^{1}(\Sigma)$ whenever  $\phi\in\dom(\Theta_{\alpha,D,\Sigma})$; thus Lemma \ref{regularity2} applies and one gets 
$$
\dom(A_{\alpha,\delta,\Sigma})\subseteq H^{2}(\RE^{n}\backslash(\overline\Sigma\cup(\partial\Sigma)^{\epsilon}))\,.
$$
\end{remark}
\begin{remark} By the results contained in the proof of Theorem \ref{compr-delta}, for the domain of the sesquilinear form $f_{\alpha,D,\Sigma}$ associated to the self-adjoint operator $\t\Theta_{\alpha,D,\Sigma}$ one has the relation $\dom(f_{\alpha,D,\Sigma})\subseteq H^{3}(\Gamma)$. Therefore Theorem \ref{teo-schatten2} applies and so, by Corollary \ref{spectrum},  $\sigma_{ac}(A_{\alpha,\delta,\Sigma})=\sigma_{ac}(A)$ and the wave operators $W_{\pm}(A_{\alpha,\delta,\Sigma},A)$, $W_{\pm}(A,A_{\alpha,\delta,\Sigma})$ exist and are complete. 
\end{remark}
\end{subsection}  
\begin{subsection} {$\delta'$-interactions.}\label{delta'-loc} Given $\Sigma\subset \Gamma$ relatively open of class $\C^{0,1}$,  we denote by $\Pi_{\Sigma}$ the orthogonal projector in the Hilbert space $H^{\frac12}(\Gamma)$ such that $\ran(\Pi_{\Sigma})=H^{\frac12}_{\Sigma^{c}}(\Gamma)^{\perp}\simeq H^{\frac12}(\Sigma)$. Here $\Pi_{\Sigma}'$, $R_{\Sigma}$ and $\U_{\Sigma}$ denote the same operators 
as in Subsection \ref{neu-loc}.
\begin{theorem}\label{compr-delta'} Let $\Omega$ be of class $\C^{2,1}$ and let $\beta\in M^{\frac12}(\Gamma)$ be real valued and $1/\beta\in L^{\infty}(\Gamma)$. Then  
$$\Theta_{\beta,N,\Sigma}:
\dom(\Theta_{\beta,N,\Sigma})\subseteq 
 H^{-\frac12}_{\overline\Sigma}(\Gamma)\to  H^{\frac12}(\Sigma)\,,
$$
$$\Theta_{\beta,N,\Sigma}\phi:=R_{\Sigma}(-1/\beta+\hat\gamma_{1}\DL)\Pi_{\Sigma}'\phi\equiv
(-1/\beta+\hat\gamma_{1}\DL)\phi)|\Sigma\,,
$$
$$
\dom(\Theta_{\beta,N,\Sigma}):=\{\varphi\in H^{\frac12}_{\overline\Sigma }(\Gamma):
((-1/\beta+\hat\gamma_{1}\DL)\varphi)|\Sigma \in H^{\frac12}(\Sigma)\}\,,
$$
is self-adjoint.
\end{theorem}
\begin{proof} By \eqref{coercive2}, since $1/\beta$ is infinitesimally $\hat\gamma_{1}\DL$-bounded (see the proof of Lemma \ref{lemmadeltaprimo}), the self-adjoint operator in $H^{\frac12}(\Gamma)$ given by 
$$
\t\Theta_{\beta,N}:=\Theta_{\beta,N}\Lambda:H^{5/2}(\Gamma)\subseteq H^{\frac12}(\Gamma)\to H^{\frac12}(\Gamma)\,,
$$ 
is upper bounded. Let $f_{\beta,N}$ be the densely defined, semibounded, closed sesquilinear form associated with $\t\Theta_{\beta,N}$, i.e.
$$
f_{\beta,N}:H^{\frac32}(\Gamma)\times H^{\frac32}(\Gamma)
\subset
H^{\frac12}(\Gamma)\times  H^{\frac12}(\Gamma)\to\RE\,,
$$
\begin{align*}
f_{\beta,N}(\varphi_{1},\varphi_{2}):=
\langle(-1/\beta+\hat\gamma_{1}\DL)\Lambda\varphi_{1},\Lambda\varphi_{2}\rangle_{-\frac12,\frac12}.
\end{align*}
Since $H^{\frac32}(\Gamma)\cap H^{\frac12}_{\Sigma^{c}}(\Gamma)^{\perp}=H^{\frac32}(\Gamma)\cap\Lambda^{-1}H^{-\frac12}_{\overline\Sigma}(\Gamma)=\Lambda^{-1}
H^{\frac12}_{\overline\Sigma }(\Gamma)$ is dense in $\Lambda^{-1} H^{-\frac12}_{\overline\Sigma }(\Gamma)$, we can use  Lemma \ref{lemma-compr} to determine the semibounded   self-adjoint operator $\t\Theta_{\beta,N,\Sigma}$ in $\ran(\Pi_{\Sigma})$ associated to the restriction of $f_{\beta,N}$ to $\ran(\Pi_{\Sigma})$. Then
$$
\Theta_{\beta,N,\Sigma}:=\U_{\Sigma}\t\Theta_{\beta,N,\Sigma}\Lambda^{-1}\,,\quad 
\dom(\Theta_{\beta,N,\Sigma}):=\Lambda\dom(\t\Theta_{\beta,N,\Sigma})\,,
$$
is self-adjoint. To conclude the proof we need to determine the operator $\breve\Theta_{\beta,N}:=(\t\Theta_{\beta,N})\breve{\,}$ and the subspace $K_{\Sigma}:=\fk_{\Pi_{\Sigma}}$ (we refer to the Appendix for the notations). Let $H_{\beta,N}$ be the Hilbert space given by $\dom(f_{\beta,N})= H^{\frac32}(\Gamma)$ endowed with the 
scalar product  $$\langle(\varphi_{1},\varphi_{2}\rangle_{\beta,N}:=(-f_{\beta,N}+\lambda_{\beta,N})(\varphi_{1}\varphi_{2})
\,,
$$ 
where $\lambda_{\beta,N}$ is chosen in such a way to have $-f_{\beta,N}+\lambda_{\beta,N}>0$.
Let $H_{\beta,N}'$ denote its dual space.
Since
\begin{align*}
f_{\beta,N}(\varphi_{1},\varphi_{2})=
\langle\Lambda^{-\frac12}(-1/\beta-\hat\gamma_{1}\DL))\Lambda\varphi_{1},\Lambda^{\frac32}\varphi_{2}\rangle_{L^{2}(\Gamma)}\,,
\end{align*}
one has 
$$H_{\beta,N}'= H^{-\frac12}(\Gamma)\,,\quad \langle (\phi,\varphi)\rangle_{H'_{\beta,N},H_{\beta,N}}=\langle \Lambda^{-\frac12}\phi,\Lambda^{\frac32}\varphi\rangle_{L^{2}(\Gamma)}\,,
$$ 
and
$$
\breve\Theta_{\beta,N}: H^{\frac32}(\Gamma)\to  H^{-\frac12}(\Gamma)\,,
\quad
\breve\Theta_{\beta,N}=(-1/\beta+\hat\gamma_{1}\DL)\Lambda\,.
$$
Moreover
\begin{align*}
K_{\Sigma}=&
\{\phi\in H^{-\frac12}(\Gamma):\forall\varphi\in\Lambda^{-1}H^{\frac12}_{\overline\Sigma}(\Gamma)\,,\ \langle\Lambda^{-\frac12}\phi,\Lambda^{\frac32}\varphi\rangle_{L^{2}(\Gamma)}=0\}\\
=&\{\phi\in H^{-\frac12}(\Gamma):\forall\varphi\in H^{\frac12}_{\overline\Sigma }(\Gamma)\,,\ 
\langle\phi,
\varphi\rangle_{-\frac12,\frac12}
=0\}\\
=&H^{-\frac12}_{\Sigma^{c}}(\Gamma)\,.
\end{align*}
Therefore, by Lemma 5.1, $\t\Theta_{\beta,N,\Sigma}$ is self-adjoint on the domain 
\begin{align*}
\dom(\t\Theta_{\beta,N,\Sigma}):=\{\varphi\in\Lambda^{-1}H^{\frac12}_{\overline\Sigma }(\Gamma):\exists\,\t\varphi\in\ran(\Pi_{\Sigma})\ \text{s.t.}\,\ (-1/\beta+\hat\gamma_{1}\DL)\Lambda\varphi-\t\varphi
\in H^{-\frac12}_{\Sigma^{c}}(\Gamma)\}
\end{align*}
and $\t\Theta_{\beta,N,\Sigma})\varphi:=\t\varphi$. Then
$$
\U_{\Sigma}\t\Theta_{\beta,N,\Sigma}\Lambda^{-1}\varphi=\t\varphi|\Sigma=
((-1/\beta+\hat\gamma_{1}\DL)\varphi)|\Sigma
$$
and
\begin{align*}
\dom(\Theta_{\beta,N,\Sigma})=&\{\varphi\in H^{\frac12}_{\overline\Sigma }(\Gamma):\exists\,\t\varphi\in\ran(\Pi_{\Sigma})\ \text{s.t.}\,\ (-1/\beta+\hat\gamma_{1}\DL)\varphi-\t\varphi
\in H^{-\frac12}_{\Sigma^{c}}(\Gamma)\}\\
\subseteq&\{\varphi\in H^{\frac12}_{\overline\Sigma }(\Gamma):((-1/\beta+\hat\gamma_{1}\DL)\varphi)|\Sigma
\in H^{\frac12}(\Sigma)\}\\
\subseteq&\{\varphi\in H^{\frac12}_{\overline\Sigma }(\Gamma):\exists\,\t\varphi\in\ran(\Pi_{\Sigma})\ \text{s.t.}\,\ ((-1/\beta+\hat\gamma_{1}\DL)\varphi)|\Sigma=\t\varphi|\Sigma\}\\
=&\{\varphi\in H^{\frac12}_{\overline\Sigma }(\Gamma):\exists\,\t\varphi\in\ran(\Pi_{\Sigma})\ \text{s.t.}\,\ (-1/\beta+\hat\gamma_{1}\DL)\varphi-\t\varphi
\in H^{-\frac12}_{\Sigma^{c}}(\Gamma)\}=\dom(\Theta_{\beta,N,\Sigma})\,.
\end{align*}
\end{proof}
\begin{corollary} The linear operator in $L^{2}(\RE^{n})$ defined by $A_{\beta,\delta',\Sigma}:=(A_{{-}}^{\max}\oplus A_{{+}}^{\max})|\dom(A_{\beta,\delta',\Sigma})$ with domain
\begin{align*}
\D(A_{\beta,\delta',\Sigma})
=&\{u\in H^{1}(\RE^{n}\backslash\overline\Sigma)\cap (\dom(A_{-}^{\max})\oplus\dom(A_{+}^{\max})): [\gamma_{0}]u\in \dom(\Theta_{\beta,N,\Sigma})\,,\ (\beta\hat\gamma_{1}u-[\gamma_{0}]u)|\Sigma=0\}\\
\subseteq& \{u\in H^{1}(\RE^{n}\backslash\overline\Sigma)\cap (\dom(A_{-}^{\max})\oplus\dom(A_{+}^{\max})): (\beta\gamma_{0}u-[\hat\gamma_{1}]u)|\Sigma=0
\}\,,
\end{align*}
is self-adjoint and its resolvent is given by 
\begin{align*}
(-A_{\beta,\delta',\Sigma}+z)^{-1}u
=
(-A+z)^{-1}+\DL_{z}\Pi_{\Sigma}'(R_{\Sigma}(1-\beta\hat\gamma_{1}\DL_{z})\Pi_{\Sigma}')^{-1}R_{\Sigma}\beta\gamma_{1}(-A+z)^{-1}\,,
\end{align*}
where $R_{\Sigma}$ is the restriction operator $R_{\Sigma}\phi=\phi|\Sigma$ and 
$\Pi'_{\Sigma}$ acts there as the inclusion map  
$\Pi'_{\Sigma}:H_{\overline\Sigma}^{-\frac12}(\Gamma)\to H^{-\frac12}(\Gamma)$.
\end{corollary}
\begin{proof}
By Theorem \ref{compr-delta} and Theorem \ref{T1},  taking $\Pi(\phi\oplus\varphi)=0\oplus \Pi_{\Sigma}$ and $\Theta(\phi\oplus\varphi)=0\oplus (-\U_{\Sigma}^{-1}\Theta_{\beta,N,\Sigma}\phi)$, one gets the self-adjoint extension $(A_{{-}}^{\max}\oplus A_{{+}}^{\max})|\dom(A_{\beta,\delta',\Sigma})$ with domain (contained in $H^{1}(\RE^{n}\backslash\Gamma)$,  by $\dom(\Theta_{\beta,N,\Sigma})\subseteq H^{\frac12}(\Gamma)$ and by Remark \ref{regularity}) 
\begin{align*}
&\D(A_{\beta,\delta',\Sigma})\\
=&\{u\in H^{1}(\RE^{n}\backslash\Gamma)\cap (\dom(A_{-}^{\max})\oplus\dom(A_{+}^{\max})): [\gamma_{0}]u\in \dom(\Theta_{\beta,N,\Sigma})\,,\ 
\Pi_{\Sigma}(\gamma_{1}(u-\DL[\gamma_{0}u]))=-U^{-1}_{\Sigma}\Theta_{\beta,N,\Sigma}[\gamma_{0}u]
\}\\
=&\{u\in H^{1}(\RE^{n}\backslash\Gamma)\cap (\dom(A_{-}^{\max})\oplus\dom(A_{+}^{\max})): [\gamma_{0}]u\in \dom(\Theta_{\beta,N,\Sigma})\,,\ 
(\gamma_{1}u)|\Sigma-(\DL[\gamma_{0}u])|\Sigma=-((-1/\beta+\hat\gamma_{1}\DL)[\gamma_{0}u])|\Sigma
\}\\
=&\{u\in H^{1}(\RE^{n}\backslash\overline\Sigma)\cap (\dom(A_{-}^{\max})\oplus\dom(A_{+}^{\max})): [\gamma_{0}]u\in \dom(\Theta_{\beta,N,\Sigma})\,,\ 
(\beta\hat\gamma_{1}u-[\gamma_{0}]u)|\Sigma=0\}\\
\subseteq& \{u\in H^{1}(\RE^{n}\backslash\overline\Sigma)\cap (\dom(A_{-}^{\max})\oplus\dom(A_{+}^{\max})): (\beta\hat\gamma_{1}u-[\gamma_{0}]u)|\Sigma=0\}\,,
\end{align*}
The formula giving  $(-A_{\beta,\delta',\Sigma}+z)^{-1}$ is consequence of \eqref{krein1}, since   
$$
(-\U_{\Sigma}\Theta_{\beta,N,\Sigma}+\Pi_{\Sigma}\gamma_{1}(\DL-\DL_{z})\Pi_{\Sigma}')^{-1}\Pi_{\Sigma}
=(-\Theta_{\beta,N,\Sigma}+R_{\Sigma}\gamma_{1}(\DL-\DL_{z})\Pi_{\Sigma}')^{-1}R_{\Sigma}
$$
and
\begin{align*}
-\Theta_{\beta,N,\Sigma}\varphi+R_{\Sigma}\gamma_{1}(\DL-\DL_{z})\varphi
=&
((1/\beta-\hat\gamma_{1}\DL)\phi)|\Sigma+(\hat\gamma_{1}\DL\phi)|\Sigma
-(\hat\gamma_{1}\DL_{z}\varphi)|\Sigma\\
=&((1/\beta-\hat\gamma_{1}\DL_{z})\varphi)|\Sigma=
(1/(\beta|\Sigma))(1-\beta\hat\gamma_{1}\DL_{z})\phi)|\Sigma\,.
\end{align*}
\end{proof}
\begin{remark}
Since $\supp([\gamma_{1}]u)\subseteq\overline\Sigma$, for any $u\in\dom(A_{\beta,\delta',\Sigma})$, one has 
\begin{align*}
A_{\beta,\delta',\Sigma}u=Au-\beta\gamma_{1}u\,\partial_{\underline a}\delta_{\overline\Sigma}
\end{align*}
and so 
$(A_{\beta,\delta',\Sigma}u)|\overline\Sigma^{c}=(Au)|\overline\Sigma^{c}$. 
This also shows that $A_{\beta,\delta',\Sigma}$ is a self-adjoint extension of the symmetric operator $A|\C^{\infty}_{\comp}(\RE^{n}\backslash\overline\Sigma)$. Hence it depends only on $\Sigma$ and $\beta|\Sigma$ and not on the whole $\Gamma$: one would obtain the same operator by considering any other bounded domain $\Omega_{\circ}$ with boundary $\Gamma_{\!\circ }$ such that $\Sigma\subset\Gamma_{\!\circ}$. 
\end{remark}
\begin{remark}
According to the definition of  $\dom(\Theta_{\beta,N,\Sigma})$, one has $\dom(\Theta_{\beta,N,\Sigma})=\dom(\Theta_{N,\Sigma})$; thus Lemma \ref{regularity2} applies and one gets 
$$
\dom(A_{\beta,\delta',\Sigma})\subseteq H^{2}(\RE^{n}\backslash(\overline\Sigma\cup(\partial\Sigma)^{\epsilon}))\,.
$$
\end{remark}
\begin{remark} By the results contained in the proof of Theorem \ref{compr-delta'}, for the domain of the sesquilinear form $f_{\beta,N,\Sigma}$ associated to the self-adjoint operator $\t\Theta_{\beta,N,\Sigma}$ one has the relation $\dom(f_{\alpha,N,\Sigma})\subseteq H^{\frac32}(\Gamma)$. Therefore Theorem \ref{teo-schatten2} applies and so, by Corollary \ref{spectrum},  $\sigma_{ac}(A_{\beta,\delta',\Sigma})=\sigma_{ac}(A)$ and the wave operators $W_{\pm}(A_{\beta,\delta',\Sigma},A)$, $W_{\pm}(A,A_{\beta,\delta',\Sigma})$ exist and are complete. 
\end{remark}
\begin{remark} Let $\varphi\in\dom(\Theta_{\beta,N,\Sigma})\subseteq H^{\frac12}(\Gamma)$. Then $(\hat\gamma_{1}\DL\varphi)|\Sigma\in H^{\frac12}(\Sigma)\subseteq H^{s}(\Sigma)$, $s<0$. 
Then, in case $n=3$ and both $\Gamma$ and $\Sigma$ are smooth, by \eqref{CS2}, $\varphi\in H^{s}(\Gamma)$, $s<1$. Therefore $\dom(\Theta_{\beta, N,\Sigma})\subseteq H^{s}(\Gamma)$, $\frac 12\le s<1$, and so, by Remark \ref{regularity} and \eqref{H4}, 
$$
\D(A_{\beta,\delta',\Sigma})\subseteq H^{\frac32-}(\RE^{3}\backslash\overline\Sigma)\,.
$$
\end{remark}

\end{subsection}


\end{section}
\begin{section}{Appendix. Some remarks on compressions of self-adjoint operators.}
Let $\Theta:\dom(\Theta)\subseteq\fh\to\fh$ be a semibounded self-adjoint operator on the Hilbert space $\fh$ and let $f:\dom(f)\times \dom(f)\subseteq\fh\times\fh\to\RE$ be the corresponding semibounded sesquilinear form. Without loss of generality, eventually by considering $(-\Theta)$ and/or adding a constant, we can suppose that $\Theta$ (and hence $f$) is strictly positive. Let $\fh_{\Theta}$ be the Hilbert space given by $\dom(f)$ endowed with the 
scalar product  $\langle\varphi_{1},\varphi_{2}\rangle_{\fh_\Theta}:=f(\varphi_{1},\varphi_{2})$. Let $\fh_{\Theta}'$ be its dual space and let $\iota:\fh_{\Theta}\to \fh'_{\Theta}$ be the injection defined by $\langle \iota\phi,\varphi\rangle_{\fh'_{\Theta}\fh_{\Theta}}=\langle\phi,\varphi\rangle_{\fh}$, where 
 $\langle\cdot,\cdot\rangle_{\fh'_{\Theta}\fh_{\Theta}}$ denotes the $\fh'_{\Theta}$-$\fh_{\Theta}$ duality and $\langle\cdot,\cdot\rangle_{\fh}$ denotes the scalar product in $\fh$; by the identification $\iota\phi\equiv\phi$, we may regard $\fh_{\Theta}\subseteq\fh\subseteq \fh'_{\Theta}$. Let $\breve\Theta:\fh_{\Theta}
\to \fh_{\Theta}^{\prime}$ be the bounded operator defined by 
$$
\langle\breve\Theta \varphi_{1},\varphi_{2}\rangle_{\fh'_{\Theta}\fh_{\Theta}}=\langle\varphi_{1},\varphi_{2}\rangle_{\fh_\Theta}\,,\qquad \varphi_{1},\varphi_{2}\in \fh_{\Theta}\,.
$$
Obviously $\breve\Theta|\dom(\Theta)=\Theta$; moreover, by \cite[Theorem 2.1]{K}, 
$$
\dom(\Theta)=\mathfrak{d}_{\Theta}:=\{\varphi\in\fh_{\Theta}: \breve\Theta\varphi\in\fh\}\,.
$$
In conclusion one gets a well know characterization of $\Theta$ (see \cite[Remark at page 13]{Far}, \cite[Proof of Theorem VIII.15]{ReSi1}):
\be\label{dtheta}
\Theta=\breve\Theta|\mathfrak{d}_{\Theta}\,.
\ee
Let now $\Pi:\fh\to\fh$ be an orthogonal projector such that $\dom(f)\cap\ran(\Pi)$ is dense in $\ran(\Pi)$. Then the sesquilinear form 
$$f_{\Pi}:\D(f_{\Pi})\times\D(f_{\Pi})\subseteq \ran(\Pi)\times\ran(\Pi)
\to\RE\,, $$
$$
f_{\Pi}(\varphi_{1},\varphi_{2}):=f(\Pi\varphi_{1},\Pi\varphi_{2})=
f(\varphi_{1},\varphi_{2})
\,,\quad \dom(f_{\Pi}):=\dom(f)\cap\ran(\Pi)\,,$$
is densely defined, closed and strictly positive. Hence 
there exists a unique strictly positive self-adjoint operator $\Theta_{\Pi}$ in $\ran(\Pi)$ corresponding to $f_{\Pi}$. By the above reasonings applied to $\Theta_{\Pi}$, we know that $\Theta_{\Pi}=\breve \Theta_{\Pi}|\mathfrak{d}_{\Theta_{\Pi}}$. For any $\varphi_{1}\in\dom(\Theta)\cap\ran(\Pi)$ and for any $\varphi_{2}\in\dom(f)\cap\ran(\Pi)$ one has
$$
\langle\breve\Theta_{\Pi} \varphi_{1},\varphi_{2}\rangle_{\fh_{\Theta_\Pi}'\fh_{\Theta_\Pi}}=\langle\varphi_{1},\varphi_{2}\rangle_{\fh_{\Theta_\Pi}}=f(\Pi\varphi_{1},\Pi\varphi_{1})
=\langle\Theta\Pi\varphi_{1},\Pi\varphi_{2}\rangle_{\fh}=
\langle\Pi\Theta\Pi\varphi_{1},\varphi_{2}\rangle_{\fh}\,.
$$
Thus $\dom(\Theta)\cap\ran(\Pi)\subseteq \mathfrak{d}_{\Theta_{\Pi}}$ and $\text{\rm C}_{\Pi}(\Theta)\subseteq\Theta_{\Pi}$, where $\text{\rm C}_{\Pi}(\Theta)$ is the compression of $\Theta$ to $\ran(\Pi)$ defined by 
$$\text{\rm C}_{\Pi}(\Theta):\dom(\text{\rm C}_{\Pi}(\Theta))\subseteq\ran(\Pi)\to\ran(\Pi)\,,$$
$$\dom(\text{\rm C}_{\Pi}(\Theta)):=\dom(\Theta)\cap\ran(\Pi)\,,\quad \text{\rm C}_{\Pi}(\Theta)\phi:=\Pi\Theta\Pi\phi= \Pi\Theta\phi\,.
$$ 
Notice that $\text{\rm C}_{\Pi}(\Theta)$ is symmetric but it can be not self-adjoint; it is self-adjoint if and only if $\mathfrak{d}_{\Theta_{\Pi}}\subseteq\dom(\Theta)$. 
\par
We now give a more explicit definition of $\breve\Theta_{\Pi}$. Let $\fh_{\Theta_\Pi}$ be the Hilbert space $\dom(f_{\Pi})$ endowed with the scalar product $$\langle\varphi_{1},\varphi_{2}\rangle_{\fh_{\Theta_\Pi}}:=f_{\Pi}(\varphi_{1},\varphi_{2})=f(\varphi_{1},\varphi_{2})=\langle\varphi_{1},\varphi_{2}\rangle_{\fh_\Theta}$$ and let $\fh_{\Theta_\Pi}'$ denote its dual. Since $\fh_{\Theta_\Pi}=\fh_{\Theta}\cap\ran(\Pi)$, by \cite[Proposition 3.5.1]{Aub} one has
$$\fh_{\Theta_\Pi}'=\fh_{\Theta}'/\fk_{\Pi}\,,$$ where 
$$
\fk_{\Pi}:=\{\phi\in \fh_{\Theta}':\forall \varphi\in \fh_{\Theta_\Pi}\,,\ \langle\phi,\varphi\rangle_{\fh_{\Theta}'\fh_{\Theta}}=0\ \}\,,
$$
i.e. 
$$\fh_{\Theta_\Pi}'=\{\, [\phi]\,,\ \phi\in \fh'_{\Theta}\}\,,\quad
[\phi]:=
\{ \psi\in \fh'_{\Theta}: \psi-\phi\in \fk_{\Pi}\}\,.
$$
The $\fh_{\Theta_\Pi}'$-$\fh_{\Theta_\Pi}$ duality is then defined by
$\langle[\phi],\varphi\rangle_{\fh_{\Theta_\Pi}'\fh_{\Theta_\Pi}}:=\langle\phi,\varphi\rangle_{\fh_{\Theta}'\fh_{\Theta}}$. Let $\iota_{\Pi}:\fh_{\Theta_\Pi}\to \fh_{\Theta_\Pi}'$ be the injection defined by $\langle\iota_{\Pi}\phi,\varphi\rangle_{\fh_{\Theta_\Pi}'\fh_{\Theta_\Pi}}:=
\langle\phi,\varphi\rangle_{\fh_{\Theta_\Pi}}$. Since
$$
\langle\phi,\varphi\rangle_{\fh_{\Theta_\Pi}}=\langle\phi,\varphi\rangle_{\fh_{\Theta}}=
\langle\iota\phi,\varphi\rangle_{\fh'_{\Theta}\fh_{\Theta}}=
\langle[\iota\phi],\varphi\rangle_{\fh_{\Theta_\Pi}'\fh_{\Theta_\Pi}}\,,
$$
one gets $\iota_{\Pi}\varphi=[\iota\varphi]$.  By the identification $\iota_{\Pi}\varphi\equiv\varphi$, we may regard $\fh_{\Theta_\Pi}\subseteq\ran(\Pi)\subseteq \fh_{\Theta_\Pi}'$. For any $\varphi_{1},\varphi_{2}\in \fh_{\Theta_\Pi}$, the bounded operator $\breve\Theta _{\Pi}:\fh_{\Theta_\Pi}\to \fh_{\Theta_\Pi}'$ satisfies the relations
$$
\langle\breve\Theta _{\Pi}\varphi_{1},\varphi_{2}\rangle_{\fh_{\Theta_\Pi}',\fh_{\Theta_\Pi}}=f_{\Pi}(\varphi_{1},\varphi_{2})=f(\varphi_{1},\varphi_{2})=\langle\breve\Theta \varphi_{1},\varphi_{2}\rangle_{\fh'_{\Theta}\fh_{\Theta}}=
\langle[\breve\Theta\varphi_{1}],\varphi_{2}\rangle_{\fh_{\Theta_\Pi}'\fh_{\Theta_\Pi}}\,.
$$ 
Thus $\breve\Theta _{\Pi}\varphi=[\breve\Theta\varphi]$ and
\begin{align*}
\mathfrak{d}_{\Theta_{\Pi}}=\{\varphi\in \fh_{\Theta_\Pi}:
[\breve\Theta\varphi]\in \ran(\Pi)\}=
\{\varphi\in \fh_{\Theta_\Pi}:\exists\t\varphi\in\ran(\Pi)\ \text{s.t.}\, \breve\Theta\varphi-
\t\varphi\in \fk_\Pi\}\,.
\end{align*}
In conclusion we have the following
\begin{lemma}\label{lemma-compr} Let $f:\dom(f)\times \dom(f)\subseteq\fh\times\fh\to\RE$ be the  closed sesquilinear form corresponding to the semibounded self-adjoint operator $\Theta$; let $\Pi:\fh\to\fh$ be an orthogonal projector such that $\dom(f)\cap\ran(\Pi)$ is  dense in $\ran(\Pi)$. Then the self-adjoint operator 
$\Theta_{\Pi}:\dom(\Theta_{\Pi})\subseteq\ran(\Pi)\to\ran(\Pi)$
associated to the closed semibounded   sesquilinear form $f_{\Pi}$ defined as the restriction of $f$ to $\dom(f)\cap\ran(\Pi)$, is given by 
$$
\dom(\Theta_{\Pi}):=\{\varphi\in \dom(f)\cap\ran(\Pi):\exists\t\varphi\in\ran(\Pi)\ \text{\rm s.t.}\, \breve\Theta\varphi-\t\varphi\in \fk_\Pi\}\,,
\qquad
\Theta_{\Pi}\varphi:=\t\varphi\,.
$$
Moreover the compression $C_{\Pi}(\Theta)$ is self-adjoint, equivalenty $C_{\Pi}(\Theta)=\Theta_{\Pi}$, if and only if  $\dom(\Theta_{\Pi})\subseteq\dom(\Theta)$.
\end{lemma}

\end{section}

\section*{References}
\addcontentsline{toc}{section}{References}%
\markboth{References}{References}%

\end{document}